\documentclass[11pt,a4paper]{amsart}
\usepackage{amsfonts,amsmath,amssymb}
\usepackage{hyperref}
\usepackage{latexsym,fullpage,amsfonts,amssymb,amsmath,amscd,graphics,epic}
\usepackage[all]{xy}
\usepackage{amssymb,amsbsy,amsthm,amsxtra}
\usepackage[usenames]{color}
\usepackage{amscd}
\usepackage{amsthm}
\usepackage{amsfonts}
\usepackage{amssymb}
\usepackage{mathrsfs}
\usepackage{url}
\usepackage{bbm}
\usepackage{wasysym}
\usepackage{enumitem}
\usepackage{framed}

\usepackage{xcolor}

\usepackage[vcentermath]{youngtab}%
\theoremstyle{plain}
\newtheorem*{theorem*}{Theorem}
\newtheorem*{remark*}{Remark}
\newtheorem*{example*}{Example}
\newtheorem{lemma}{Lemma}[subsection]
\newtheorem{proposition}[lemma]{Proposition}
\newtheorem{corollary}[lemma]{Corollary}
\newtheorem{theorem}[lemma]{Theorem}

\newtheorem*{conjecture*}{Conjecture}

\newtheorem{sublemma}[lemma]{Sublemma}

\newtheorem{introtheorem}{Theorem}
\newtheorem{introcorollary}[introtheorem]{Corollary}

\theoremstyle{definition}
\newtheorem{definition}[lemma]{Definition}

\newtheorem{example}[lemma]{Example}

\theoremstyle{remark}
\newtheorem{remark}[lemma]{Remark}

\newtheorem{notation}[lemma]{Notation}

\newcommand{\Hom}{\operatorname{Hom}}

\newcommand{\triv}{{\mathbbm 1}}

\newcommand{\Ind}{\operatorname{Ind}}
\newcommand{\Coind}{\operatorname{Coind}}
\newcommand{\id}{\operatorname{Id}}

\renewcommand{\Im}{\operatorname{Im}}
\newcommand{\Ker}{\operatorname{Ker}}
\newcommand{\Ext}{\operatorname{Ext}}

\newcommand{\Aut}{{\operatorname{Aut}}}

\newcommand{\End}{\operatorname{End}}

\newcommand{\C}{{\mathbb C}}
\newcommand{\Z}{{\mathbb Z}}

\newcommand{\g}{{\mathfrak g}}

\newcommand{\eps}{{\varepsilon}}

\newcommand{\lam}{{\lambda}}

\newcommand{\T}{{\mathcal{T}}}

\newcommand{\gl}{{\mathfrak{gl}}}

\newcommand{\ad}{\operatorname{ad}}

\newcommand{\abs}[1]{\left|{#1}\right|}

\newcommand{\p}{\mathfrak{p}}

\renewcommand{\O}{\mathcal{O}}

\newcommand{\bV}{\mathbf{V}}
\newcommand{\bT}{\mathbf{T}}
\newcommand{\ex}{\mathit{ex}} 
\newcommand{\faith}{\mathit{faith}} 
\newcommand{\Fun}{\mathrm{Fun}}

\newcommand{\Mod}{\operatorname{Mod}}
\newcommand{\Vect}{\mathtt{Vect}}
\newcommand{\sVect}{\mathtt{sVect}}
\newcommand{\PP}{{\mathfrak P}}
\newcommand{\Rep}{\mathrm{Rep}}
\newcommand{\urep}{\Rep(\underline{P})}
\newcommand{\urepk}[1]{\Rep^{#1}(\underline{P})}
\newcommand{\s}{\mathtt{s}}
\newcommand{\bdel}{\pmb{\Delta}}
\newcommand{\bnab}{\pmb{\nabla}}
\renewcommand{\L}{\pmb{L}}
\newcommand{\Alg}{\mathrm{Alg}}
\newcommand{\Grps}{\mathrm{Grps}}

\newcommand{\InnaA}[1]{{#1}}
\newcommand{\InnaB}[1]{{#1}}
\newcommand{\InnaC}[1]{{#1}}
\newcommand{\InnaD}[1]{{#1}}
\newcommand{\VeraC}[1]{{#1}}
\newcommand{\comment}[1]

\def\quotient#1#2{%
    \raise1ex\hbox{$#1$}\Big/\lower1ex\hbox{$#2$}%
}

\begin{document}

\date{\today}
\title{Deligne categories and the periplectic Lie superalgebra}
 \author{Inna Entova-Aizenbud, Vera Serganova}
\address{Inna Entova-Aizenbud, Dept. of Mathematics, Ben Gurion University,
Beer-Sheva,
Israel.}
\email{entova@bgu.ac.il}
\address{Vera Serganova, Dept. of Mathematics, University of California at
Berkeley,
Berkeley, CA 94720.}
\email{serganov@math.berkeley.edu}
\begin{abstract}
We study stabilization of finite-dimensional representations of the periplectic
Lie superalgebras $\p(n)$ as $n \to \infty$.

The paper gives a construction of the tensor category $\urep$, possessing nice
universal properties among tensor categories over the category $\sVect$ of
finite-dimensional complex vector superspaces.

First, it is the ``abelian envelope'' of the Deligne category corresponding to
the periplectic Lie superalgebra, in the sense of \cite{EHS}.

Secondly, given a tensor category $\mathcal{C}$ over $\sVect$, exact tensor
functors $\urep\rightarrow \mathcal{C}$ classify pairs $(X, \omega)$ in
$\mathcal{C}$ where $\omega: X \otimes X \to \Pi \triv$ is a non-degenerate
symmetric form and $X$ not annihilated by any Schur functor.

The category $\urep$ is constructed in two ways. The first construction is
through an explicit limit of the tensor categories $\Rep(\p(n))$ ($n\geq 1$)
under Duflo-Serganova functors. The second construction (inspired by P.
Etingof) \InnaB{describes $\urep$} as the category of representations of a periplectic Lie supergroup in the Deligne category $\sVect \boxtimes \Rep(\underline{GL}_t)$.

An upcoming paper will give results on the abelian and tensor structure of
$\urep$.
\end{abstract}

\maketitle
\setcounter{tocdepth}{1}
\tableofcontents

\section{Introduction}
\subsection{}
The (complex) periplectic Lie superalgebra $\p(V)$ is the Lie superalgebra of
endomorphisms of a complex vector superspace $V$ possessing a non-degenerate
symmetric form $\omega: V \otimes V \to \Pi \C \cong \C^{0|1}$, where $\Pi$
denotes the \InnaB{parity} shift on vector superspaces (this is also referred to as an
``odd form''). An example of such superalgebra is $\p(n)= \p(\C^{n|n})$, where
$\omega_n: \C^{n|n} \otimes \C^{n|n} \to \C^{0|1}$ pairing the even and odd
parts of the vector superspace $\C^{n|n}$.

The periplectic Lie superalgebras
possess an interesting non-semisimple representation theory; some results
on the category $\Rep(\p(n))$ of finite-dimensional integrable representations
of $\p(n)$ can be found in \cite{BDE+, Chen, DLZ, Gor, Moon, Ser}.

In this paper, we propose studying the stabilization properties of
finite-dimensional
integrable representations of $\p(n)$ as $n$ grows large by considering a
certain ``limit category'' $\urep$. This category is also important in its own
right, possessing nice universal properties.
\subsection{Deligne categories}

The inspiration for the construction of $\urep$ comes from the theory of Deligne categories. In \cite{DM}, Deligne and Milne constructed a family of categories\footnote{Denoted also by $\Rep(GL_t), \mathcal{D}_t$.} $Tilt(\underline{GL}_t)$, parameterized by $t \in \mathbb{C}$.

These are Karoubian additive rigid SM categories. Among such categories, $Tilt(\underline{GL}_t)$ is the universal category generated by a dualizable object $V$ of categorical dimension $t$, and can be considered an interpolation of the family of categories $\Rep(GL_n(\C))$, $n\geq 0$.
The categories $Tilt(\underline{GL}_t)$ are called Deligne categories for the general linear group.

For $t$ non-integer, these are semisimple tensor categories; for $t\in\Z$, the category $Tilt(\underline{GL}_t)$ is Karoubian but not abelian. An abelian envelope $\Rep(\underline{GL}_t)$ of $Tilt(\underline{GL}_t)$, possessing nice universal properties, was constructed in \cite{EHS}.

An important feature of the category $\Rep(\underline{GL}_t)$ (and of $Tilt(\underline{GL}_t)$, for $t \notin \Z$) is that it an example of a non-Tannakian tensor category: namely, it is not equivalent to the category of representations of any affine algebraic group, \InnaB{nor} supergroup. More examples of non-Tannakian categories of this form were constructed in \cite{D, Et}.

\subsection{Periplectic Deligne category}
In the periplectic case, a Karoubian additive Deligne category $\PP$ was constructed in \cite{CoulEhr,KT}. This is a Karoubian additive rigid SM category, and is a module category over $\sVect$ (hence endowed with an endofunctor $\Pi$ such that $\Pi^2 \cong \id$).

Among such categories, it is the universal category generated by a dualizable object $\tilde{V}$ and a non-degenerate symmetric form $\omega_{\tilde{V}}: \tilde{V} \otimes \tilde{V} \to \Pi \triv$.

The Karoubian structure of the category $\PP$ was studied in \cite{CoulEhr}, and its tensor ideals were classified in \cite{Coul}.

\begin{remark}
 The endomorphism algebras $$\End^{\bullet}_{\PP}(\tilde{V}^{\otimes k}):= \End^{\bullet}_{\PP}(\tilde{V}^{\otimes k}) \oplus \Hom_{\PP}(\tilde{V}^{\otimes k}, \Pi \tilde{V}^{\otimes k})$$ are known as the periplectic (or marked) Brauer algebras, and have been studied in \cite{BDE+2, ChenPeng, CoulBrauer, CoulEhr2, KT, Tharp}.
\end{remark}

\begin{remark}
 Unlike the general linear case, here we do not have a parameter $t$: the categorical dimension of $\tilde{V}$ is automatically zero, since $\tilde{V} \cong \Pi \tilde{V}^*$ and hence $\dim \tilde{V} = -\dim \tilde{V}$.
\end{remark}

\subsection{Construction of the category $\urep$}
Let $\sVect$ denote the symmetric monoidal
category of finite-dimensional vector superspaces with \InnaB{parity}-preserving maps.

In what follows, we work in the $2$-category of module categories over
$\sVect$. The bottom line of this is that all the categories we consider are
endowed with an autoequivalence called the \InnaB{parity} shift and denoted by $\Pi$,
and the functors we consider respect intertwine with the \InnaB{parity} shifts. In
particular, any symmetric monoidal structure on the categories considered below
respects the $\sVect$-module structure.

We construct a tensor category $\urep$ as a certain limit of categories $\Rep(\p(n))$ with respect to the Duflo-Serganova functors $DS_x: \Rep(\p(n)) \to \Rep(\p(n-2))$. Such functors are defined as follows (see \cite{DS}).

Given a Lie superalgebra $\g$, let $x\in\g$ be an odd element satisfying $[x,x]=0$. Let $\g_x:=\Ker\ad_x/\Im \ad_x$. This is again a Lie superalgebra, and we can define
a functor $$DS_x: \Rep(\g) \longrightarrow \Rep(\g_x), \;\; M \longmapsto M_x:=\quotient{\Ker x}{\Im x}.$$ It is easy to see that this functor is symmetric monoidal. In our case, taking $\g = \p(n)$ and $x$ of rank $2$, we obtain $\g_x \cong \p(n-2)$.

The functors $DS_x$ are in general not exact, but turn out to be exact on certain subcategories of $\Rep(\p(n))$, allowing us to consider a limit of $\Rep(\p(n))$ over $n$.

As in \cite{EHS}, the obtained category $\urep$ is not Tannakian: that is, it is not equivalent to the category of representations of any supergroup. Furthermore, it is an interesting example of a tensor $\sVect$-module category which is not split, in the sense that we cannot present $\urep$ as $\sVect \boxtimes_{\Vect} \T$ for any tensor category $\T$.

The category $\urep$ is a lower highest weight category (see Definition \ref{def:lower_hw_cat}): namely, we have a filtration $\urep = \bigcup_{k\geq 0} \urepk{k}$ by full highest weight subcategories,  whose standard,
costandard and tilting objects play the same role in each
$\urepk{k'}$ for $k' \geq k$. The full subcategory of tilting objects in $\urep$ is then precisely $\PP$. The highest weight structure of $\urep$ is described in Section \ref{sec:properties}, and will be further investigated in an upcoming paper.

\subsection{Universal property}
Let us now describe the universal properties of $\urep$ proved in this paper.

Let $\T$ be \InnaD{a tensor category (more generally, an abelian symmetric monoidal category with
biexact bilinear bifunctor
$\otimes$ and a simple unit object $\triv$)}.

Our first theorem establishes that $\urep$ is universal among
\InnaD{tensor categories generated by an} object with a
non-degenerate symmetric odd form (see Theorems \ref{thrm:uni-1} and
\ref{thrm:uni-2}):
\begin{introtheorem}\label{introthrm:uni-1}
Let $X \in \T$ be a dualizable
object with
a
non-degenerate symmetric form $$\omega_X:
X
\otimes X \to
\Pi \triv.$$ Assume $X$ is not annihilated by any Schur
functor. There exists an essentially unique exact SM functor $\urep
\longrightarrow \T$ carrying $\bV$ to $X$ and $\omega_{\bV}$ to
$\omega_X$.

\end{introtheorem}
This result is proved using the explicit construction of $\urep$ given in
Section \ref{sec:Rep_p} and the techniques developed in \cite{EHS}.

Furthermore, we show that the collection of categories $\langle
\urep, \Rep(\p(n))\rangle_{n\geq 1}$ forms the ``abelian
envelope'' of the Karoubian Deligne category $\PP$, in the sense of
\cite{D, EHS}:
\begin{introtheorem}\label{introthrm:uni-2}
 Let $X \in \T$  be a dualizable object with
a
non-degenerate symmetric form $$\omega_X:
X
\otimes X \to
\Pi \triv.$$ Consider the canonical SM functor
$F_X:
\PP \longrightarrow \T$ sending the generator
$\tilde{V}$ of $\PP$ to $X$ and the form $\omega_{\tilde{V}}$ on $\tilde{V}$ to
$\omega_X$.
\begin{enumerate}
 \item If $X$ is not annihilated by any Schur functor then $F_X$
factors through the embedding $I:\PP \to\urep$ and gives
rise
to a faithful exact symmetric
monoidal functor $$\mathbf{F}_X:\urep \longrightarrow \T,
\;\;\;
\bV \longmapsto X, \omega_{\bV} \to \omega_X. $$

 \item If $X$ is annihilated by some Schur functor then there exists a unique
$n
\in \Z_+$ such that $F_X$ factors through the symmetric monoidal
functor
$\PP \longrightarrow \Rep(\p(n))$ and gives rise to a faithful exact symmetric
monoidal
functor
 $$ \mathbf{F}_X:\Rep(\p(n)) \longrightarrow \T, \;\;\;\C^{n|n} \longmapsto
X$$ with the canonical form on $\C^{n|n}$ sent to $\omega_X$.
\end{enumerate}
\end{introtheorem}
This result is stated in stronger form in Theorem \ref{thrm:uni-2}, and is
proved in Section \ref{sec:uni-2} using previous results on extension of
functors from $\PP$ to $\urep$ (see Theorem \ref{thrm:uni-1}) and Tannakian
formalism in tensor categories, discussed in Section \ref{sec:uni-2} and
Appendix \ref{app:affine_grp_schemes}.

As a corollary, we obtain an alternative construction of $\urep$, inspired by
\cite{Et}. It is much shorter and more compact than the first
construction of $\urep$, although not as explicit.

For any $t \in \C$, consider the tensor Deligne category
$\Rep(\underline{GL}_t)$ constructed in \cite{EHS}\footnote{In \cite{EHS}, this
category is denoted by $\mathcal{V}_t$; \InnaD{here we consider a ``doubled'' version $\Rep(\underline{GL}_t) = \sVect \otimes \mathcal{V}_t$ which is a module category over $\sVect$}.}. We denote its $t$-dimensional
generator by $\left(X_t\right)_{\bar 0}$.

Let $\Rep(\underline{GL}_t)$ be the corresponding module category
over $\sVect$, and consider the object $$X_t := \left(X_t\right)_{\bar 0} \oplus \Pi
\left(X_t\right)^*_{\bar 0}$$ in $ \Rep(\underline{GL}_t)$.

The object $X_t$ is not annihilated by any Schur functor, and comes equipped
with a canonical non-degenerate symmetric form $\omega_{X_t}: X_t \otimes X_t
\to \Pi \triv$.

By Theorem
\ref{introthrm:uni-2}, we obtain an exact symmetric monoidal functor
$$\mathbf{F}:\urep \longrightarrow \Rep(\underline{GL}_t),
\;\;\;
\bV \longmapsto X_t, \;\;\omega_{\bV} \to \omega_{X_t}. $$

\InnaC{Let $P(X_t)$ be the group scheme in $\Rep(\underline{GL}_t)$ preserving the form $\omega_{X_t}$: namely, we consider the universal commutative Hopf algebra object $\O(P(X_t)) \subset Sym(X_t \oplus \Pi X_t)$ in $Ind-\Rep(\underline{GL}_t)$ whose action on $X_t$ preserves $\omega_{X_t}$. We then have: }
\begin{introcorollary}
The functor $\mathbf{F}$ induces an equivalence of tensor categories $$\urep
\to \Rep(P(X_t)).$$
\end{introcorollary}
This result is proved in Section \ref{ssec:second_constr_Rep_p}.

\begin{remark}
 \InnaC{Let $ \underline{\End}(X_t) \cong X_t \otimes X_t^*$ be internal endomorphism
algebra of $X_t$, seen as a Lie algebra object in $\s
\Rep(\underline{GL}_t)$. Let $\p(X_t) \subset \underline{\End}(X_t)$ be
Lie subalgebra preserving $\omega_{X_t}$. Then $\p(X_t) = Lie P(X_t)$, and $\gl\left(\left(X_t\right)_{\bar 0} \right) \subset \p(X_t)$. Unlike in the classical representation theory of Lie superalgebras, we do not expect that any representation of $\p(X_t)$ whose restriction to $\gl\left(\left(X_t\right)_{\bar 0} \right)$ is integrable (that is, given by objects in $\Rep(\underline{GL}_t)$) integrates to a representation of the group scheme $P(X_t)$. See Remark \ref{rmk:integrable_repr}.}
\end{remark}

\subsection{Structure of the paper}
In Section \ref{sec:inf_p_n}, we present some results about the representation theory of the Lie superalgebra $\p(\infty)$ used in the construction of $\urep$.

In Section \ref{sec:Rep_k}, we define full highest weight subcategories $\Rep^k(\p(n)) \subset \Rep(\p(n))$, which will be the building blocks for $\urep$. In Section \ref{sec:DS_functor}, we define the $DS$ functor, and prove some stabilization results for the subcategories $\Rep^k(\p(n))$ under $DS$. In Section \ref{sec:Rep_p}, we construct the category $\urep$ itself, and present some of its basic properties.

Sections \ref{sec:universality} and \ref{sec:uni-2} are devoted to the proofs of Theorems \ref{introthrm:uni-1} and \ref{introthrm:uni-2} above.

Section \ref{sec:uni-2} uses heavily the terminology of affine algebraic groups in pre-Tannakian tensor categories; a short summary of the necessary definitions and results in this subject is given in Appendix \ref{app:affine_grp_schemes}.

Section \ref{sec:properties} describes the lower highest weight structure on $\urep$.

Finally, in Appendix \ref{appendix_A}, we give a combinatorial proof of a result of \cite{DLZ} describing the a spanning set of the superspace $\Hom^{\bullet}_{\p(n)}(V_n^{\otimes k}, \triv)$.
\subsection{Acknowledgements}
This project was conceived during the first authors visit to UC Berkeley; a large part of the work was done when both authors were in residence at the Mathematical Sciences Research Institute (MSRI) in Berkeley, California, during the Spring 2018 semester. The reseach in MSRI was supported by the National Science Foundation under Grant No. DMS-1440140. We would like to thank both institutions for providing the support and the work environment for this project. We would also like to thank P. Etingof for his comments and suggestions, K. Coulembier for the very helpful comments on the first drafts of the paper, and S. Sahi for useful comments and suggestions. I.E.-A. was supported by the Israel Science Foundation (ISF grant no. 711/18). V.S. was supported by NSF grant 1701532.

\section{Preliminaries and notation}\label{sec:notn}
\subsection{General}

Throughout this paper, we will work over the base field $\mathbb C$, and all
the categories considered will be $\C$-linear.

A {\it
vector superspace} will be defined as a $\mathbb{Z}/2\mathbb{Z}$-graded vector
space $V=V_{\bar 0}\oplus V_{\bar 1}$. The {\it parity} of a homogeneous vector
$v \in V$ will be denoted by $p(v) \in \mathbb{Z}/2\mathbb{Z}=\{\bar 0, \bar
1\}$ (whenever the notation $p(v)$ appears in formulas, we always assume that
$v$ is homogeneous).

\subsection{Tensor categories}\label{ssec:notn:SM}
In the context of symmetric monoidal (SM) categories, we will denote by $\triv$
the unit object, and by $\sigma$ the symmetry morphisms.

A functor between symmetric monoidal categories will be called a {\it SM
functor}
if it respects the SM structure in the sense of \cite[2.7]{D1}
(there it is called "foncteur ACU"); similarly, an $\otimes$-natural
transformation between SM functors is a transformation respecting the
monoidal structure in the sense of \cite[2.7]{D1}.

\begin{remark}
Any natural transformation between SM functors on rigid SM
categories is an
isomorphism (see \cite{DM}).
\end{remark}

In this paper a {\it tensor category} is a rigid SM abelian
$\mathbb{C}$-linear category, where the bifunctor $\otimes$ is bilinear on
morphisms, and $\End(\triv) \cong \mathbb{C}$. An explicit definition can be
found in \cite{D1, EGNO}.
Note that in such a category the bifunctor $\otimes$ is biexact.

A {\it tensor functor} between two tensor categories is an exact SM functor.
\begin{remark}
Any tensor functor is automatically faithful (see \cite{DM}).
\end{remark}

As in \cite[2.12]{D1}, a {\it pre-Tannakian category} is a tensor category
satisfying
finiteness conditions: namely, every object has finite length and every
$\Hom$-space is finite-dimensional over $\mathbb{C}$.

Given an object $V$ in a SM category, we will denote by $$coev: \triv \to V
\otimes V^*, \;\; ev: V^* \otimes V \to \triv$$ the coevaluation and evaluation
morphisms for $V$. We will also denote by $\gl(V):= V \otimes V^*$ the internal
endomorphism space with the obvious Lie algebra structure on it. The object $V$
is then a module over the Lie algebra $\gl(V)$; we denote the action by $act:
\gl(V) \otimes V \to V$. \InnaD{The notation $\Rep(\gl(V))$ will stand for the category of all\footnote{Caveat: this notation differs from \cite{EHS}.} integrable finite-dimensional super-representations of the Lie superalgebra $\gl(V)$, even if this algebra is purely even.}

\subsection{Super structure on SM categories}\label{ssec:notn:super}
Let $\sVect$ be the SM
category of supervectorspaces and even morphisms.
The categories appearing in this paper will be module categories over the SM
category
$\sVect$: this essentially means that they will be equipped with a $\C$-linear
endofunctor $\Pi$ (exact if the category is abelian) and an isomorphism $\Pi^2
\to \id$.

Our categories will be enriched over $\Vect$ but not over $\sVect$ (for
instance,
we
will only consider ``even morphisms'' in the categories of representation of
superalgebras). On the few occasions when we would like to
consider $\Z/2\Z$-graded $\Hom$-spaces, we will denote these by
$$\mathtt{s}\Hom(X, Y) := \Hom(X, Y) \oplus \Hom(X, \Pi Y).$$

Moreover, any monoidal structure considered on such categories will be
compatible with
the
module structure over $\sVect$, making them analogues of algebras over
$\mathbb{Z}/2\mathbb{Z}$. Such categories will be called {\it
$\sVect$-categories} for short. All functors between $\sVect$-categories
will
respect the $\sVect$-module structure, and will be called {\it
$\sVect$-functors}; these are functors commuting with
the \InnaB{parity} shift
functors $\Pi$. When all the categories involved are tensor
$\sVect$-categories, we will drop the ``$\sVect$-'' and just use the term
``tensor functor'' for short.

{Given any \InnaD{additive linear} category $\T$, we will denote by $\s \T$ the corresponding tensor $\sVect$-category $$\s \T:=\sVect \boxtimes_{\Vect} \T.$$ \InnaB{This can be viewed as a special case of the Deligne tensor product of abelian categories, described in \cite{D1}.}}

\subsection{The periplectic Lie superalgebra}\label{ssec:notn:periplectic}
\subsubsection{Definition of periplectic Lie superalgebra}
Let $n \in \Z_{>0}$, and let $V_n$ be an $(n|n)$-dimensional vector superspace
equipped with a non-degenerate odd symmetric form
\begin{eqnarray}
\label{beta}
\beta:V_n\otimes V_n\to\mathbb C,\quad \beta(v,w)=\beta(w,v),
\quad\text{and}\quad
\beta(v,w)=0\,\,\text{if}\,\,p(v)=p(w).
\end{eqnarray}

Then $\operatorname{End}_{\mathbb C}(V_n)$ inherits the structure of a vector
superspace from $V_n$. We denote by $\mathfrak{p}(n)$ the Lie superalgebra of
all
$X\in\operatorname{End}_{\mathbb C}(V_n)$ preserving $\beta$, i.e. satisfying
$$\beta(Xv,w)+(-1)^{p({X})p(v)}\beta(v,Xw)=0.$$

\InnaC{Let $V_{\bar{0}, n}$ be the even part of $V_n$, and let $V_{\bar{1}, n}$ be the odd part of $V_n$:
$$ V_n = V_{\bar{0}, n} \oplus V_{\bar{1}, n}.$$ The form $\beta$ induces a non-degenerate pairing of the underlying vector spaces $ V_{\bar{0}, n}$ and $V_{\bar{1}, n}$.}
\begin{remark}\label{rmk:basis}
Choosing dual bases $v_1, v_2, \ldots, v_n$ in $V_{\bar{0}, n}$ and $v_{1'},
v_{2'},\ldots v_{n'}$ in $V_{\bar{1}, n}$, we can write the matrix of $X\in
\mathfrak{p}(n)$ as
$\left(\begin{smallmatrix}A&B\\C&-A^t\end{smallmatrix}\right)$
where $A,B,C$ are $n\times n$ matrices such that $B^t=B,\, C^t=-C$.
\end{remark}

We will occasionally denote ${\g}=\p(n)$, with triangular decomposition
${\g} \cong
{\g}_{-1} \oplus {\g}_0 \oplus {\g}_1$ where $${\g}_0 \cong \gl(n), \;\;
{\g}_{-1}
\cong \VeraC{\Pi}\InnaC{S^2 V_{\bar{1}, n}}, \;\; {\g}_1 \cong \VeraC{\Pi}\InnaC{\wedge^2 V_{\bar{1}, n}}.$$
Then the action of $\g_{\pm 1}$ on any $\g$-module is ${\g}_0$-equivariant.

\subsubsection{Weights for the periplectic
superalgebra}\label{sssec:notn:weight_p_n}
The integral weight lattice for $\mathfrak{p}(n)$ will be $span_{\mathbb
Z}\{\eps_i\}_{i=1}^n$.

\begin{itemize}[label={$\star$}]
 \item We fix a set of simple roots $-\eps_1-\eps_2,\eps_2 - \eps_{1},
\ldots, \eps_n-\eps_{n-1}$, the first root is odd and all others are even.

Hence the dominant integral weights will be given by $\lambda = \sum_i
\lambda_i \eps_i$, $\lambda_1 \leq \lambda_2 \leq \ldots\leq \lambda_n$.
\item We fix \InnaB{an} order on the weights of $\p(n)$: for weights $\mu, \lambda$, we
say that $\mu \geq \lambda$ if $\mu_i \leq \lambda_i$ for each $i$.

\begin{remark}
 It was shown in \cite[Section 3.3]{BDE+} that if $\leq$ corresponds to a
highest weight structure on the category of finite-dimensional representations
of $\p(n)$.
\end{remark}

\item
The simple finite-dimensional representation of $\p(n)$ \VeraC{with dominant integral
highest weight $\lambda$} whose highest weight vector is {\it even} will be denoted by
$L_n(\lambda)$.
\begin{example}
 The natural representation $V_n$ of $\p(n)$ has highest weight $-\eps_1$,
with odd highest weight vector; hence $V_n \cong \Pi L_n(-\eps_1)$. The
representation $\wedge^2 V_n$ has highest weight $-2\eps_1$, and the
representation $S^2 V_n$ has highest weight $-\eps_1 - \eps_2$ \VeraC{for $n\geq 2$}; both have even
highest weight vectors, and $$\InnaC{\wedge^2 V_n \twoheadrightarrow L_n(-2\eps_1), \;\; L_n(-\eps_1 - \eps_2) \hookrightarrow S^2 V_n .}$$
\end{example}
\item Set $\rho^{(n)} =
\sum_{i=1}^n (i-1)\eps_i$, and for any weight
$\lambda$, denote $$\bar{\lambda}  = \lambda+\rho^{(n)} .$$

\item We will associate to $\lambda$ a weight diagram $d_{\lambda}$, defined as
a labeling of the integer line by symbols $\bullet$ (``black ball'') and
$\circ$
(``empty'') such that such that $j$ has label $\bullet$ if $j \in
\{\bar{\lambda_i}\,|\,i=1, 2, \ldots \}$, and label
$\circ$ otherwise.

\item We denote $\abs{\lambda} := - \sum_i {\lambda_i}$.

\item When $\lambda \geq 0$ and $\abs{\lambda}\leq k$, we say that
$\lambda$ is {\it $k$-admissible}. For such weights, we have:
$$\lam_1 \leq \lam_2 \leq \ldots \leq \lam_k \leq \lam_{k+1} = \lam_{k+2} =
\ldots = 0.$$

\item Let $\lambda \geq 0$. We denote by $\pmb{\lam}$ be the Young diagram
corresponding to $\lam$ (so the
columns of $\pmb{\lam}$ have lengths $-\lam_1, -\lam_2$, etc); then the
number of boxes in $\pmb{\lam}$ is $\abs{\lam}$. By abuse of notation,
$\pmb{\lam}$
will also stand for the corresponding irreducible representation of the
symmetric group $S_k$. We will also denote by $\pmb{\lam}^{\vee}$ the
transpose of the Young diagram $\pmb{\lam}$, and by $\lambda^{\vee}$ the
corresponding weight (hence
$$ \lambda^{\vee}_1 = -\sharp\{i: \lambda_i \leq -1\}, \; \lambda^{\vee}_2 =
-\sharp\{i: \lambda_i \leq -2\}$$ and so forth). \InnaC{In terms of diagrams, $d_{\lambda^{\vee}}$ is obtained from $d_{\lambda}$ by reflecting the diagram with respect to the point $-1/2$, and then reversing the colors in the diagram.}
\end{itemize}

\subsubsection{Representations of
\texorpdfstring{$\p(n)$}{p(n)}}\label{sssec:notn:repr_periplectic}

We denote by $\Rep(\p(n))$ the category of
finite-dimensional representations of
$\p(V)$ whose restriction to $\p(V)_{\bar 0} \cong
\mathfrak{gl}(n)$ integrates to an action of $GL(n)$.

By definition, the morphisms in $\Rep(\p(n))$ will be {\it
grading-preserving}
$\p(V)$-morphisms,
i.e., $\operatorname{Hom}_{\Rep(\p(n))}(X,Y)$ is a vector space and not a
vector superspace. This is important in order to ensure that the category
$\Rep(\p(n))$ be abelian.

The category $\Rep(\p(n))$ is not semisimple. In fact, this
category is a highest weight category,
having simple, standard, costandard, and projective modules (these are also
injective and tilting, per \cite{BKN}). Given a simple module $L_n(\lambda)$ in
$\Rep(\p(n))$, we denote the corresponding standard, costandard, and projective
modules by $\Delta_n(\lambda)$, $\nabla_n(\lambda)$, $P_n(\lambda)$
respectively.

This paper heavily relies on the results in \cite{BDE+}. In
particular, we use the following two results throughout the paper:

\begin{theorem}[See \cite{BDE+}.]\label{old_thrm:transl_func}
There exists a natural $\Z$-grading on the functor $-\otimes V
\cong \bigoplus_{j
\in \Z}
\Theta_j$, given by generalized eigenspaces of the tensor Casimir.
The relations on the translation functors $\Theta_j$, $j \in \Z$ induce a
representation of the infinite Temperley-Lieb algebra $TL_{\infty}(q=i)$ on the
Grothendieck ring on $\Rep(\p(n))$.
\end{theorem}
\begin{theorem}[See \cite{BDE+}.]\label{old_thrm:transl_proj}

Indecomposable projectives in $\Rep(\p(n))$ satisfy the following properties:

\begin{enumerate}
 \item Indecomposable projectives $P_n(\lambda)$ are
multiplicity-free for any $\lambda$.
\item For any $\lambda$ and any $i$, $\Theta_i P_n(\lambda)$ is indecomposable
projective or zero.
\end{enumerate}

\end{theorem}
For more details on the structure of $\Rep(\mathfrak{p}(n))$ we refer the reader
to
\cite{BDE+}.

\subsection{The Deligne
category \texorpdfstring{category
$\Rep(\underline{GL}_t)$}{for general linear group}}\label{ssec:notn:GL_t}

The Deligne categories $\Rep(\underline{GL}_t)$ ($t\in \C$) are a family of
tensor categories possessing a number of remarkable
properties.

\InnaD{Such categories were constructed in \cite{DM} for
$t\notin \Z$, and in \cite{EHS} for $t \in \Z$. We will denote the category constructed in \cite{DM, EHS} by $\Rep_{halved}(\underline{GL}_t)$, and by $\Rep(\underline{GL}_t)$ the ``doubled'' Deligne category $$\Rep(\underline{GL}_t) := \sVect \otimes_{\Vect} \Rep_{halved}(\underline{GL}_t).$$} The original motivation of
\cite{DM} was to construct an example of a tensor category which is not
Tannakian, i.e. not the category of representations of any algebraic group (nor
supergroup).

Below we give a short overview of the construction of $\Rep(\underline{GL}_t)$.

\subsubsection{Tensor category
\texorpdfstring{$\Rep(\underline{GL}_t)$}{Rep(GL)}}\label{sssec:notn:abelian_Deligne_GL}

Let $t \in \C$.
We begin with a category\footnote{Deligne denotes this category by $\underline{Rep}_0(\gl_t)$, see \cite[Section 10]{D2};
also known as the oriented Brauer category with parameter $t$, see \cite{BCNR}.} $\mathcal{D}^0_t$ which is freely
generated,
as a SM $\C$-linear category, by one dualizable object of dimension $t$ (we
will denote it by $V$ in this section).

The category $\mathcal{D}^0_t$ can be embedded naturally into a Karoubian additive rigid symmetric
monoidal category. The Karoubian additive envelope of $\mathcal{D}^0_t$ will be denoted $Tilt(\underline{GL}_t)$, for reasons which will be explained later on.

Yet a priori it is not clear how to embed it into a tensor
(abelian) category. For $t \notin \Z$, the category $Tilt(\underline{GL}_t)$ is
indeed abelian and even semisimple. We will set $\Rep(\underline{GL}_t):=Tilt(\underline{GL}_t)$ when
$t \notin \Z$.

For $t \in \Z$, this is not the case. Yet it turns out the one can
construct a tensor category $\Rep(\underline{GL}_t)$ into which
$Tilt(\underline{GL}_t)$ embeds as a full additive
rigid SM subcategory, as was done in \cite{EHS}.

The main idea behind the construction is using the Duflo-Serganova functors
defined in \cite{DS}. This is a collection of SM functors, defined for each $m,
n \geq 1$: $$ DS_x:
\Rep(\mathfrak{gl}(m|n)) \longrightarrow \Rep(\mathfrak{gl}(m-1|n-1)).$$

Although the functors $DS_x$ are not exact, but are exact on certain
subcategories
of $\Rep(\mathfrak{gl}(m|n))$. This allows us to construct a new  tensor
category
$\Rep(\underline{GL}_t)$, $t:=m-n$, together with a collection of SM functors
$F_{m, n}: \Rep(\underline{GL}_t) \longrightarrow \Rep(\mathfrak{gl}(m|n))$ which
are compatible with the
functors $DS_x$. Note that the SM functors $F_{m, n}$ are not
exact.

The category $\Rep(\underline{GL}_t)$ should be seen as an inverse limit of the
system $(\Rep(\mathfrak{gl}(m|n)), DS_x)$.

\begin{remark}
The functors ${F}_{m,n}$ are ``local'' equivalences: the
categories $\Rep(\underline{GL}_t)$, $\Rep(\gl(m|n))$ have natural $\Z_+$-filtrations on objects which are preserved by the functors ${F}_{m,n}$, and the latter induce an equivalence between each filtra $\Rep(\underline{GL}_t)^k$ and $\Rep^k(\gl(m|n))$ for $m, n>>k$.
\end{remark}

The obtained category $\Rep(\underline{GL}_t)$ is a {\it lower highest weight
category}: namely, for each $k \geq 0$, the subcategory
$\Rep(\underline{GL}_t)^k$ is a highest weight category, whose standard,
costandard and tilting objects play the same role in each
$\Rep(\underline{GL}_t)^{k'}$ for $k' \geq k$. The full subcategory of tilting objects in $\Rep(\underline{GL}_t)$ is then precisely $Tilt(\underline{GL}_t)$.

\comment{Furthermore, for each isomorphism
class of simple objects $L \in \Rep(\underline{GL}_t)$, we have distinguished
isomorphism classes of
objects named $\Delta_L, \nabla_L, T_L$ so that whenever
$L \in \Rep(\underline{GL}_t)^k$, the above objects are respectively standard,
costandard, and indecomposable tilting objects corresponding to $L$ in
$\Rep(\underline{GL}_t)^k$.}

Finally, we mention the ``interpolation property'' of
$\Rep(\underline{GL}_t)$.

\begin{theorem}[\cite{D}]
 Let $t=n \in \Z_+$. There exists an $\C$-linear SM functor
$Tilt(\underline{GL}_{t=n})  \longrightarrow \Rep(GL_n)$ which is full and
essentially surjective, making $\Rep(GL_n)$ the quotient of
$Tilt(\underline{GL}_{t=n})$ by the ideal of negligible morphisms.
\end{theorem}

\subsubsection{Universal property}\label{sssec:notn:univ_GL}
The collection $\{\Rep(\underline{GL}_t), \Rep(\gl(m|n)) \rvert m-n =t\}$ acts as
a system of
abelian envelopes of the Deligne category $Tilt(\underline{GL}_t)$:
\begin{theorem}[\cite{EHS}]\label{old_thrm:Deligne_conj}
Let $\mathcal{C}$ be \InnaD{a tensor} category, and let $X$ be an object in $\mathcal{C}$
 of integral dimension $t$. Consider the canonical SM functor
$$F_X:
Tilt(\underline{GL}_t)\longrightarrow \mathcal{C}
$$ carrying the $t$-dimensional generator $V$ of $Tilt(\underline{GL}_t)$ to $X$.
\begin{enumerate}[label=(\alph*)]
 \item If $X$ is not annihilated by any Schur functor then $F_X$ uniquely
factors through the embedding $I:Tilt(\underline{GL}_t)\to \Rep(\underline{GL}_t)$ and gives rise
to an exact SM functor $$\Rep(\underline{GL}_t) \longrightarrow \mathcal{C}.$$
 \item If $X$ is annihilated by some Schur functor then there exists a unique
pair $m, n \in \mathbb{Z}_+$, $m-n=t$, such that $F_X$ factors through the SM
functor $Tilt(\underline{GL}_t) \longrightarrow \Rep(\gl(m|n))$ and gives rise to an exact SM
functor
 $$ \Rep(\gl(m|n)) \longrightarrow \mathcal{C}$$
 sending the standard representation $\mathbb{C}^{m |n}$ to $X$.
\end{enumerate}

\end{theorem}

\subsection{The Deligne category
\texorpdfstring{$\PP$}{P}}\label{ssec:notn:Deligne_P_Karoubi}

Let $\mathfrak{S}$ be a $\C$-linear category on two objects: $\C, \Pi \C$, both
having one-dimensional endomorphism spaces and no morphisms between them. The
category $\mathfrak{S}$ has a rigid SM structure: considered as
the full subcategory of $\sVect$ whose objects
are $\C, \Pi \C$, it inherits the SM structure on vector superspaces.

The Deligne category $\PP_0$ is the universal SM
$\mathfrak{S}$-module category which is generated by an
object $\tilde{V}$ with a fixed symmetric non-degenerate form
$\omega_{\tilde{V}}:\tilde{V} \otimes \tilde{V} \to \Pi \triv$. The objects in
$\PP_0$ are enumerated by pairs
$[(n, \epsilon)]$ where $n \in \Z_{\geq 0}$ and $\epsilon \in
\Z/2\Z = \{\bar{0},\bar{1}\}$. In particular, $[(1, \bar{0})]=\tilde{V}$. The
endomorphism spaces in $\PP_0$
described by the signed Brauer algebras (see e.g. \cite{BDE+2,CoulEhr,KT}), and
the monoidal structure is given by
$$ [(n, \epsilon)] \otimes [(n', \epsilon')] := [(n+n', \epsilon +
\epsilon')].$$
\begin{example}
 In this notation, $\triv = [(0,
\bar{0})]$ and $[(1, \epsilon)]$ is a generator (for any $\epsilon$).
\end{example}

The action
of $\mathfrak{S}$ given by $\Pi [(n, \epsilon)] = [(n, \bar{1}-\epsilon)]$.

A formal definition of this monoidal category
via generators and relations can be found in \cite{BE, CoulEhr,
KT}\footnote{In \cite{KT}, this category is called the
marked Brauer category for $\delta=0, \eps=-1$.}.

We will denote by $\PP$ the Karoubian additive envelope of $\PP_0$.
For details on $\PP$, see \cite{CoulEhr}.

\begin{remark}
 Caveat: the definitions in \cite{BE, CoulEhr, KT} produce a
monoidal supercategory, namely a category enriched over $\sVect$. In
particular,
the objects $[(n, \pm \eps)]$ are identified for any $n$.

To pass from the
supercategory $\widetilde{\PP}$ to our version of $\PP$, we define
$\Hom{\PP}([(n_1, \eps_1)], [(n_2, \eps_2)])$ to be the morphisms in
$\widetilde{\PP}$ between the corresponding objects $[n_1]$, $[n_2]$ whose
\InnaB{parity} is $(-1)^{\eps_1 + \eps_2}$.
\end{remark}

Finally, we define the functor $I_n: \PP \to \Rep(\p(n))$ to be the $\C$-linear
SM $\sVect$-functor which sends the
generator $\tilde{V}$ of $\PP$
to $V_n$, and the form $\omega_{\tilde{V}}$ to the form $\omega_n: V_{n}
\otimes V_{n} \to \Pi \C$. Such a functor is unique up to a
$\otimes$-isomorphism of $\sVect$-functors\footnote{The group of the
automorphisms of $I_n$ is $\Z /2\Z$.}.

\begin{lemma}\label{lem:I_n_full}
 The functor $I_n$ is full.
\end{lemma}
This statement is a direct consequence of the following result:
\begin{proposition}\label{prop:contractions}
 The $\mu_2$-graded space $\mathtt{s}\Hom_{\mathfrak{p}(n)}(V_n^{\otimes k},
\triv)$ is non-zero only
when $k$ is even, and is spanned by morphisms given by partitioning the $k$
factors into (disjoint) pairs, and considering a tensor product of $k/2$
contraction maps $V_n^{\otimes 2} \to \Pi \triv$ on these pairs.
\end{proposition}
This result is proved in \cite{DLZ} using geometric
methods. We present an alternative proof of their result in Appendix
\ref{appendix_A} (see Corollary \ref{appcor:contractions}), using translation
functors and Temperley-Lieb relations between them.

\section{The infinite periplectic Lie superalgebra}\label{sec:inf_p_n}
In this section we recall some facts about the category of algebraic
representations of Lie superalgebra $\p(\infty)$; for more details, see
\cite{SerInf}. We also prove some additional results which will be used later.
\subsection{Definition}
Let $V_{\infty} = \C^{\infty|\infty}$ be a countable-dimensional vector
superspace,\InnaC{ with even part $V_{\bar{0}, \infty}$ and odd part $V_{\bar{1}, \infty}$ (that is, $V_{\infty} = V_{\bar{0}, \infty} \oplus V_{\bar{1}, \infty}$)}. We fix a non-degenerate symmetric pairing form $\omega: V_{\infty} \otimes V_\infty \to \Pi \C$.
Consider the Lie superalgebra $\p(\infty) \subset \gl(\infty|\infty)$ of
finite-rank operators on $V_{\infty}$ preserving $\omega$.

For each $n \geq 1$, consider an orthogonal decomposition $V_{\infty} = V_n
\oplus V_n^{\perp}$, where $V_n = \C^{n|n}$, such that $\omega$ restricts to
non-degenerate forms on both $V_n$ and $V_n^{\perp}$.
This decomposition induces an
inclusion $\p(n) \subset \p(\infty)$ (in fact, $\p(\infty) =
\varinjlim_n \p(n)$), and defines a Lie super subalgebra
$$\p(n)^{\perp}= \p(\infty) \cap \End(V_n^{\perp}).$$ This
subalgebra commutes with $\p(n)$.

Let $\Rep(\p(\infty))$ be the abelian SM category of {\it algebraic}
$\p(\infty)$-modules: namely, modules occurring as subquotients of finite
direct
sums of tensor powers of $V_{\infty}$. This category is studied extensively
in \cite{SerInf}; see also \cite{NSS}.

\begin{remark}\label{rmk:Koszul_duality}
Let $\Rep(\mathfrak{sp}(\infty))$ be the category of algebraic \InnaD{super-}representations of the Lie algebra $\mathfrak{sp}(\infty) = \bigcup_{n \geq 1} \mathfrak{sp}(2n)$. \InnaD{This category is equivalent to the category of algebraic \InnaD{super-}representations of
$\mathfrak{spo}(\infty|\infty)$ (see \cite{SerInf}).}

The category $\Rep(\p(\infty))$ is Koszul dual
to $\Rep(\mathfrak{sp}(\infty))$ and not equivalent to it, as
stated in \cite{SerInf} (cf. \cite{NSS}). The Koszul duality can be seen as follows. Consider the category
$\Rep(\gl(\infty))$ of algebraic \InnaD{super-}representations of $\gl(\infty)$ (see \cite{PS, Ser}). This category is equivalent to the category of algebraic \InnaD{super-}representations of $\gl(\infty|\infty)$.

Let $\C^{\infty}$ be the defining representation of $\gl(\infty)$, and $\C^{\infty}_*$ its restricted dual, with the obvious pairing $\C^{\infty} \otimes \C^{\infty}_* \to
\C$.

Consider the object $V_{\infty} = \C^{\infty} \oplus \Pi\C^{\infty}_*$ in $\Rep(\gl(\infty))$.

Then the category
$\Rep(\p(\InnaB{\infty}))$ can be described as the category of finite-length
$\gl(\infty)$-equivariant (super-)modules over the Lie algebra object
$\Pi S^2(V_{\infty})$ in $\Rep(\gl(\infty))$.

In the same spirit, let $E = \left(\C^{\infty} \oplus \C^{\infty} \right)\oplus
\Pi\C^{\infty}_*$, and consider a symplectic form on $E$ given by a symplectic
form on the even part $\C^{\infty} \oplus \C^{\infty} $, and a symmetric form
on the odd part $\C^{\infty}_*$.
The superalgebra preserving such a form is $\mathfrak{spo}(\infty|\infty)$,
whose algebraic representations can be described as follows: they form the
category of finite-length
$\gl(\infty)$-equivariant (super-)modules over the Lie algebra object
$S^2(E)$ in $\Rep(\gl(\infty))$. Since $E\cong V_{\infty}$, we can consider
these as (super-)modules over the Lie algebra object
$S^2(V_{\infty})$ in $\Rep(\gl(\infty))$.

The corresponding enveloping algebras of the Lie algebras $\Pi S^2 (V_\infty)$, $S^2(V_{\infty})$ will then be
$\bigwedge(S^2(V_{\infty}))$ and $Sym(S^2(V_{\infty}))$. These are clearly Koszul dual.

\end{remark}

We have a left-exact functor
$$\Phi_n: \Rep(\p(\infty)) \longrightarrow \Rep(\p(n)), \;\; M \longmapsto
M^{\p(n)^{\perp}}$$ as well as an exact SM functor
$$Res: \Rep(\p(\infty)) \longrightarrow \Rep(\gl(\infty))$$
retricting to the even part $\gl(\infty)$ of $\p(\infty)$.

\begin{lemma}\label{lem:Phi_monoidal}
 The functor $\Phi_n$ is a \InnaA{SM functor}.
\end{lemma}
\begin{proof}
 We need to check that $$M^{\p(n)^{\perp}} \otimes K^{\p(n)^{\perp}} \cong (M
\otimes K)^{\p(n)^{\perp}}$$ for any $M, K$. We clearly have an inclusion
$\subset$, so we only need to check that we have an isomorphism of vector
spaces.
Let us now forget about the action of $\p(n)$ and consider
$(-)^{\p(n)^{\perp}}$ as composition of
functors $$Res^{\p(\infty)}_{\p(n)^{\perp}}: \Rep(\p(\infty)) \longrightarrow
\Rep(\p(\infty)) $$ (notice that we use here that $\p(n)^{\perp}
\cong \p(\infty)$) and invariants $\Hom_{\p(\infty)}(\C, -)$.
Hence it is enough to check that the functor of taking invariants
$\Hom_{\p(\infty)}(\C, -)$ in $\Rep(\p(\infty))$ is a SM $\sVect$-functor.

We now use the fact that $\C$ is a simple injective object in
$\Rep(\p(\infty))$,
as is any tensor power of $V_{\infty}$ (these generate the subcategory of
injective objects under taking $\oplus$ and direct summands; see
\cite{SerInf}).
This allows us to write for any $M, K \in \Rep(\p(\infty))$
$$M = \C^{\oplus m} \oplus M', \;\; K = \C^{\oplus k} \oplus K'$$ where
\begin{equation}\label{eq:p_infty}
 \dim \Hom_{\p(\infty)}(\C, M') =\dim \Hom_{\p(\infty)}(\C, K')=0
\end{equation}
 Hence it remains to check that $\dim \Hom_{\p(\infty)}(\C, M' \otimes K')=0$.

Indeed, let $I, J$ be direct sums of tensor powers of $V_{\infty}$ such that
$M'
\subset I$, $K' \subset J$. By \eqref{eq:p_infty}, we may choose $I, J$ which
do
not have $\C$ as direct summands. Hence $I \otimes J$ is also a direct sum of
{\it positive} tensor powers of $V_{\infty}$, which implies that
$$\dim \Hom_{\p(\infty)}(\C, M' \otimes K')= \dim \Hom_{\p(\infty)}(\C, I
\otimes J) =0.$$ This completes the proof of the lemma.

\end{proof}

\subsection{Restriction functors}\label{ssec:restr_functors}
\InnaC{Let $p,  q \geq 0$ such that $p+q = n$. Let $W_{p,q}\subset V_n$ be an
isotropic subspace of dimension $(p|q)$ and
$W'_{q,p}$ be another isotropic subspace of dimension $(q|p)$ such that
$V_n=W_{p,q}\oplus W'_{q,p}$. Note that the non-degenerate symmetric form
$\beta$
establishes an isomorphism $W'_{p,q}\simeq \Pi W_{p,q}^*$. Let
$$\g_{p,q}=\{X\in \p(n)\,|\, X(W_{p,q})\subset W_{p,q},\,X(W'_{q,p})\subset
W'_{q,p}\}.$$
\begin{remark}
 Choose a basis $e_1, \ldots, e_n, e_{1'}, \ldots, e_{n'}$ for
$V_n$ such that $$W_{p,q} = span\{e_1, \ldots, e_p, e_{(p+1)'}, \ldots,
e_{n'}\}.$$ Then $X \in \g_{p,q}$ would be of the form
$$\left(\begin{array}{cc|cc}
A &{} &{} &B \\
{} &-D^t &B^t &{} \\ \hline
{} &-C^t &-A^t &{} \\
C &{} &{} &D
  \end{array}\right)$$
\end{remark}

One can easily see that $\g_{p,q}$ is isomorphic to
$\gl(p|q)=\mathfrak{gl}(W_{p,q})$, where the isomorphism is defined by
$$\gl(p|q) \to \g_{p,q}, \;\; x\mapsto (x \oplus \Pi x^*): W_{p,q}\oplus
W'_{q,p} \to W_{p,q}\oplus W'_{q,p} = V_n.$$

Consider the inclusion of Lie superalgebras $\g_{p,q}\simeq \gl(p|q) \subset
\p(n)$. This induces a tensor $\sVect$-functor $$Res: \Rep(\p(n))
\longrightarrow \Rep(\gl(p|q))$$

Applying this functor to the standard representation $V_n = \C^{n|n}$ of
$\p(n)$, we have: $Res(V) = \C^{p|q} \oplus \Pi (\C^{p|q})^*$.}
\begin{lemma}\label{lem:Phi_res_commute}
 We have a natural isomorphism
 $$\xymatrix{&\Rep(\p(\infty)) \ar[rr]^-{Res}
\ar[d]^{\Phi_n} &{}
&\Rep(\gl(\infty|\infty))
\ar[d]^{\Gamma_{p|q}}\\ &\Rep(\p(n)) \ar@{=>}[rru] \ar[rr]^-{Res}
&{} &\Rep(\gl(p|q)) }$$
for any $p+q=n$; here $\Gamma_{p|q}: \Rep(\gl(\infty)) \to
\Rep(\gl(p|q))$ is the functor of invariants with respect to the
corresponding Lie super subalgebra $$\gl(p|q)^{\perp} =
\End\left(\left(\C^{p|q}\right)^{\perp}\right) \subset
\gl(\infty|\infty).$$
\end{lemma}
\comment{\begin{remark}
This is an example of the idea that $\Rep(\p(\infty))$ is the
universal SM $\sVect$-category generated by an object
$V_{\infty}$ and
a pairing $V_{\infty} \otimes V_{\infty} \to \Pi\C$.
\end{remark}}

Until the end of this section, we will use the following shorthand:
\begin{notation}\label{notn:g_notn}
 We denote $\tilde{\g}=\p(\infty)$, with triangular decomposition $\tilde{\g}
\cong
\tilde{\g}_{-1} \oplus \tilde{\g}_0 \oplus \tilde{\g}_1$ where $$\tilde{\g}_0
\cong \gl(\infty), \;\; \tilde{\g}_{-1}
\cong \InnaC{S^2 V_{\bar{1}, \infty}}, \;\; \tilde{\g}_1 \cong \InnaC{ \wedge^2 V_{\bar{0}, \infty}}.$$
Then the action of $\tilde{\g}_{\pm 1}$ on any $\tilde{\g}$-module is
$\tilde{\g}_0$-equivariant.

We will also use the notation ${\g}:=\p(n)$, with triangular decomposition
${\g}
\cong
{\g}_{-1} \oplus {\g}_0 \oplus {\g}_1$.
\end{notation}

\begin{proof}

In the spirit of the proof of Lemma \ref{lem:Phi_monoidal}, it is enough to
check this for $p=q=0$: namely, it is enough to check that for any
$\tilde{\g}$-module $M$ and any $\tilde{\g}_0$-map $f:\C \to  M$,
this map is in fact a map of $\tilde{\g}$-modules.

In particular, the $\tilde{\g}_1$-module structure on $M$ will give a
$\tilde{\g}_0$-equivariant
map $\tilde{f}: \tilde{\g}_1\otimes \C \to M$; the $\tilde{\g}_0$-module
$\tilde{\g}_1$ is
irreducible, hence this map is either $0$ (which means that the original map
$f$ was $\tilde{\g}_{1}$-equivariant, as required) or injective. In the latter
case, we
see that $\tilde{\g}_1$ acts faithfully on the image of $f$. Yet this
contradicts the
``large annihilator condition'' given in \cite[Proposition 4.2]{SerInf} on the
objects of
$\Rep(\p(\infty))$: every vector in $M$ must be annihilated by some
finite-corank subalgebra of $\tilde{\g}$, and such a subalgebra must have a
non-trivial
intersection with $\tilde{\g}_1$. Hence $f$ was $\tilde{\g}_{1}$-equivariant,
and similarly one
shows that $f$ is $\tilde{\g}_{-1}$-equivariant as well. This proves the result
of the
lemma.
\end{proof}

\begin{corollary}\label{cor:Phi_res_borel_commute}
 We have a natural isomorphism
 $$\xymatrix{&\Rep(\tilde{\g}) \ar[rr]^-{Res}
\ar[d]^{\Phi_n} &{}
&\Rep(\tilde{\g}_0 \oplus \tilde{\g}_{\pm 1})
\ar[d]^{\Gamma_n}\\ &\Rep(\g)  \ar@{=>}[rru] \ar[rr]^-{Res}
&{} &\Rep(\g_0) }$$ induced by the functors in Lemma
\ref{lem:Phi_res_commute}.
\end{corollary}

The isomorphism classes of simple modules in $\Rep(\p(\infty))$ are parametrized (up to \InnaB{parity} shift)
by non-decreasing integer sequences $\lam=(\lambda_1, \lambda_2, \ldots)$ such
that $\lambda_n =0$ for $n>>0$. Let $L_{\infty}(\lam)$ denote the simple $\p(\infty)$-representation corresponding to $\lam$, with even highest weight vector.

Alternatively, these can be parametrized by partitions of arbitrary
size, as done in \cite{SerInf}, with sequence $\lambda$ as above corresponding
to Young diagram $\pmb{\lam}$ (notation as in Section
\ref{ssec:notn:periplectic}).

Let $\underline{\lambda}^-$ denote the set of non-decreasing integer sequences
$\mu= (\mu_1, \mu_2, \ldots)$ which differ from $\lambda$ by exactly one
entry\footnote{For any $\mu \in \underline{\lambda}^{\pm}$, $\pmb{\mu}$ is
obtained from $\pmb{\lam}$ by adding or removing one box.},
and $\sum_i {\lam_i
- \mu_i} =1$. Similarly, let  $\underline{\lambda}^+$ denote the set of
non-decreasing integer sequences $\mu = (\mu_1, \mu_2, \ldots)$ which differ
from $\lambda$ by exactly one entry, and $\sum_i {\lam_i - \mu_i} =-1$.

\InnaA{This parametrization of simples is non-orthodox, see examples below.}
\begin{example}
\InnaA{The natural representation $V_{\infty}$ of $\p(\infty)$ has highest weight $-\eps_1$,
with odd highest weight vector; hence $$V_{\infty} \cong \Pi L_{\infty}(-\eps_1) = \Pi L_{\infty}\left(\yng(1)\right).$$ Similarly, we have $$\wedge^2 V_{\infty} \cong L_{\infty}(-2\eps_1) = L_{\infty}\left(\yng(1,1) \right), \;\; S^2 V_{\infty}
\cong L_{\infty}(-\eps_1 - \eps_2) = L_{\infty}\left(\yng(2) \right).$$}
\end{example}
The following lemma is proved in \cite{SerInf}:
\begin{lemma}\label{lem:ses_infty}
  For any $\lambda $ as above, we have a short
exact
sequence
 $$ 0 \to \bigoplus_{\mu\in \underline{\lambda}^-} \Pi L_{\infty}(\mu)
\longrightarrow
L_{\infty}(\lambda) \otimes V_{\infty} \longrightarrow \bigoplus_{\mu \in
\underline{\lambda}^+} L_{\infty}(\mu) \to 0$$
\end{lemma}

\begin{lemma}\label{lem:p_infty_cosocle}
 Every simple object in $\Rep(\p(\infty))$ occurs in the cosocle of some injective
object.
\end{lemma}
\begin{proof}
 Recall that the isomorphism
classes of simple objects (up to \InnaB{parity} shift) in
$\Rep(\p(\infty))$ may be enumerated by Young
diagrams of arbitrary size.
For any Young diagram $\pmb{\beta}$, let $L_{\infty}(\pmb{\beta})$
denote
the corresponding simple object with even highest weight vector, and $Y(\pmb{\beta})$ its injective hull.

We use
Koszul duality between $\Rep(\p(\infty))$ and $\Rep(\mathfrak{sp}(\infty))$ (see Remark \ref{rmk:Koszul_duality})
to compute the multiplicities of composition factors in the socle filtration of
the indecomposable injective objects of $\Rep(\p(\infty))$. The injective resolutions of simple objects in $\Rep(\mathfrak{sp}(\infty))$ are given in \cite[4.3.5, 4.3.9]{SS}; this immediately implies that for any Young
diagrams $\pmb{\beta}, \pmb{\zeta}$, and any $k \geq 0$, we have:
$$[\overline{soc}^k Y(\pmb{\zeta}):
L_{\infty}(\pmb{\beta})] = \sum_{\pmb{\gamma} \vdash k,
\pmb{\gamma} \in QSym}
N_{\pmb{\gamma},
\pmb{\beta}}^{\pmb{\zeta}}$$ where $QSym$ is the set of all
quasi-symmetric
partitions
(such
that $\pmb{\gamma}^{\vee}_i =
\pmb{\gamma}_i -1$) and $N_{\pmb{\gamma},
\pmb{\beta}}^{\pmb{\zeta}}$ is the Littlewood-Richardson
coefficient.

\InnaA{Let $\pmb{\beta}$ be any Young diagram. We wish to find $\pmb{\zeta}$ such that $L_{\infty}(\pmb{\beta})\subset cosoc Y(\pmb{\zeta})
$.

Take $\pmb{\delta}$
to be a rectangular
partition of length $\abs{\pmb{\beta}} $ and width
$\abs{\pmb{\beta}}+1$, and let $\pmb{\zeta} =
\pmb{\delta} + \pmb{\beta}$ be the Young diagram with $\pmb{\zeta}_i =
\pmb{\delta}_i + \pmb{\beta}_i$ for any $i$.

Then $N_{\pmb{\delta},
\pmb{\beta}}^{\pmb{\zeta}} =1$, so $[\overline{soc}^{k} Y(\pmb{\zeta}):
L_{\infty}(\pmb{\beta})] = 1$ for $k = \abs{\pmb{\delta}}$.

Let us show that $\overline{soc}^{k} Y(\pmb{\zeta})$ is contained in the cosocle; that is, we want to show that for any $\pmb{\alpha}$ such
that $\abs{\pmb{\alpha}} < \abs{\pmb{\beta}}$, we have
$[Y(\pmb{\zeta}): L_{\infty}(\pmb{\alpha})]=0$.

Indeed, assume that $N_{\pmb{\gamma}, \pmb{\alpha}}^{\pmb{\zeta}} \neq 0$ for some $\pmb{\gamma} \in QSym$.
Then $\ell(\pmb{\gamma}) \leq
\ell(\pmb{\zeta}) = \ell(\pmb{\delta})$, where $\ell()$ denotes the length of the partition, that is $\ell(\pmb{\gamma}) = \pmb{\gamma}^{\vee}_1$. Since $\pmb{\gamma} \in QSym$, we have $$\pmb{\gamma}_1 = \ell(\pmb{\gamma})+1 \leq \ell(\pmb{\delta})+1=\pmb{\delta}_1.$$
Hence $\pmb{\gamma}$ can be embedded into $\pmb{\delta}$. On the other hand, $\abs{\pmb{\gamma}} = \abs{\pmb{\zeta}} -
\abs{\pmb{\alpha}} >
\abs{\pmb{\zeta}} -
\abs{\pmb{\beta}} = \abs{\pmb{\delta}}$. This is a contradiction, thus the required statement is proved. }
\end{proof}

\section{The subcategories
\texorpdfstring{$\Rep^k(\p(n))$}{Repk}}\label{sec:Rep_k}
\subsection{Definition}\label{ssec:Rep_k_def}
Throughout this subsection, we will work with fixed $n \geq 1$, and will omit
it from the notation.
\begin{definition}\label{def:repk}
 Let $\Rep^k(\p(n)) \subset \Rep(\p(n))$ be the full subcategory of
$\p(n)$-modules occurring as subquotients in finite direct sums of
$V_n^{\otimes j}$ (or their \InnaB{parity} shifts)
for
$j=0, \ldots, k$.
\end{definition}

 Clearly, $\Rep^k(\p(n))$ is an abelian subcategory of $\Rep(\p(n))$.
Moreover, although it is not closed under $\otimes$, it is closed under
(tensor) duality: for any $M \in \Rep^k(\p(n))$, its dual $M^*$ also
belongs to $\Rep(\p(n))$.

\begin{remark} By Corollary \ref{cor:Rep_gen_by_V}, the category $\Rep(\p(n))$
is a direct limit of the subcategories
$\Rep^k(\p(n))$ when $k \to \infty$.
\end{remark}

\begin{lemma}\label{lem:simples_rep_k}
 Let $k <n+2$. The simple modules $L_n(\lambda) \in \Rep^k(\p(n))$ are precisely
those for which $\lambda$ is $k$-admissible.
\end{lemma}

\begin{proof}
 We prove the statement by induction on $k\in \Z_{\geq 0}$, with the trivial
base case
$k=0$. Assume the claim is true for $k$; we will prove it for $k+1$.

First of all, we show that for any $(k+1)$-admissible $\lambda$, the
module $L_n(\lambda)$ lies inside
$\Rep^{k+1}(\p(n))$. Indeed, we only need to check this for weights $\lambda$
such that $\abs{\lambda} =
k+1$, $\lambda \geq 0$, in which case this follows directly from
the construction above.

Now, we need to check that for any $k$-admissible $\lambda$ and any
subquotient $L_n(\mu)$ of $V_n \otimes
L_n(\lambda)$, its highest weight $\mu$ is $(k+1)$-admissible.

Indeed, $\Delta_n(\lambda) \otimes V_n$ has a filtration by standard objects
$\Delta(\mu)$ with $\mu \in \{\lambda \pm \eps_i \rvert i \geq 1\}$ (see
\cite[Section 4]{BDE+}). Hence for any subquotient $L_n(\mu)$ of $V_n \otimes
L_n(\lambda)$, we have $\mu \leq \lambda \pm
\eps_i$ for some $i$; thus $\mu$ satisfies: $\abs{\mu} \leq \abs{\lambda}
+1 \leq k+1$.

To check that $\mu \geq 0$, it is enough to check that
$\lambda \geq 0$ implies
that $\mu \geq 0$ for any $\mu = \lambda \pm \eps_i$, $i \geq 1$.

Let us check that $\mu \geq 0$ for $\mu = \lambda \pm \eps_i$. This is clearly
true whenever $\mu = \lambda - \eps_i$ or when $\mu = \lambda + \eps_i$, $i<n$,
so we may assume that $i=n$, $\lambda_n =0$, $\mu_n =1$. In that case $d_{\mu}$
would be obtained from $d_{\lam}$ by moving the black ball in position $n-1$ to
position $n$:
$$d_{\lambda} = \xymatrix @C=1pc @M =0pc { &\underset{n-2}{\bullet}
&\underset{n-1}{\bullet} &\underset{n}{\circ} &\underset{n+1}{\circ} } \;\;\;
\text{  and  }\;\;\;
d_{\mu} = \xymatrix @C=1pc @M =0pc { &\underset{n-2}{\bullet}
&\underset{n-1}{\circ} &\underset{n}{\bullet} &\underset{n+1}{\circ}  }$$

In that case, $L_n(\mu)$ would be a subquotient of $\Theta_i L_n(\lambda)$ for
some $i$ (in fact, for $i=n+1$), where $\Theta_i$ is the translation functor
as in Theorem \ref{old_thrm:transl_func}. Then
$$\dim \Hom_{\p(n)}(\Theta_{i+1} P_n(\mu),L_n(\lambda)) =
\dim \Hom_{\p(n)}(P_n(\mu), \Theta_i L_n(\lambda)) = 1$$ and hence
$\Theta_{i+1} P_n(\mu)$ has $P_n(\lambda)$ as a direct summand. But the latter
is impossible, due to the translation rules for indecomposable projectives
\cite[Section 7]{BDE+}.
This completes the proof of the lemma.
\end{proof}

\begin{lemma}\label{lem:rep_k_enough_proj}
 The category $\Rep^{k}(\p(n))$ has enough projective and injective objects.
\end{lemma}
\begin{proof}
 Consider the inclusion functor $i: Rep^{k}(\p(n)) \to \Rep(\p(n))$. This
functor is exact and has a left
adjoint $i^*$ and a right adjoint ${}^*i$. Then $i$ takes
projective objects to
projective objects, hence $i^*(P_n(\lambda))$ is projective in $\Rep^{k}(\p(n))$
for any $\lambda$. Furthermore, $i^*(P_n(\lambda)) \twoheadrightarrow
L(\lambda)$ for any $L(\lambda) \in \Rep^k(\p(n))$, which means that there are
enough projective objects in $\Rep^{k}(\p(n))$.

Similarly, ${}^*i$ takes injective objects to injective objects, which implies
that there are
enough injective objects in $\Rep^{k}(\p(n))$.
\end{proof}

\subsection{Standard objects in \texorpdfstring{$\Rep^k(\p(n))$}{Repk(p)}}\label{ssec:Rep_k_standard}
We now describe the standard (highest weight) objects in the category $\Rep^{k}(\p(n))$.

Let $k \leq n$.

Consider $V_{\bar{1}, n}^{\otimes k}  \subset V_n^{\otimes k}$. Then
$\g_{-1}$
acts on $V_{\bar{1}, n}^{\otimes k}$ by zero. Consider the $GL(V_{\bar{0},
n}) \times S_k$-decomposition
$$V_{\bar{1}, n}^{\otimes k} \cong \bigoplus_{\abs{\lambda}= k}
S^{\pmb{\lambda}}\InnaC{V_{\bar{1}, n}}\otimes \pmb{\lambda} \cong
\Pi^k\bigoplus_{\abs{\lambda}= k}
S^{\pmb{\lambda}^{\vee}}V_{\bar{0}, n}^\InnaC{*} \otimes \pmb{\lambda}.$$ \InnaC{By definition of $\Delta(\lambda)$ (see \cite[Section 3.1]{BDE+}),} we have
a map of $\p(n)\boxtimes \C[S_k]$-modules
\begin{equation}\label{eq:trunc_delta_def_map}
 \bigoplus_{\abs{\lambda}=  k} \Delta_n (\lambda) \otimes
\pmb{\lambda^{\vee}}
 \to \Pi^k V_n^{\otimes k}
\end{equation}
 which is non-zero on each of the summands.

 \begin{definition}\label{def:trunc_delta}
  Let $\lambda \geq 0$, $\abs{\lambda}=  k$. We denote the image of the map
$$ \Delta_n (\lambda) \longrightarrow \Pi^k
S^{\pmb{\lambda}^{\vee}} V_n $$ by $\overline{\Delta}^k_n
(\lambda)$.

More generally, for any $\lambda \geq 0$, $\abs{\lam}
\leq k$, we
denote $$\overline{\Delta}^k_n (\lambda) :=
\overline{\Delta}^{\abs{\lambda}}_n (\lambda).$$
 \end{definition}

Similarly, we define the objects $\overline{\nabla}^k_n(\lambda)$:

Consider the $\g_0 \oplus \g_1$-map $ V_n^{\otimes k} \to \InnaC{ V_{\bar{1},
n}^{\otimes k}}$ with $\g_1$ acting trivially on the latter. \InnaC{As before, we have an isomorphism of $GL(V_{\bar{0},
n}) \times S_k$-modules $$ V_{\bar{1},
n}^{\otimes k}\cong \Pi^k\bigoplus_{\abs{\lambda}= k}
S^{\pmb{\lambda}^{\vee}}V_{\bar{0}, n}^\InnaC{*} \otimes \pmb{\lambda}.$$ Hence} we have
a map of $\p(n)\boxtimes \C[S_k]$-modules
\begin{equation}\label{eq:trunc_nabla_def_map}
 \bigoplus_{\abs{\lambda}=  k} V_n^{\otimes k} \to \Pi^{k} \nabla_n
(\lambda)
\otimes
\pmb{\lambda}^{\InnaC{\vee}}
\end{equation}
 which is non-zero on each of the summands.

 \begin{definition}\label{def:trunc_nabla}
  Let $\lambda \geq 0$, $\abs{\lambda}=  k$. We denote the image of the map
$$ \Pi^{k} S^{\pmb{\lambda}^{\InnaC{\vee}}} V_n \to \nabla_n
(\lambda)$$ by $\overline{\nabla}^k_n
(\lambda)$.

More generally, for any $\lambda \geq 0$, $\abs{\lam}
\leq k$, we
denote $$\overline{\nabla}^k_n (\lambda) :=
\overline{\nabla}^{\abs{\lambda}}_n (\lambda).$$
 \end{definition}

Clearly the highest weight module $\overline{\Delta}^k_n (\lambda)$ lies in
$\Rep^k(\p(n))$, and so does its simple head
$L_n(\lambda)$. The latter is also the socle of $\overline{\nabla}^k_n
(\lambda)$.

\subsection{Connection to representations of
\texorpdfstring{$\p(\infty)$}{p(infty)}}

Let $\Rep^k(\p(\infty))$ be defined as the full subcategory of $\Rep(\p(\infty))$
whose objects occur as subquotients in finite direct sums of $\triv, V_{\infty},
V_{\infty}^{\otimes 2}, \ldots, V_{\infty}^{\otimes k}$.

Clearly, for any $n\geq 1$, $\Phi_n \left(Rep^k(\p(\infty))\right) \subset
Rep^k(\p(n))$.

\begin{lemma}\label{lem:Phi_exact}
 For any $n> 2k \geq 0$, the restriction of $\Phi_n$ to
$\Rep^k(\p(\infty))$
is exact and faithful.
\end{lemma}
\begin{proof}
The functor $$Res: \Rep(\p(\infty)) \to \Rep(\gl(\infty))$$ is an exact
monoidal functor taking $V_{\infty}$ to $\C^{\infty} \oplus \Pi \C_*^{\infty}$,
where $\C_{\infty}$ is the tensor generator of $\Rep(\gl(\infty))$.
Hence the image of $\Rep^k(\p(\infty))$ under $Res$ lies in
$\Rep^k(\gl(\infty))$, which is the full subcategory whose objects are
(up to change of \InnaB{parity}) subquotients of finite direct sums of
$\left(\C^{\infty}\right)^{\otimes r} \otimes
\left(\C_*^{\infty}\right)^{\otimes s}$ for $r+s \leq k$.

Let $p=\lfloor \frac{n}{2}\rfloor$ and let $q = n-p$. Consider the monoidal
functor $\Gamma_{p|q}: \mathtt{s}\Rep(\gl(\infty)) \to
\Rep(\gl(p|q))$. This functor is exact on the subcategory
$\Rep^k(\gl(\infty))$ when $k< \min(p, q)$, i.e. when $2k<n$ (see
\cite[Theorem 6.1.3]{EHS}), and
hence we
have a natural isomorphism
 $$\xymatrix{&\Rep^k(\p(\infty)) \ar[rr]^-{Res}
\ar[d]^{\Phi_n} &{}
&\Rep^k(\gl(\infty))
\ar[d]^{\Gamma_n}\\ &\Rep(\p(n))  \ar@{=>}[rru] \ar[rr]^-{Res}
&{} &\Rep(\gl(p|q)) }$$
where the functors are all exact and faithful except perhaps $\Phi_n:
Rep^k(\p(\infty)) \to \Rep(\p(n))$, which is known to be left exact. Hence the
latter is
exact and faithful as well.
\end{proof}

Throughout the paper, we will use the following notation:
\begin{notation}
 A
$\p(n)$-weight $\lam = (\lam_1, \lam_2, \ldots, \lam_n)$ such that $\lam \geq
0$ will be interpreted as a weight for $\p(\infty)$ by considering the
infinite sequence $ (\lam_1, \lam_2, \ldots, \lam_n, 0, 0, \ldots)$ (adding an
infinite tail of zeroes). The latter sequence will also be denoted $\lam$.

For $\lam \not\geq 0$, we set $L_{\infty}(\lam):=0$.
\end{notation}

\begin{proposition}\label{prop:Phi_simples_to_delta}
 Let $n> 2k \geq 0$, and let $\lam$
be a $k$-admissible weight for $\p(n)$.
Consider the corresponding simple module $L_{\infty}(\lambda)
\in \Rep^k(\p(\infty))$. Then $$\Phi_n(L_{\infty}(\lambda)) \cong
\overline{\Delta}_n^k(\lambda).$$
\end{proposition}
\begin{proof}
\InnaC{
It is enough to check this statement for $\lambda$ such that $k=\abs{\lam}$.

Recall Notation \ref{notn:g_notn}. By \cite[Lemma 17]{SerInf} we have:
$$ L_{\infty}(\lam) = \bigcap_{\psi \in \Hom_{\p(\infty)}( V_{\infty}^{\otimes k}, V_{\infty}^{\otimes k-2})} \Ker \left( \psi \rvert_{S^{\pmb{\lam}^{\vee}} V_{\infty}} \right)$$

Since $\Phi_n$ is exact on $\Rep^k(\p(\infty))$, we have:
$$\Phi_n( L_{\infty}(\lam)) = \bigcap_{\psi \in \Hom_{\p(\infty)}( V_{\infty}^{\otimes k}, V_{\infty}^{\otimes k-2})} \Ker \left( \Phi_n(\psi) \rvert_{S^{\pmb{\lam}^{\vee}} V_{n}} \right)$$

\comment{
In fact, we have: $$\Phi_n( L_{\infty}(\lam)) = \bigcap_{\psi \in \Hom_{\p(n)}( V_{n}^{\otimes k}, V_{n}^{\otimes k-2})} \Ker \left( \psi \rvert_{S^{\pmb{\lam}^{\vee}} V_{n}} \right)$$}

Next, consider the map $f:\Delta_n(\lam) \to V_{n}^{\otimes k}$ described in Definition \ref{def:trunc_delta}. For any $s< k$, we have: $$ \Hom_{\p(n)}( \Delta_n(\lam), V_{n}^{\otimes s} ) =  \Hom_{\g_{0} \oplus \g_{-1}}( S^{\pmb{\lam}^{\vee}} V_{\bar{1},n}, Res^{\p(n)}_{\g_{0} \oplus \g_{-1}} V_{n}^{\otimes s})  =0$$
Hence $\Im(f) \subset \Ker(\psi)$ for any $ \psi \in \Hom_{\p(n)}( V_{n}^{\otimes k}, V_{n}^{\otimes s})$. Thus $\overline{\Delta}_n^k(\lambda) \subset \Phi_n( L_{\infty}(\lam))$.

We now consider the surjective map of $\tilde{\g}$-modules $$\psi:U(\tilde{\g}) \otimes_{U(\tilde{\g}_0 \oplus \tilde{\g}_{-1})} S^{\pmb{\lam}^{\vee}} V_{\bar{1}, \infty} \twoheadrightarrow L_{\infty}(\lam)$$ as in \cite[Section 4.5]{SerInf}. As a $\tilde{\g}_0$-map, this can be written as
$$\psi:\Lambda (\tilde{\g}_1) \otimes S^{\pmb{\lam}^{\vee}} V_{\bar{1}, \infty} \twoheadrightarrow L_{\infty}(\lam), \;\;\; \forall x \in \tilde{\g}_1, \, v\in S^{\pmb{\lam}^{\vee}} V_{\bar{1}} , \;\; x\otimes v \mapsto x.v$$

Furthermore, we have: $$L_{\infty}(\lam) \subset V_{\infty}^{\otimes k}$$ and so $$\Lambda^{> k} (\tilde{\g}_1) \otimes S^{\pmb{\lam}^{\vee}} V_{\bar{1}, \infty} \subset \Ker(\psi)$$

This implies that the map $\psi$ restricts to the surjective map of modules in $\Rep(\tilde{g}_{0})$:
$$\psi:\Lambda^{\leq k} (\tilde{\g}_1) \otimes S^{\pmb{\lam}^{\vee}} V_{\bar{1}, \infty} \twoheadrightarrow L_{\infty}(\lam), \;\;\; \forall x \in \tilde{\g}_1, \, v\in S^{\pmb{\lam}^{\vee}} V_{\bar{1}, \infty} , \;\; x\otimes v \mapsto x.v$$

When we take $\p(n)^{\perp}$-invariants in the above picture, we obtain a surjective map
$$\Lambda^{\leq k} ({\g}_1) \otimes S^{\pmb{\lam}^{\vee}} V_{n,\bar{1}} \twoheadrightarrow L_{\infty}(\lam)^{\p(n)^{\perp}}=\Phi_n(L_{\infty}(\lam)), \;\;\; \forall x \in {\g}_1, \, v\in S^{\pmb{\lam}^{\vee}} V_{n,\bar{1}} , \;\; x\otimes v \mapsto x.v$$

Hence the map $f:\Delta_n(\lam) \to\Phi_n(L_{\infty}(\lam))$ is surjective, and thus $\overline{\Delta}_n^k(\lambda) \cong\Phi_n( L_{\infty}(\lam))$, as required.}
\end{proof}

Similarly, we have:
\begin{proposition}\label{prop:Phi_dual_simples_to_nabla}
 Let $n>2k \geq 0$, and let $\lam$
be a $k$-admissible weight for $\p(n)$.
Then $$\Phi_n(L_{\infty}(\lambda))^* \cong \Pi^{\abs{\lam}}
\overline{\nabla}_n^k(\lambda^{\vee})$$ with the weight $\lambda^{\vee}$ interpreted
as in \InnaC{Section \ref{sssec:notn:weight_p_n}}.
\end{proposition}
\begin{proof}
\InnaC{Clearly, it is enough to prove the statement for $k:= \abs{\lam}$.

Recall that the module $\Phi_n(L_{\infty}(\lambda))^*$ is a quotient of $\left(\Pi^k S^{\pmb{\lam}^{\vee}}V_n \right)^* \cong S^{\pmb{\lam}}V_n$.

Let $f: S^{\pmb{\lam}}V_n \to \Pi^k \nabla_n(\lam^{\vee})$ be the map described in Section \ref{ssec:Rep_k_standard}.

For any $s< k$, we have: $$ \Hom_{\p(n)}(  V_{n}^{\otimes s}, \Pi^k \nabla_n(\lam^{\vee}) ) =  \Hom_{\g_{0} \oplus \g_{1}}( Res^{\p(n)}_{\g_{0} \oplus \g_{1}} V_{n}^{\otimes s}, S^{\pmb{\lam}} V_{\bar{0},n}^*)  =0$$
Hence $\psi^* \circ f =0$ for any $ \psi \in \Hom_{\p(n)}( V_{n}^{\otimes k}, V_{n}^{\otimes s})$, and thus we have a map $$ \Phi_n( L_{\infty}(\lam))^* \twoheadrightarrow \Pi^k \overline{\nabla}_n^k(\lambda) \hookrightarrow \Pi^k \nabla_n(\lam^{\vee}).$$

Now, to prove the required statement, we only need to show that $\Phi_n( L_{\infty}(\lam))^*$ is generated by a lowest weight vector with respect to the usual Borel in $\g_0 \oplus \g_{1}$. In other words, we would like to show that $\Phi_n( L_{\infty}(\lam))$ is a highest weight module with respect to the same Borel. This is a direct consequence of \cite[Section 4.5]{SerInf}, inferred as in the proof of Proposition \ref{prop:Phi_simples_to_delta}. }
\end{proof}

\subsection{Further properties}
\InnaC{We now give some immediate corollaries of Propositions \ref{prop:Phi_simples_to_delta} and \ref{prop:Phi_dual_simples_to_nabla}.}
\begin{corollary}\label{cor:Delta_nabla_dual}
 Let $n>2k \geq 0$, and let $\lam$
be a $k$-admissible weight for $\p(n)$. Then
$$\overline{\Delta}^k_n(\lam)^*\cong \Pi^{\abs{\lam}}
\overline{\nabla}_n^k(\lambda^{\vee}).$$
\end{corollary}

\begin{corollary}\label{cor:ses}
Let $n >2k$, and $\lambda$ be a $k$-admissible weight
such $\abs{\lambda} \leq k-1$. Then we have short exact sequences
 $$ 0 \to \bigoplus_{\substack{\mu \in \{\lambda - \eps_i| i\geq 1\}, \\ \mu \;
\text{ is } k\text{-admissible} } } \Pi \overline{\Delta}^k_n(\mu)
\longrightarrow \overline{\Delta}^k_n(\lambda) \otimes V_n \longrightarrow
\bigoplus_{\substack{\mu \in \{\lambda + \eps_i| i\geq 1\}, \\ \mu \; \text{ is
} k\text{-admissible} }} \overline{\Delta}^k_n(\mu) \to 0.$$
\end{corollary}
\begin{proof}
 This is a direct consequence of Lemma \ref{lem:ses_infty} together with the
fact that $\Phi_n$ is a SM functor and is exact on $\Rep^k(\p(\infty))$.
\end{proof}

Similarly, we have:
\begin{corollary}\label{cor:ses_nabla}
Let $n>2k+2$, $k\geq 0$, and $\lambda$ be a $k$-admissible
weight
such $\abs{\lambda} \leq k-1$. Then we have short exact sequences
 $$ 0 \to \bigoplus_{\substack{\mu \in \{\lambda + \eps_i| i\geq 1\}, \\ \mu \;
\text{ is } k\text{-admissible} } } \overline{\nabla}^k_n(\mu)
\longrightarrow \overline{\nabla}^k_n(\lambda) \otimes V_n \longrightarrow
\bigoplus_{\substack{\mu \in \{\lambda - \eps_i| i\geq 1\}, \\ \mu \; \text{ is
} k\text{-admissible} }}  \overline{\nabla}^k_n(\mu) \to 0.$$
\end{corollary}

The following statement is the analogue of \cite[Lemma 6.2.4]{EHS}:
\begin{proposition}\label{prop:simples_delta_k_n}
 Let $n> 2k+2$, $k\geq 0$, and $\lambda$ be a $k$-admissible
weight,
and $\mu \geq 0$. Then
 $$[\Delta_n(\lambda):L_n(\mu)] = [\overline{\Delta}^k_n(\lambda):L_n(\mu)].$$
\end{proposition}
\begin{proof}
For any module $M \in \Rep(\p(n))$, let $[M]$ denote the corresponding element
of the \InnaC{reduced}\footnote{\InnaC{The reduced Grothendieck group (resp. ring) of an $\sVect$ category is the quotient of the usual Grothendieck group (resp. ring) by the subspace (resp. ideal) generated by $[X] - [\Pi X]$ for every object $[X]$. }} Grothendieck group of $\Rep(\p(n))$. We wish to prove that
$$[\Delta_n(\lambda)]-[\overline{\Delta}^k_n(\lambda)] \in
span_{\Z_+}\{[L_n(\nu)]
\rvert \nu \not\geq 0\}.$$

We prove this statement by induction on $\abs{\lambda}$.
The base of the induction is $\lambda = 0$, in which case
$\overline{\Delta}^k_n(\lambda) =\C$. Furthermore, by \cite[Theorem
6.3.1]{BDE+}
all composition factors $L_n(\nu)$ of $\Delta_n(0)$ satisfy $\nu \not\geq
0$ except $\nu=0$; this implies the required statement for $\lam=0$.

We now proceed to the step of induction. Assume the statement is true for all
$\lambda$ such that $\abs{\lambda} =s$. Then for each such $\lambda$, we have:
$$[\Delta_n(\lambda) \otimes V_n]-[\overline{\Delta}^k_n(\lambda) \otimes V_n]
\in
span_{\Z_+}\{[L_n(\nu)\otimes V_n]
\rvert \nu \not\geq 0\}.$$

\begin{sublemma}
 For any $\tau \not\geq 0$, we have $$[L_n(\tau)\otimes V_n]
 \in span_{\Z_+}\{[L_n(\nu)]
\rvert \nu \not\geq 0\}.$$
\end{sublemma}
\begin{proof}
Assume we have $[L_n(\tau)\otimes V_n: L_n(\nu)] >0$ for some
$\nu\geq 0$. Then $$\dim\Hom(P_n(\nu)\otimes V_n, L_n(\tau)) =
\dim\Hom(P_n(\nu), L_n(\tau)\otimes V_n) >0$$ and hence $P_n(\nu)\otimes
V_n$ has a direct summand $ P_n(\tau)$. But the translation rules
for indecomposable projectives \cite[Section 7]{BDE+} show that for any direct
summand $P_n(\tau')$ of $P_n(\nu)\otimes V_n$, we would have: $\tau'
\geq \nu$ and hence $\tau'>0$. Contradiction.
\end{proof}

Hence $$[\Delta_n(\lambda) \otimes V_n]-[\overline{\Delta}^k_n(\lambda) \otimes
V_n]
\in
span_{\Z_+}\{[L_n(\nu)]
\rvert \nu \not\geq 0\}.$$

Next, recall from \cite[Section 4.3]{BDE+} that we have short exact sequences
$$ 0
\to \bigoplus_{\mu \in \{\lambda - \eps_i| i\geq 1\} } \Pi {\Delta}_n(\mu)
\longrightarrow {\Delta}_n(\lambda) \otimes V_n \longrightarrow
\bigoplus_{\mu \in \{\lambda + \eps_i| i\geq 1\}} {\Delta}_n(\mu) \to 0.$$
Using these, as well as Corollary \ref{cor:ses}, we conclude that
$$[\Delta_n(\lambda) \otimes V_n]-[\overline{\Delta}^k_n(\lambda) \otimes V_n]
=
\sum_{\mu \in \{\lambda \pm \eps_i| i\geq 1\}}
[\Delta_n(\mu)]-[\overline{\Delta}^k_n(\mu)]
$$
where we use the convention $\overline{\Delta}^k_n(\mu) =0$ for
non-$k$-admissible
$\mu$.

The
statement now follows for any $\mu \in \{\lambda \pm \eps_i| i\geq 1\}$;
since we started with an arbitrary $k$-admissible $\lambda$ such
that $\abs{\lambda}=s$, the above argument shows that the statement holds for
any $\mu$ such that
$\abs{\mu}=s+1$. This completes the proof of the lemma.

\end{proof}

\begin{corollary}\label{cor:delta_k_n_max_quo}
Let $ n>2k+2$, $k\geq 0$, and $\lambda$ be a $k$-admissible weight.
Then $\overline{\Delta}^k_n(\lambda)$ is the maximal quotient of
$\Delta_n(\lambda)$ belonging to $\Rep^k(\p(n))$.
\end{corollary}
An analogous statement holds for $\overline{\nabla}^k_n$ as well.

Finally, we compute the extensions between $\overline{\Delta}^k_n(\lambda)$
and $\overline{\nabla}^k_n(\mu)$. This will be used later in Section
\ref{ssec:hw}.

\begin{lemma}\label{lem:ext}
Let $ n>2k+2$, $k\geq 0$, and $\lambda$, $\mu$ be $k$-admissible weights.

Then $$\dim \Hom_{\p(n)}
\left(\overline{\Delta}^k_n(\lambda), \overline{\nabla}^k_n(\mu) \right) =
\delta_{\lambda, \mu}$$ and $$\Ext^1_{\p(n)}
\left(\overline{\Delta}^k_n(\lambda), \overline{\nabla}^k_n(\mu) \right) = 0.
$$
\end{lemma}
\begin{proof}
 Let $$K:= \Ker\left(\Delta_n(\lambda) \twoheadrightarrow
\overline{\Delta}^k_n(\lambda)\right) \;\;\;C :=
\operatorname{Coker}\left(\overline{\nabla}^k_n(\mu) \hookrightarrow
\nabla^k_n(\mu)\right).$$

We have long exact sequences
\begin{align*}
 0 \to &\Hom_{\p(n)}\left(\overline{\Delta}^k_n(\lambda),
\overline{\nabla}^k_n(\mu) \right) \to
&\Hom_{\p(n)}\left({\Delta}_n(\lambda),
\overline{\nabla}^k_n(\mu) \right) \to
&\Hom_{\p(n)}\left(K,
\overline{\nabla}^k_n(\mu) \right) \to \\
\to &\Ext^1_{\p(n)}
\left(\overline{\Delta}^k_n(\lambda), \overline{\nabla}^k_n(\mu) \right)  \to
&\Ext^1_{\p(n)}
\left(\overline{\Delta}^k_n(\lambda), \overline{\nabla}^k_n(\mu) \right) \to
&\Ext^1_{\p(n)}
\left(K, \overline{\nabla}^k_n(\mu) \right) \to \ldots
\end{align*}

and
\begin{align*}
 0 \to
&\Hom_{\p(n)}\left({\Delta}_n(\lambda),
\overline{\nabla}^k_n(\mu) \right) \to
&\Hom_{\p(n)}\left({\Delta}_n(\lambda),
{\nabla}_n(\mu) \right) \to
&\Hom_{\p(n)}\left({\Delta}_n(\lambda),
C \right)
\\
\to &\Ext^1_{\p(n)}
\left({\Delta}_n(\lambda),
\overline{\nabla}^k_n(\mu)\right)  \to
&\Ext^1_{\p(n)}
\left({\Delta}_n(\lambda),{\nabla}_n(\mu) \right) \to
&\Ext^1_{\p(n)}
\left({\Delta}_n(\lambda), C \right) \to \ldots
\end{align*}

First, we use the fact that $\Hom_{\p(n)}\left({\Delta}_n(\lambda),
{\nabla}_n(\mu) \right) = \delta_{\lambda, \mu}$ with the
image of ${\Delta}_n(\lambda) \to {\nabla}_n(\lam)$ being $L_n(\lam)$.

Secondly, Corollary \ref{cor:delta_k_n_max_quo} implies that $(C: L_n(\lam)) =
(K:L_n(\mu))=0$ so $\Hom_{\p(n)}\left({\Delta}_n(\lambda),
C \right) =0$, $\Hom_{\p(n)}\left(K,
\overline{\nabla}^k_n(\mu) \right)=0$.

Substituting these into the long exact sequences above, we conclude that $$\dim
\Hom_{\p(n)}
\left(\overline{\Delta}^k_n(\lambda), \overline{\nabla}^k_n(\mu) \right) =
\delta_{\lambda, \mu}.$$

Next, we recall that $$\Ext^1_{\p(n)}
\left({\Delta}_n(\lambda), {\nabla}_n(\mu) \right) = 0.
$$
The second long exact sequence then implies that $\Ext^1_{\p(n)}
\left({\Delta}_n(\lambda),
\overline{\nabla}^k_n(\mu)\right)=0$ and hence $$\Ext^1_{\p(n)}
\left(\overline{\Delta}^k_n(\lambda), \overline{\nabla}^k_n(\mu) \right) = 0.
$$
\end{proof}

\section[The DS functor]{The \texorpdfstring{$DS$}{DS}
functor}\label{sec:DS_functor}
\subsection{Definition and basic properties}

Let $p,q>0$. \InnaC{We continue with the notation of Section \ref{ssec:restr_functors}.}
Let $y\in \gl(p|q)$ be an odd element of rank $1$ such that $y(W_{p,q})\subset
(W_{p,q})_0$ and $x=(y,\Pi y^*) \in \g_{p,q}\subset\p(n)$ be the corresponding
odd element.

Let $M \in \Rep(\p(n))$. We define
$$ {DS}_x(M) = \quotient{Ker(x \rvert_M)}{Im(x \rvert_M)}, \;\;
\text{ and } \;\; \mathrm{DS}_y(M) = \quotient{Ker(y \rvert_M)}{Im(y
\rvert_M)}.$$

\begin{lemma}
 We have an isomorphism of Lie superalgebras $DS_x(\p(n))
\cong \p(n-2)$ and $\mathrm{DS}_y(\gl(p|q)) \cong \gl(p-1|q-1)$.
\end{lemma}
\begin{proof}
The first statement follows from the fact that $DS_x(M \otimes N) \cong DS_x(M)
\otimes DS_x(N)$ as supervector spaces \InnaA{(see \cite{Sersd})}. This means that
$DS_x$ defines a SM $\sVect$-functor $DS_x: \Rep(\p(n)) \to \sVect$; one
immediately
sees that $DS_x(V_n)$ is an $(n-2|n-2)$ dimensional superspace with an
(induced)
odd symmetric form, and hence $DS_x(\p(n))$ is a superalgebra preserving that
form.

To prove that it is isomorphic to $\p(n-2)$, we use the fact that $\p(n) \cong
\bigwedge^2 V_n$, and hence $$DS_x(\p(n)) \cong \bigwedge^2 \C^{n-2|n-2} \cong
\p(n-2)$$ as supervector spaces.

The second statement can be proved analogously (see also \cite{DS}).

\end{proof}

Hence we obtain functors $$DS_x: \Rep(\p(n)) \to \Rep(\p(n-2))$$
and $$\mathrm{DS}_y: \Rep(\gl(p|q)) \to \Rep(\gl(p-1|q-1)).$$
These functors are \InnaA{again} SM $\sVect$-functors.

\begin{lemma}\label{lem:DS_Res_commute}
 The functors $DS_x$ commute with the functors $Res$; that is, we have a
natural isomorphism:
$$\xymatrix{&\Rep(\p(n)) \ar[rr]^-{Res} \ar[d]^{DS_x} &{}
&\Rep(\gl(p|q))
\ar[d]^{\mathrm{DS}_y}\\ &\Rep(\p(n-2)) \ar@{=>}[rru] \ar[rr]^-{Res}
&{} &\Rep(\gl(p-1|q-1)) }$$
\end{lemma}
\begin{proof}
 This follows directly from the fact that $y$ goes to $x$ under the inclusion
$\gl(p|q) \subset
\p(n)$.
\end{proof}

Finally, we state the following result. Although it is straightforward, will be
very useful later on.

\begin{lemma}\label{lem:DS_act}
 Let $\g$ be $\p(n)$ or $\gl(p|q)$, $z$ be $x$, $y$ as above, respectively, and
$DS_z$ be $DS_x$ or $\mathrm{DS}_y$ respectively. Consider the adjoint
representation of $\g$ (which we denote by $\g$ as well). This is a Lie algebra
object in $\Rep(\g)$, and any $M\in \Rep(\g)$ has a natural module structure over
it (in $\Rep(\g)$), given by a natural transformation $act^{\g}: \g \otimes -
\to
\id$.

 Let $\g_z:= DS_z(\g_z)$.

 Then $DS_z(act^{\g}) \cong act^{\g_z}$, where the latter is the natural action
of the adjoint representation of $\g_z$ on any $\g_z$-module.
\end{lemma}
\comment{\begin{proof}
 Let $M \in \Rep(\g)$ and denote $\rho:=act^{\g}_M$. Consider the construction
of
$DS_z(\rho)$: the map $\rho:\g \otimes M \to M$ induces a map $$\rho:\Ker(\ad z
)\otimes \Ker z\rvert_M \to \Ker z\rvert_M \twoheadrightarrow
\quotient{\Ker\rho(z)_M}{\Ker\rho(z)_M} \cong DS_z(M)$$ which then gives the
map
$$DS_z(\rho):DS_z(\g) \otimes DS_z(M)\to DS_z(M).$$ Yet the action of $\gg_z$
on
$DS_z(M)$ is clearly defined in the same way! Hence the statement of the lemma
follows.
\end{proof}}
\begin{remark}
\VeraC{Clearly, the restriction of the functor $Res$ to the full
subcategory $\Rep^k(\p(n)) \subset \Rep(\p(n))$ has image
in the full subcategory $ \Rep^k\gl(p|q)$.}
\end{remark}

\subsection{Main result on \texorpdfstring{$DS$}{DS} functor}
We now prove a central result on DS functors for periplectic Lie superalgebras.
Now, fix $k\geq 0$.
\begin{proposition}\label{prop:DS_equiv}
 For $n \geq 8k+2$, the functor $DS_x: \Rep^k(\p(n)) \to
Rep^k(\p(n-2))$ is an
equivalence.
\end{proposition}
\begin{proof}
Let $p, q \geq 0$ be such that $n= p+q$ and $\min(p, q) > 4k$; by
\cite[Theorem 7.1.1]{EHS} the
functor $$\mathrm{DS}_y: \Rep^k(\gl((p|q)) \longrightarrow
Rep^k(\gl(p-1|q-1))$$ becomes an equivalence of categories (in
particular, exact and faithful).

By Lemma \ref{lem:DS_Res_commute}, the functor $DS_x: \Rep^k(\p(n)) \to
Rep^k(\p(n-2))$ is exact and faithful as well, so we only need to show that it
is full and essentially surjective. Below we describe the main steps of the
proof, with details given in Lemmas \ref{lem:DS_on_objects},
\ref{lem:DS_socle_mult}.

 {\bf Proof that $DS_x$ is full:}

 In Lemma \ref{lem:DS_on_objects}, we show
that for any simple $L$,
$DS_x(L)$ is simple. We use this fact to prove Lemma \ref{lem:DS_socle_mult}:
for any $M \in \Rep^k(\p(n))$ we have $$[soc(M):\C]=[soc(DS_x(M)):\C].$$
Clearly, this implies
 \begin{equation}\label{eq:DS_full}
 \dim \Hom_{\p(n)}(\C, M) =  \dim \Hom_{\p(n-2)}(\C, DS_x M)
 \end{equation}

 Next, let $M, N \in \Rep^k(\p(n))$. We have isomorphisms
 \begin{align*}
&\Hom_{\p(n)}(N, M) \cong  \Hom_{\p(n)}(\C, N^* \otimes M), \;\; \text{  and
}\\
&\Hom_{\p(n-2)}(DS_x N, DS_x M) \cong \Hom_{\p(n-2)}(\C, (DS_x N)^* \otimes
DS_x M)\cong \Hom_{\p(n-2)}(\C, DS_x (N^* \otimes M))
 \end{align*}

 Hence \eqref{eq:DS_full} implies: $$\dim \Hom_{\p(n)}(N, M) =  \dim
\Hom_{\p(n-2)}(DS_x N, DS_x M)$$ which means that $DS_x$ is full.

 {\bf Proof that $DS_x$ is essentially surjective:}
 \mbox{}

We wish to show that for any $M'  \in \Rep^k(\p(n-2))$ there exists $M \in
Rep^k(\p(n))$ such that $DS_x M \cong M'$.

 We prove this by induction on the length of $M'$.

 For $\ell(M')=0$, this is trivial. Assume now that the statement holds for any
module in $\Rep^k(\p(n-2))$ of length strictly less that $\ell(M')$.

 Let $L'$ be a simple submodule of $M'$:
 $$ 0 \to L' \longrightarrow M' \longrightarrow N' \to 0$$ By Lemma
\ref{lem:DS_on_objects}, there exists a simple module $L \in \Rep^k(\p(n-2))$
such that $DS_x(L) = L'$. By induction assumption, there exists $N\in
Rep^k(\p(n-2))$ such that $DS_x(N) = N'$.

 Now we recall that any module in $\Rep^k(\p(n-2))$ is a subquotient of a \InnaC{certain finite direct sum $T'$
of tensor powers $V_{n-2}^{\otimes s}$, $0 \leq s \leq k$. Let $T'$ be the appropriate direct sum for $L'$.

It is easy to see that $T'$ lies in the essential image
of $DS_x$. Indeed, let us write $$T' = \bigoplus_{0 \leq s \leq k} \left(V_{n}^{\otimes s}\right)^{\oplus m_s}$$ for some $m_1, m_2, \ldots, m_k \in \Z_{\geq 0}$, and set $$T :=  \bigoplus_{0 \leq s \leq k} \left(V_{n}^{\otimes s}\right)^{\oplus m_s}$$ Using the fact that $DS_x$ respects direct sums and tensor products, as well as $DS_x(V_n) \cong V_{n-2}$, we obtain:
 $DS_x \left(T \right) \cong T$.}

 Denote $Q' := \quotient{T'}{L'}$:
  $$\xymatrix{&0 \ar[r] &L' \ar@{=}[d] \ar[r] &M' \ar@{^{(}->}[d] \ar[r] &N'
\ar@{^{(}->}[d] \ar[r] &0 \\ &0\ar[r]  &L' \ar[r] &T' \ar[r] &Q' \ar[r] &0}$$

  Notice that $N'  \twoheadleftarrow M' \hookrightarrow T'$ is the pullback of
the arrows $ N' \hookrightarrow Q' \twoheadleftarrow T'$.

 Since $DS_x$ is fully faithful and exact, there is an injective map $L
\hookrightarrow T$. Denote the cokernel of this map by $Q$. Then $DS_x(Q) = Q'$
and there is an injective map $N \hookrightarrow Q$, such that the diagram
below is taken by $DS_x$ to the corresponding maps in the diagram above:

   $$\xymatrix{&{} &L \ar@{=}[d]  &{}  &N \ar@{^{(}->}[d] &{}  \\ &0\ar[r]  &L
\ar[r] &T \ar[r] &Q \ar[r] &0}$$

   Consider the pullback $N  \twoheadleftarrow M \hookrightarrow T$ of the
arrows $ N \hookrightarrow Q \twoheadleftarrow T$. Since $DS_x$ is exact, it
preserves pullbacks, and hence $DS_x(M) = M'$, as required.

This completes the proof of Proposition \ref{prop:DS_equiv}.
\end{proof}
\begin{remark}
 The fact that $DS_x$ is full can also be concluded from the result of
\cite{DLZ} (see Proposition \ref{prop:contractions}).
\end{remark}

\begin{lemma}\label{lem:DS_on_objects}
Let $\abs{\lambda} \leq k$ and $n>8k+2$, $k\geq 0$. Consider
the
simple module $L_n(\lambda)$, and the highest weight module
$\overline{\Delta}^k_n(\lambda)$ in $\Rep^k(\p(n))$. We have:
\begin{enumerate}
 \item $DS_x \overline{\Delta}^k_n(\lambda) =
\overline{\Delta}_{n-2}^k(\lambda)$.
 \item $DS_x L_n(\lambda) = L_{n-2}(\lambda)$.
 \end{enumerate}
\end{lemma}
\begin{proof}
\begin{enumerate}
 \item We prove the statement by induction on $\abs{\lambda}$.

For $\abs{\lambda} =0$, we have $\lambda = 0$ and $\overline{\Delta}_n^k(0) =
\Phi_n(\triv)$, $\overline{\Delta}_{n-2}^k(0) = \Phi_{n-2}(\triv)$. Hence the
statement holds in this case.

The induction step follows directly from Lemma \ref{cor:ses}.

\item Consider the surjective map $f: \overline{\Delta}^k_n(\lambda)
\twoheadrightarrow L_n(\lambda)$. Since $DS_x$ is exact, we have: $DS_x
L_n(\lambda)$ is the quotient of
$DS_x \overline{\Delta}^k_n(\lambda) = \overline{\Delta}_{n-2}^k(\lambda)$, and
hence
$DS_x L_n(\lambda)$ has $L_{n-2}(\lambda)$ as quotient.

Moreover, $[DS_x L_n(\lambda): L_{n-2}(\lambda)] =1$, since $$[DS_x
L_n(\lambda): L_{n-2}(\lambda)] \leq [\overline{\Delta}_{n-2}^k(\lambda):
L_{n-2}(\lambda)] \leq [\Delta_{n-2}(\lambda): L_{n-2}(\lambda)] =1$$

On the other hand, we have the transposed map $$ f^t: L_n(\lambda)^*
\hookrightarrow \left(\overline{\Delta}^k_n(\lambda) \right)^*$$

Denote by $\lambda^{\vee}$ the highest weight of the simple module $
L_n(\lambda)^*$. Clearly, $L_n(\lambda^{\vee}) = L_n(\lambda)^*$ sits in
$\Rep^k(\p(n))$ as well.

Then $$L_n(\lambda) \subset \left(\overline{\Delta}^k_n(\lambda^{\vee})
\right)^*$$
and hence
$$DS_x L_n(\lambda) \subset DS_x \left(\overline{\Delta}^k_n(\lambda^{\vee})
\right)^*
= \left(DS_x \overline{\Delta}^k_n(\lambda^{\vee}) \right)^* =
\left(\overline{\Delta}^k_{n-2}(\lambda^{\vee}) \right)^*$$

In particular, the socle of $DS_x L_n(\lambda)$ coincides with the socle of
$\left(\overline{\Delta}^k_{n-2}(\lambda^{\vee}) \right)^*$, which is
$L_{n-2}(\lambda)$. Hence $DS_x L_n(\lambda) = L_{n-2}(\lambda)$, as required.

\end{enumerate}
\end{proof}

\begin{lemma}\label{lem:DS_socle_mult}
 Let $n>8k+2$, $k\geq 0$ and $M \in \Rep^k(\p(n))$. Then
$[soc(M):\C]=[soc(DS_x(M)):\C]$.
\end{lemma}
\begin{proof}
 Recall that $DS_x(\C) = \C$. Moreover, by Lemma \ref{lem:DS_on_objects},
$DS_x$ takes simple modules to simple modules, and induces a bijection on the
set of isomorphism classes of simple modules in $\Rep^k(\p(n))$,
$\Rep^k(\p(n-2))$. Hence in order to prove the statement of the lemma, it is
enough to show that given a non-split short exact sequence
 $$ 0 \to \C \to M \to N \to 0$$
 its image $$0 \to \C \to DS_x M \to DS_x N \to 0$$ will not split either.

 Assume the contrary, and consider a module $M$ of minimal length for which a
non-split short exact sequence
 $$ 0 \to \C \to M \to N \to 0$$ is sent by $DS_x$ to a split sequence $$0 \to
\C \to DS_x M \to DS_x N \to 0$$

 Let $L$ be a simple quotient of $N$, and denote: $$M' := Ker(M
\twoheadrightarrow L), \;\; N' := Ker(N \twoheadrightarrow L)$$

 Then we have a short exact sequence
 \begin{equation}\label{eq:split_ses}
  0 \to \C \to M' \to N' \to 0
 \end{equation}
 is sent by $DS_x$ to a split sequence $$0 \to \C \to DS_x M' \to DS_x N' \to
0$$

 By assumption, the sequence \eqref{eq:split_ses} is split, since $\ell(M') <
\ell(M)$. Hence we have a section $M' \twoheadrightarrow \C$.

 Consider the pushout of the maps $M' \to \C$, $M' \to M$:
 $$\xymatrix{&M \ar@{-->>}[r] &K \\ &{M'} \ar@{^{(}->}[u] \ar@{->>}[r]
&{\C}\ar@{^{(}-->}[u]}$$

 Then $$0 \to \C \hookrightarrow K \twoheadrightarrow L \to 0$$ is a short
exact sequence, whose image under $DS_x$ is clearly split. We now consider two
possibilities:

 \begin{itemize}
  \item[Case $N \neq L$.] In this case $N' \neq 0$ and hence $\ell(K) <
\ell(M)$. This implies that we have a section $K \twoheadrightarrow \C$,
inducing a section $M \twoheadrightarrow \C$, which in turn contradicts the
assumption that the original short exact sequence was not split.
  \item[Case $N = L$.] In this case, we have a non-split extension  $$ 0 \to \C
\to M \to L \to 0$$ and we claim that $M$ is a highest weight
module whose highest weight coincides with that of $L$.

  Indeed, denote by $\lambda$ the highest weight of $L$. Consider the
weights of the module $M$. These are the weights of $L$ together with the
trivial weight. Since $\lambda$ is $k$-admissible, we have: $\lambda_i \leq 0$
for any $i$. Hence weight $0$ is not higher than $\lambda$, and thus
$M$ has no weights higher than $\lambda$. This means that we
have a map $\Delta(\lambda) \to M$. Since $M$ is
not
split, this map is surjective, and we conclude that $[\Delta(\lambda) : \C] >
0$. Now \cite[Theorem 6.3.3]{BDE+} implies that there exists $i$ such that
$\lambda_i \geq 0$, which contradicts our
conditions on $\lambda$ such that $L_n(\lambda) \in \Rep^k(\p(n))$.

 \end{itemize}
This completes the proof of Lemma \ref{lem:DS_socle_mult}.
\end{proof}
\subsection{Compatability with the specialization
functors}\label{ssec:DS_Phi_compat}

Let $n \geq 3$. Consider the functors $\Phi_n: \Rep(\p(\infty)) \to
\Rep(\p(n)) $ and $\Phi_{n-2}: \Rep(\p(\infty)) \to
\Rep(\p(n-2)) $.

Let $x \in \p(n)_{\InnaC{\bar{1}}}$ be such that $x$ lies in the centralizer of $\p(n-2)$. Then
for any $M \in \Rep(\p(\infty))$ we have:
$$M^{\p(n-2)^{\perp}} \subset Ker \, x\rvert_{M^{\p(n)}}$$ and hence we can
consider the
composition
$$M^{\p(n-2)^{\perp}} \hookrightarrow Ker  \, x\rvert_{M^{\p(n)}}
\twoheadrightarrow
DS_x\left(M^{\p(n)}\right)$$ which we denote by $\eta_M$. This gives a natural
$\otimes$-transformation $\eta: \Phi_{n-2} \to DS_x \circ \Phi_n$.
\begin{remark}
 Applying the functor $Res: \Rep(\p(n-2)) \to \Rep(\gl(p-1|q-1))$ for $p+q =n$
to $\eta$, we obtain a natural $\otimes$-transformation $\Gamma_{p-1|q-1} \to
\mathrm{DS}_y\Gamma_{p|q}$ which is described in \cite[Section 7.2]{EHS}.
\end{remark}

\begin{lemma}\label{lem:DS_Phi_compat}
 Let $n > 4k+2$. The restriction of $\eta$ to
$\Rep^k(\p(\infty))$ is an isomorphism.
\end{lemma}
\begin{proof}
 As it is mentioned in the Remark above, $\Gamma_{p-1|q-1} \to
\mathrm{DS}_y\Gamma_{p|q}$ is the natural transformation described in
\cite[Section
7.2]{EHS}. In particular, by \cite[Lemma 7.2.1]{EHS}, the restriction of
$Res(\eta)$ to $\Rep^k(\gl(p-1|q-1))$ for $p,q>2k$ is an isomorphism. Since
$Res$
is a faithful exact functor, we conclude that $\eta$ is both surjective and
injective, and hence an isomorphism as well.
\end{proof}

\section{Construction of the category \texorpdfstring{$\Rep(P)$}{Rep(P)}}\label{sec:Rep_p}

\subsection{Construction}\label{ssec:constr_urep}
For each $n\geq 4,n \in 2\Z$, fix $x_n \in \p(n)$ an odd element of rank
$2$ (as in Section \ref{sec:DS_functor}). Recall that by Proposition
\ref{prop:DS_equiv}, the functors
$$DS_{x_n}: \Rep^k(\p(n)) \to
Rep^k(\p(n-2))$$ are equivalences of SM categories for $n>>k$. Furthermore,
for each $k, n \geq 1$ we have a fully faithful exact SM embedding
$\Rep^k(\p(n))
\to Rep^{k+1}(\p(n))$. This allows
us to define
$$\urepk{k} := \varprojlim_{n \in 2\Z, n \to \infty}
Rep^k(\p(n)), \;\;\;\; \urep := \varinjlim_{k \to
\infty}\urepk{k}$$ with respect to the functors $DS_{x_n}$  and the inclusions
above.

The category $\urep$ will be called the {\it Deligne category for the
periplectic Lie superalgebra}.

The category $\urep$ is clearly a $\sVect$-category. Furthermore,
we have:
\begin{lemma}
 The category $\urep$ is a tensor $\sVect$-category.
\end{lemma}
\begin{proof}
 The categories $\Rep^k(\p(n))$ are abelian for any $k$, so
$\urepk{k}$ is abelian as well, and so is $\urep$.

For any $n, k, m$, the tensor structure on $\Rep(\p(n))$ induces a biexact
bilinear bifunctor $$\Rep^k(\p(n)) \times Rep^m(\p(n)) \longrightarrow
Rep^{k+m}(\p(n)).$$ This induces a bifunctor $$\urepk{k} \times \urepk{m}
\longrightarrow
\urepk{k+m}$$ for any $k, m\geq 0$, and hence a rigid SM structure on $\urep$,
which satisfies all the requirements of a tensor category structure.
\end{proof}

By definition,
$\urepk{k}$ comes equipped with SM functors
$$F_n: \urepk{k} \longrightarrow \Rep^k(\p(n))$$ for each $n \in 2\Z$
and $k$, which are equivalences when $n >8k+2$. These
induce SM functors
$$F_n: \urep \longrightarrow \Rep(\p(n))$$ for any $n \in 2\Z$.

We denote by $\bV$ the object of $\urep$ corresponding to the natural
representation: \InnaC{$$\bV = (V_n)_{n \in 2\Z} \in \urepk{1} \subset \urep.$$ Clearly, for any $n \in 2\Z$, we have $F_n(\bV) = V_n$.} The object $\bV$
is equipped with a pairing $\omega_{\bV}: \bV \otimes \bV \to \Pi \triv$.

\begin{remark}
 We can also define
 $${Rep}^k(\underline{P}') := \varprojlim_{n \in 2\Z+1, n \to \infty}
Rep^k(\p(n)), \;\;\;\; \Rep(\underline{P}') := \varinjlim_{k \to
\infty}{Rep}^k(\underline{P}')
$$ and set $\bV'$ to be the corresponding ``natural representation'' with a
pairing $\omega_{\bV'}$.

We will show in Corollary \ref{cor:two_urep_equiv} that these two constructions
are
equivalent, i.e. we have a (unique) equivalence of tensor categories $\urep
\cong  \Rep(\underline{P}')$ with $\bV$ sent to $\bV'$ and $\omega_{\bV}$ sent
to $\omega_{\bV'}$.

This allows us to consider $F_n: \urep \longrightarrow \Rep(\p(n))$ for any
$n \in \Z$. For this reason, from now on, we will impose any conditions on the
\InnaB{parity} of $n$, although such assumption would not influence the results below.
\end{remark}

Consider the morphism $$F_n: \Hom_{\urep}(\mathbf{V}^{\otimes l}, \triv)
\longrightarrow \Hom_{\p(n)}(F_n(\mathbf{V}^{\otimes l}), F_n(\triv)) =
\Hom_{\p(n)}(V_n^{\otimes l}, \C).$$

\begin{lemma}\label{lem:F_n_surjective}
 The morphism $F_n$ above is surjective.
\end{lemma}
\begin{proof}
By the results of \cite{DLZ} (also proved in
Corollary
\ref{appcor:contractions}), the spaces $\Hom_{\p(n)}(V_n^{\otimes l}, \C)$,
$\Hom_{\p(n-2)}(V_{n-2}^{\otimes l}, \C)$ are spanned by contraction maps for
any $n \geq 3$ and any $l\geq 0$. Hence the map
$$ DS_{x_n}:\Hom_{\p(n)}(V_n^{\otimes l}, \C) \to
\Hom_{\p(n-2)}(V_{n-2}^{\otimes l}, \C)$$ is surjective. The lemma now follows.
\end{proof}

\begin{corollary}\label{cor:F_n_full_tensor}
 The functor $F_n$ is full on the (full) Karoubian additive SM subcategory
generated by $\bV$.
\end{corollary}

\subsection{Translation functors and the
Casimir}\label{ssec:trans_functors_urep}
Recall the notation $act:\gl(\bV) \otimes
\bV
\to \bV$ which stands for the action of the Lie algebra object
$\gl(\bV)$ in $\urep$ on $\bV$ (see Section \ref{ssec:notn:SM}).
\begin{definition}\label{def:p_V}
 Let $$ \triangle:\gl(\bV) \otimes \bV \otimes
\bV \to \bV \otimes
\bV, \;\; \triangle
:= act \otimes
\id + (\id \otimes act ) \circ \InnaC{\left(\sigma_{\gl(\bV), \bV} \otimes \id \right)}$$ and let $\p(\bV) \in
\urep$ be the maximal subobject of $
\gl(\bV)$ such that the map $$\omega\circ \triangle:\p(\bV) \otimes \bV \otimes
\bV \to \bV \otimes
\bV \to \Pi \triv$$ is zero.
\end{definition}
That is, $\p(\bV)$ is the subobject of $\gl(\bV)$ preserving the form $\omega$.

\begin{lemma}\label{lem:Lie_alg_p_in_urep}

The following is true:
\begin{enumerate}
 \item\label{itm:pV_1}
The object $\p(\bV)$ is a Lie algebra subobject of $\gl(\bV)$ in
$\urep$.
 \item\label{itm:pV_2} For any $n \geq 1$, $F_n(\p(\bV))
\cong \p(n)$ (adjoint representation) as Lie algebra objects in $\Rep(\p(n))$.
 \item\label{itm:pV_3} We have an orthogonal decomposition
$$\gl(\bV) \cong\p(\bV) \oplus \p(\bV)^*$$ of $\urep$ objects with respect to
the form $$tr:=ev\circ \sigma_{\bV, \bV^*}: \gl(\bV) \cong \bV \otimes \bV^*
\to
\triv.$$
\item\label{itm:pV_5} We have an isomorphism: $\p(\bV)\cong\Pi S^2 \bV$.
\item\label{itm:pV_4} There is a
natural transformation of functors $act: \p(\bV) \otimes (-) \longrightarrow
\id$ making any $M \in \urep$ a module over the Lie algebra object $\p(\bV)$ in
$\urep$.
\end{enumerate}

\end{lemma}
\begin{proof}

The statement \eqref{itm:pV_1} is straightforward. For \eqref{itm:pV_2}, recall
that for $n>>0$, $F_n$ is both
a SM functor and an equivalence on $\urepk{4}$ to which all the objects
appearing
in Definition \ref{def:p_V} belong.
Hence $F_n(\p(\bV))$ would be the maximal Lie subalgebra of
$F_n(\gl(\bV))$ statisfying conditions as in Definition \ref{def:p_V}, i.e.
preserving the form $F_n(\omega): V_n \otimes V_n \to \Pi \triv$.

This implies
that $F_n(\p(\bV)) \cong \p(n)$ (as Lie algebras in $\Rep(\p(n))$). This proves
the statement for large $n$.
Now $DS_x(\p(n)) \cong \p(n-2)$ (as Lie algebras) for any $n\geq
3$, which proves \eqref{itm:pV_2} for any $n$.

The decomposition in \eqref{itm:pV_3} now follows from an analogous
decomposition for $\gl(n|n)$ and $\p(n)$.

To prove \eqref{itm:pV_5}, let $\eta:\Pi \bV \to \bV^*$ be the isomorphism
induced by the form $\omega$.

Consider the isomorphism $f: \Pi \bV^{\otimes 2} \to \gl(\bV) = \bV \otimes
\bV^*$ given by $$f = \sigma_{\bV, \bV^*} \circ (\eta \otimes \id) + (\id
\otimes \eta) \circ \sigma_{\Pi \triv, \bV}.$$
Now $f$ induces an isomorphism $\p(\bV)\cong\Pi S^2 \bV$, which can be verified
by applying $F_n$ to both sides, for large enough $n$, and obtaining $\p(n)
\cong \Pi S^2 V_n$, which is known to be true.

Finally, \eqref{itm:pV_4} follows from Lemma \ref{lem:DS_act}.

\end{proof}

We now define a natural endomorphism $\pmb{\Omega}$ of the endofunctor $\InnaA{\bV \otimes (-)}$ on $\urep$.

\begin{definition}
For any $M \in \urep$, let $\pmb{\Omega}_M$ be the composition
$$ \bV \otimes M \xrightarrow{\id \otimes coev \otimes \id} V \otimes \p(\bV)^*
\otimes \p(\bV) \otimes M \xrightarrow{i_* \otimes \id}  \bV \otimes \gl(\bV)
\otimes \p(\bV) \otimes M   \xrightarrow{\InnaA{(act \circ \sigma)} \otimes act}
\bV \otimes M $$ where $i_*: \p(\bV)^* \to \gl(\bV)$ is the embedding defined
in Lemma \ref{lem:Lie_alg_p_in_urep}, \eqref{itm:pV_3}.
\end{definition}
\begin{definition}
 For $k \in \C$, we define a functor
$\pmb{\Theta}'_k:\urep\to\urep$ as the functor $\InnaA{\bV \otimes (-)}$
followed by the projection onto the generalized $k$-eigenspace for
$\pmb{\Omega}$. \InnaA{That is, we set
\begin{eqnarray}
\label{thetak}
\pmb{\Theta}'_k(M):=\bigcup_{m>0}\Ker (\pmb{\Omega}
-k\operatorname{Id})^m_{|_{M\otimes \bV}}
\end{eqnarray}
The fact that $\pmb{\Theta}'_k(M)$ is a direct summand of $\bV \otimes M$, is proved in the same way as for operators on finite-dimensional vector spaces.

Finally, we} set $\pmb{\Theta}_k:=\Pi^k\pmb{\Theta}'_k$ in case $k\in\mathbb{Z}$.
\end{definition}

\begin{lemma}\label{lem:casimir_goes_to_casimir}
 Let $n \geq 1$. We have: $F_n(\pmb{\Omega}) = \Omega^n F_n$, where $\Omega^n$
is the tensor Casimir for $\p(n)$ (see \cite[4.1.1]{BDE+}). That is,
for any
$M \in \urep$, $F_n(\pmb{\Omega}_M) = \Omega^n_{F_n(M)}$ as endomorphisms of
$V_n \otimes F_n(M)$.
\end{lemma}
\begin{proof}
 This follows directly from the fact that $F_n$ is a SM
functor, together with \eqref{itm:pV_4}.
\end{proof}

\begin{corollary}\label{cor:transl_goes_to_transl}
\mbox{}
 \begin{enumerate}
  \item For any $k \notin \Z$, we have $\pmb{\Theta}'_k \cong 0$.
  \item We have a natural isomorphism of functors $$F_n \pmb{\Theta}_k \cong
\InnaA{\Theta^{(n)}_k} F_n$$ where $\InnaA{\Theta^{(n)}_k}: \Rep(\p(n)) \to \Rep(\p(n))$ is the
translation functor defined in \cite[4.1.7]{BDE+}.
 \end{enumerate}
\end{corollary}
\begin{proof}
 Follows directly from Lemma \ref{lem:casimir_goes_to_casimir} together with
the definition of $\InnaA{\Theta^{(n)}_k}$ as the generalized $k$-eigenspace of $\Omega^n$
on $-\otimes V_n$ (up to a twist by $\Pi^k$).
\end{proof}
\begin{remark}
 A similar idea is used to prove the Kac-Wakimoto conjecture for $\p(n)$, as
will be shown in an upcoming work by the authors.
\end{remark}

\begin{lemma}\label{lem:translation_simple_socle}
 If $Y$ has simple socle, then $\pmb{\Theta}_{i}(Y)$ is either zero or has
simple
socle. Similarly, if $Y$ has simple cosocle, then
$\pmb{\Theta}_{i}(Y)$ is either zero or has simple cosocle.
\end{lemma}

 \begin{proof}
 We prove the second statement (the first one is proved analogously).
Let $k\geq 0$ be such that $Y \in
\urepk{k}$. Let $n>>k$ be such that $\urepk{k+1} \cong
Rep^{k+1}(\p(n))$.
Then $F_n(Y)$ has simple cosocle and hence is a quotient of
an indecomposable
projective. Applying translation functors to an indecomposable projective in
$\Rep(\p(n))$, one obtains either zero or once again an indecomposable
projective.
Thus $\Theta_i F_n(Y)$ is either zero, or has simple cosocle as well. Now,
$\Theta_i F_n(Y)
\in Rep^{k+1}(\p(n)) \cong \urepk{k+1}$ so $\pmb{\Theta}_{i}(Y)$ is either
zero,
or has simple cosocle, as required.
\end{proof}

\subsection{Properties of the category
\texorpdfstring{$\urep$}{Rep(P)}}\label{ssec:urep_prop}

We list below some ``local'' and ``global'' properties of the category $\urep$.

We begin with the ``global'' properties:
\begin{enumerate}
 \item The isomorphism classes
of simple objects in $\urep$ (up to \InnaB{parity} shift) are parametrized by infinite non-decreasing
integer sequences
$\lambda=(\lambda_1, \lambda_2, \ldots)$ with $\lambda_i=0$ for $i>>0$. Such
sequences will be called
{\it weights} for $\urep$, and the set of weights will be denoted by $\Lambda$.

Every simple object
is isomorphic to $\L(\lambda) \in \urepk{\abs{\lam}}$ which we define as the
inverse limit of simple
$\p(n)$-modules $L_n(\lambda)$.
\item Any object in $\urep$ has finite length.
\item For any infinite non-decreasing
integer sequence $\lambda=(\lambda_1, \lambda_2, \ldots)$ with $|\lambda|\leq
k$ we define $\bdel(\lambda)$ ({\it standard objects}) as the inverse limit of
$\overline{\Delta}^k_n(\lambda)$. Similarly, we define $\bnab(\lambda)$ ({\it
costandard objects}) as the
inverse limit of
$\bar{\nabla}^k_n(\lambda)$. Then
the cosocle of $\bdel(\lambda)$ and the socle of $\bnab(\lambda)$ are
isomorphic to
$\L(\lambda)$.
\item We have $$\dim \Hom_{\urep}
\left(\bdel(\lambda), \bnab(\mu) \right) =
\delta_{\lambda, \mu} ,\;\;\;\; \Ext^1_{\urep}
\left(\bdel(\lambda), \bnab(\mu) \right) = 0.$$ (see Lemma \ref{lem:ext}).
 \item
  For any weight $\lambda$ in $\urep$, we have (see Corollary
\ref{cor:Delta_nabla_dual}): $$\bdel(\lambda)^* \cong
\Pi^{\abs{\lambda}} \bnab(\lambda^{\vee}) \;\; \InnaA{\text{ and } \L(\lambda)^* \cong
\Pi^{\abs{\lambda}} \L(\lambda^{\vee}) }$$ and have maps $$\bdel(\lambda)
\hookrightarrow S^{\pmb{\lambda}^{\vee}} \bV, \;\;  S^{\pmb{\lambda}\InnaC{^{\vee}}} \bV
\twoheadrightarrow
\bnab(\lambda).$$
\item For each $\lambda$, the functor $\InnaA{\bV \otimes (-)}$ takes $\bdel(\lam)$ to a
$\bdel$-filtered object, and $\bnab(\lam)$ to a $\bnab$-filtered object,
according to Corollaries \ref{cor:ses}, \ref{cor:ses_nabla}.
In particular, the tensor powers of $\bV$ have both standard and costandard
filtration.
\end{enumerate}

Below are some ``local'' properties of $\urep$, namely, properties
of the subcategories $\urepk{k}$:
\begin{enumerate}
\item Simple objects $\L(\lambda)$ lying in $\urepk{k}$ are those for which
$\lambda$ satisfies: $|\lambda|\leq k$. By abuse of notation, we will call such
weights $k$-admissible.
\item The category $\urepk{k}$ has enough projectives and injectives (see Lemma
\ref{lem:rep_k_enough_proj}); the projective cover of $\L(\lambda)$ in
$\urepk{k}$ will be denoted by $\pmb{P}_k(\lambda)$.
\item The tensor structure on $\urep$ is given by functors $\urepk{k} \otimes
\urepk{l}
\rightarrow \urepk{k+l}$ and $\urepk{k}$ is closed under the tensor duality
contravariant functor $(\cdot)^*$.
\end{enumerate}

\subsection{Connection to the category
\texorpdfstring{$\PP$}{P}}\label{ssec:func_I}

\begin{lemma}\label{lem:funct_I}
 There exists a fully faithful SM $\C$-linear $\sVect$-functor $I:\PP \to
\urep$ which sends the generator $\tilde{V}$ of
$\PP$
to $\bV$, and the form $\omega_{\tilde{V}}$ on $\tilde{V}$ to $\omega_{\bV}$.
Furthermore, there is
a natural $\otimes$-isomorphism making the following diagram of functors
commutative:
$$ \xymatrix{&{\PP} \ar[r]^{I} \ar[rd]_{I_n}
&\urep\ar[d]^{F_n}  \\ &{} &{\Rep(\p(n))} }.$$
\end{lemma}
\begin{proof}
 Consider the additive SM $\sVect$-functor $I:\PP \to \urep$ defined by sending
$\tilde{V}$ to $\bV$, and the form on
$\tilde{V}$
to $\omega_{\bV}$. Such a functor is uniquely defined (up to a unique natural
$\otimes$-isomorphism).

The composition $F_n \circ I$ is then an additive SM functor $\PP \to
\Rep(\p(n))$ satisfying similar conditions, and hence it is isomorphic to $I_n$
(by the uniqueness of $I_n$) under some natural $\otimes$-isomorphism, making
the above diagram commutative.

It remains to check that $I$ is fully faithful. Indeed, it is enough to check
that for any $r, s \geq 0$, $$I: \Hom_{\PP}\left(\tilde{V}^{\otimes r},
\tilde{V}^{\otimes s}\right) \longrightarrow \Hom_{\urep}\left(\bV^{\otimes
r},\bV^{\otimes s}\right)$$ is an isomorphism. Taking some $n>>r, s$ we see
that
$$I_n \cong F_n \circ I: \Hom_{\PP}\left(\tilde{V}^{\otimes r},
\tilde{V}^{\otimes s}\right) \longrightarrow
\Hom_{Rep^k(\p(n))}\left(V_n^{\otimes
r},V_n^{\otimes s}\right)$$ is surjective, implying that the previous map was
surjective as well. To check that it is also injective, it remains to check
that $$\dim\Hom_{\PP}\left(\tilde{V}^{\otimes r},
\tilde{V}^{\otimes s}\right) = \dim\Hom_{\urep}\left(\bV^{\otimes
r},\bV^{\otimes s}\right)$$ which follows directly from Proposition
\ref{prop:contractions}.

\end{proof}

\subsection{Connection to representations of
\texorpdfstring{$\p(\infty)$}{p(infty)}}\label{ssec:funct_Phi_to_Deligne}
\begin{proposition}\label{prop:func_Phi_to_Deligne}
 There exists an exact faithful SM $\sVect$-functor $\Phi: \Rep(\p(\infty)) \to
\urep$ taking $V_{\infty}$ to $\bV$, and the form
$V_{\infty} \otimes V_{\infty} \to \Pi \C$ to $\omega_{\bV}$. Furthermore,
there is a natural $\otimes$-isomorphism making the following diagram of
functors commutative:
$$ \xymatrix{&{\Rep(\p(\infty))} \ar[r]^{\Phi} \ar[rd]_{\Phi_n}
&\urep\ar[d]^{F_n}  \\ &{} &{\Rep(\p(n))} }$$
\end{proposition}
\begin{proof}
Recall that for $n>>k$, the SM functor $\Phi_n: \Rep(\p(\infty)) \to
\Rep(\p(n))$ restricts to an exact functor $\Rep^k(\p(\infty)) \to \Rep^k(\p(n))$.
This functor commutes with the $DS_{x_n}$ functor, as shown in Lemma
\ref{lem:DS_Phi_compat}. Hence we have a well-defined exact functor $\Phi:
Rep^k(\p(\infty)) \to \urepk{k}$ extending to a functor $\Phi: \Rep(\p(\infty))
\to \urep$. This functor clearly satisfies the statement of the proposition.
\end{proof}
By Lemma \ref{prop:Phi_simples_to_delta}, we also have:

\begin{corollary}\label{cor:Phi_simples_to_delta_Deligne}
 Let $L_{\infty}(\lam)$ be a simple object in $\Rep(\p(\infty))$. Then
$\Phi(L_{\infty}(\lam)) \cong \bdel(\lam)$.
\end{corollary}
\begin{corollary}\label{cor:Phi_inj_to_tilt}
 Let $E$ be an injective object in $\Rep(\p(\infty))$. Then
$\Phi(E) \in I(\PP)$.
\end{corollary}
\begin{proof}
The subcategory of injective objects in $\Rep(\p(\infty))$ is generated, as an
additive Karoubian category, by the tensor powers of $V_{\infty}$. Since $\Phi$
is $\C$-linear and monoidal, for any injective object $E \in
\Rep(\p(\infty))$, $\Phi(E)$ belongs to the full additive Karoubian subcategory
generated by tensor powers of $\bV$. This subcategory is precisely $I(\PP)$.
\end{proof}

\section{The first universal property}\label{sec:universality}
To prove the universality of category $\urep$, we use the following theorem,
proved in \cite{EHS}.

Consider a SM functor  $I:\mathcal{D}\to\mathcal T$ from an
additive
$\C$-linear rigid SM category $\mathcal{D}$ to a tensor $\C$-linear category
$\mathcal T$.

We assume the following conditions hold:
\begin{enumerate}\label{cond:uni}
\item\label{cond:uni1} The functor $I:\mathcal{D}\to \mathcal T$ is fully
faithful and $\C$-linear.
\item\label{cond:uni2} Any object $X\in \mathcal T$ can be presented as an
image of a map $I(f)$
for some $f:P\to Q$ in $\mathcal{D}$.
\item\label{cond:uni3} For any epimorphism $X\to Y$ in $\mathcal T$ there
exists a nonzero
$T\in\mathcal{D}$
such that the epimorphism $X\otimes I(T)\to Y\otimes I(T)$ splits.
\end{enumerate}

\begin{theorem}[\cite{EHS}]\label{old_thm:uni-2}
Under these assumptions the functor $I$ induces for any tensor category
$\mathcal A$ an equivalence of the following categories
\begin{itemize}
\item $\Fun^{\ex}(\mathcal V,\mathcal A)$, the category of faithful exact SM
functors
$\mathcal V\to \mathcal A$,
\item $\Fun^{\faith}(\mathcal{D},\mathcal A)$, the category of faithful SM
functors $\mathcal{D}\to \mathcal A$.
\end{itemize}
\end{theorem}
We will apply this theorem to $\mathcal{D} = \PP$, $\mathcal V = \urep$ to
obtain the following results:
\begin{theorem}\label{thrm:uni-1}
For any \InnaD{tensor} $\sVect$-category
$\mathcal A$ we have an equivalence of the following categories
\begin{itemize}
\item $\Fun^{\ex}(\urep,\mathcal A)$, the category of exact SM
$\sVect$-functors
$\mathcal T\to \mathcal A$,
\item $\Fun^{\faith}(\PP,\mathcal A)$, the category of faithful SM
$\sVect$-functors $\PP\to \mathcal A$.
\end{itemize}
\end{theorem}
\begin{proof}
 We only need to check that there is a fully faithful SM embedding $\PP
\hookrightarrow \urep$ satisfying Conditions \ref{cond:uni}. This will be
done in the remainder of this Section.
\end{proof}

In particular, we have:
\begin{corollary}\label{cor:two_urep_equiv}
 The categories $\urep$ and $\Rep(\underline{P}')$, defined in Section
\ref{ssec:constr_urep} are equivalent as tensor $\sVect$-categories. This
equivalence commutes with the embeddings of
$\PP$ into both categories.
\end{corollary}

\subsection{Overview of the proof}\label{ssec:universality_proof}

The main ingredient in proving Theorem \ref{thrm:uni-1} is
showing that there exists a SM functor $I: \PP \to \urep$ satisfying the
Conditions \ref{cond:uni}.

Condition \eqref{cond:uni1} has been proved in Lemma \ref{lem:funct_I}.
Condition \eqref{cond:uni2} will be proved in Section \ref{ssec:presentation}
below,
and Condition \eqref{cond:uni3} will be proved in Section \ref{ssec:epi}.

\subsection{Presentation of objects}\label{ssec:presentation}

In this section we prove the following Proposition, required for Theorem
\ref{thrm:uni-1}:
\begin{proposition}\label{prop:presentation}
 For any $M \in \urep$, there exist objects $T_1, T_2 \in \PP$ (finite
direct summands of direct sums of tensor powers of $\bV$) and maps $T_1
\twoheadrightarrow M \hookrightarrow T_2$.
\end{proposition}
The outline of the proof follows that of \cite[Proposition 8.4.1]{EHS}:
\begin{proof}
 Due to the existence of duality $( \cdot )^*$ in $\urep$ which preserves
$I(\PP)$, it is enough to
prove the existence of $T_1$ as above.

We will prove the statement in several steps.

\begin{enumerate}
 \item[Step 1]: We prove the statement for $M= \pmb{P}_k(0)$, the projective
cover of $\triv$ in $\urepk{k}$ (see Lemma \ref{lem:trivial_projective_cover}).
 \item[Step 2]: We prove the statement for any standard object $\bdel(\lambda)$
in $\urep$ (see Lemma \ref{lem:Delta_cover}).
 \item[Step 3]: We prove the statement for any projective object $P$ in
$\urepk{k}$, for any $k \geq 0$.
 \item[Step 4]: We prove the statement for any object $M$ in $\urep$.
\end{enumerate}

{\bf Proof of Step 3}: It is enough to prove the statement for indecomposable
objects $\pmb{P}_k(\lambda)$ where $|\lambda| \leq k$.

Consider the inclusion functor $\mathit{i}^{k}:\urepk{k} \to \urep$ and its
left adjoint $\mathit{i}^{k,*}: \urep \to \urepk{k}$. Let $\pmb{P}_{2k}(0)$ be
the projective cover of $\triv$ in $\urepk{2k}$ and let $\L(\lambda) \in
\urep$. Consider $$Y:= \mathit{i}^{k,*}( \pmb{P}_{2k}(0) \otimes \L(\lambda)
),$$ the
maximal quotient of $\pmb{P}_{2k}(0) \otimes \L(\lambda)$ lying in $\urepk{k}$.
Then for any $M \in \urepk{k}$, we have:
$$\Hom_{\urepk{k}}(Y, M) \cong \Hom_{\urep}(\pmb{P}_{2k}(0) \otimes
\L(\lambda), M) \cong \Hom_{\urepk{2k}}(\pmb{P}_{2k}(0),
\L(\lambda)\otimes M)$$ which means that $Y$ is a projective object in
$\urepk{k}$ (see also \cite[Lemma 8.2.1]{EHS}).

The covering epimorphism $
\pmb{P}_{2k}(0) \twoheadrightarrow
\triv$ induces an epimorphism $$\pmb{P}_{2k}(0) \otimes \L(\lambda)
\twoheadrightarrow \L(\lambda) $$ which factors through an epimorphism $Y
\twoheadrightarrow \L(\lambda) $.

By definition of projective cover, the latter induces a split epimorphism $$Y
\twoheadrightarrow \pmb{P}_k(\lambda).$$

We now consider the composition $$ \pmb{P}_{2k}(0) \otimes \bdel(\lambda)
\twoheadrightarrow \pmb{P}_{2k}(0) \otimes \L(\lambda)  \twoheadrightarrow Y
\twoheadrightarrow \pmb{P}_k(\lambda)$$ where the first map is induced by the
epimorphism $\bdel(\lambda) \twoheadrightarrow \L(\lambda)$.

Applying Steps 1 and 2, we conclude that there exists an epimorphism $T
\twoheadrightarrow P_k(\lambda) $ where $T$ is in the image of $\PP$.

{\bf Proof of Step 4}:
Let $k$ be such that $M$ belongs to $\urepk{k}$. The category $\urepk{k}$ has
enough projectives, so there exists an epimorphism $ P \twoheadrightarrow M$
where $P$ is a projective object in $\urepk{k}$. Applying Step 3, we obtain an
epimorphism $T \twoheadrightarrow P$, with $T$ in the image of $\PP$;
composed with the former, it gives an epimorphism $$ T \twoheadrightarrow M$$
as
wanted.

\end{proof}

 \begin{lemma}\label{lem:lift_proj_cover}
Let $n \geq 0$. Then there exists $\pmb{X}_n \in I(\PP)$ such that
$F_n(\pmb{X}_n) =
P_n(0)$ (projective cover of trivial
module in $\Rep(\p(n))$), and the cosocle of $\pmb{X}_n$ is $\triv$.
 \end{lemma}
\begin{proof}
We will use the following auxilary statement.
\begin{sublemma}
 There exist translation functors
$\Theta_{i_1}, \ldots, \Theta_{i_l}$ on $\Rep(\p(n))$ such that
$P_n(0)
= \Theta_{i_l} \circ \ldots \circ \Theta_{i_1} \C$.
\end{sublemma}
\begin{proof}
Recall that by Corollary \ref{cor:Rep_gen_by_V}, any object in $\Rep(\p(n))$ is
a
subquotient of a direct sum of tensor powers of $V_n$. Hence any indecomposable
projective is an indecomposable direct summand of some tensor power of $V_n$.
It remains to check that summands of the form $\Theta_{i_l} \circ \ldots
\circ \Theta_{i_1} \C$ are indeed indecomposable. Indeed, translation functors
take indecomposable projectives either to zero or to indecomposable
projectives (see Theorem \ref{old_thrm:transl_proj}), and hence take objects
with simple cosocle to either zero or objects with simple cosocle. Thus objects
of the form $\Theta_{i_l} \circ \ldots
\circ \Theta_{i_1} \C$ have simple cosocle and are indecomposable.
 \end{proof}

 We now define $\pmb{X}_n := \pmb{\Theta}_{i_l} \circ \ldots \circ
\pmb{\Theta}_{i_1} \triv \in \urep$ where $i_1, \ldots, i_l$ are as in the
Sublemma above. Then $F_n(\pmb{X}_n) \cong P_n(\triv)$ by Corollary
\ref{cor:transl_goes_to_transl}.

The object $\pmb{X}_n$ has simple cosocle
by Lemma
\ref{lem:translation_simple_socle}, so we only need to check that there exists
a non-zero map $\pmb{X}_n \to \triv$, which follows from Lemma
\ref{lem:F_n_surjective}.
\end{proof}

\begin{lemma}\label{lem:trivial_projective_cover}
 Let $k\geq 0$ and let $\pmb{P}_k(0)$ be the projective
cover of $\triv$ in $\urepk{k}$. Then there exists a pair $(T, f)$ where $T \in
I(\PP)$ and $f: T \to \pmb{P}_k(0)$ is a surjective map.
\end{lemma}
The proof is completely analogous to \cite[Lemma 8.4.3]{EHS}. We
present it here for completeness.
\begin{proof}
Let $n>>k$ be such that $\urepk{k} \cong \Rep^k(\p(n))$. Let $P_n(0)$ be the
projective cover of the trivial module in $\p(n)$ and let $\pmb{X}_n \in
I(\PP)$
be
as in Lemma \ref{lem:lift_proj_cover}. Let $\pmb{P}_k(0)$ be the projective
cover of $\triv$ in
$\urepk{k}$, and let $\pi: P_n(0) \to F_n(\pmb{P}_k(0))$ be the quotient
map.

Denote by $\pmb{Z}:=\mathit{i}^{k,*}(\pmb{X}_n)$ the
maximal quotient of $\pmb{X}_n$ lying in $\urepk{k}$, with projection $\phi:
\pmb{X}_n
\to \pmb{Z}$. Then $cosoc(\pmb{Z}) \cong \triv$, and hence we have a surjective
map
$p:\pmb{P}_k(0) \to \pmb{Z}$ making the diagram
$$ \xymatrix{&{\pmb{P}_k(0)} \ar[r]^{p} \ar@{->>}[rd] &{\pmb{Z}} \ar@{->>}[d]
\\ &{}
&{\triv} }$$
commutative. Clearly it is
enough to prove that $p$ is an isomorphism.

We will prove that $p$ is a monomorphism (and hence an isomorphism). Indeed,
recall that $\pmb{P}_k(0)$ is a projective object in $ \urepk{k}$; the objects
in this subcategory are \InnaB{subquotients} of objects in $I(\PP)\cap \urepk{k}$. Hence
there exists $\pmb{D} \in
I(\PP)\cap \urepk{k}$ such that $\pmb{P}_k(0) \subset \pmb{D}$. We
denote
this inclusion by $\bar{f}$. Composing $F_n(\bar{f})$ with $\pi:F_n(\pmb{X}_n)
\cong
P_n(0) \to F_n(\pmb{P}_k(0))$, we obtain a map $g=F_n(\bar{f}) \circ \pi:
F_n(\pmb{X}_n) \to F_n(\pmb{D})$.
By Corollary \ref{cor:F_n_full_tensor}, the functor $F_n$ is full on $I(\PP)$,
hence there exists a map $\bar{g}:\pmb{X}_n \to \pmb{D}$ in $\urep$ such that
$F_n(\bar{g}) =g$. Since $\pmb{D} \in \urepk{k}$, this map factors through
$\pmb{X}_n
\twoheadrightarrow \pmb{Z}$, inducing a map $\alpha: \pmb{Z} \to \pmb{D}$ such
that
$\bar{g}=\alpha\circ \phi$.

Hence we have:
$g=F_n(\bar{f})\circ \pi = F_n(\alpha)\circ F_n(p) \circ \pi$. The map $\pi$ is
surjective,
so by cancellation law, $F_n(\bar{f}) = F_n(\alpha)\circ F_n(p)$. Since
$F_n(\bar{f})$ is
injective, so is $F_n(p)$, and so is $p$ (here we use again that $p
\in \urepk{k} \cong \Rep^k(\p(n))$).

To sum up our constructions: we have commutative diagrams in $\urep$ (left) and
$\Rep(\p(n))$ (right):
$$ \xymatrix{ &\pmb{X}_n \ar[r]^{\bar{g}} \ar@{->>}[rrd]|{\,\phi\,} &\pmb{D}
&{} \\  &\pmb{P}_k(0) \ar@{^{(}->}[ru]|{\,\bar{f}\;} \ar@{->>}[rrd]
\ar@{->>}[rr]_p &{} &\pmb{Z}
\ar@{->>}[d] \ar[ul]_{\alpha} &{} \\ &{} &{} &\triv}
 \xymatrix{ &{P_n(0) = F_n(\pmb{X}_n)} \ar[rr]^-{g} \ar@{->>}[d]
\ar[rrrrd]|{\,F_n(\phi)\,} &{}  &F_n(\pmb{D})  &{} &{} &{} \\
&{F_n(\pmb{P}_k(0))}
\ar@{->>}[rrrr]|{\;F_n(p)\;}  \ar@{->>}[rrrrd]\ar@{^{(}->}[rru]|{\,F_n(f)\;\;\;} &{} &{}&{}
&F_n(\pmb{Z})
\ar@/_1.5pc/[llu]|{\;\;F_n(\alpha)\;\;} \ar@{->>}[d] \\ &{} &{} &{} &{} &\C}.
$$

\end{proof}

\begin{lemma}\label{lem:Delta_cover}
 Let $\lambda$ be a partition. There exists $T \in \PP$ such that $ T
\twoheadrightarrow \bdel(\lam)$ in $\urep$.
\end{lemma}
\begin{proof}
Recall that by Lemma \ref{lem:p_infty_cosocle}, every simple in
$\Rep(\p(\infty))$
occurs in
the cosocle of some injective object, i.e. as a cosocle of some
direct summand of a tensor power of $V_{\infty}$. Together with Corollary
\ref{cor:Phi_simples_to_delta_Deligne}, this implies that $\bdel(\lam)$ is
a
quotient of some direct summand of a tensor power of $\bV$ in $\urep$, as
required.
\end{proof}

\subsection{Splitting of epimorphism}\label{ssec:epi}

\begin{proposition}\label{prop:splitting} For any epimorphism $\pmb{X}\to
\pmb{Y}$ in
$\urep$ there exists a nonzero $T\in\PP$
such that the epimorphism $\pmb{X}\otimes I(T)\to \pmb{Y}\otimes I(T)$ splits.
\end{proposition}
\begin{proof} We will start with the case that $\pmb{Y}=\triv$ or $\Pi\triv$.
It
suffices to prove the statement for $\pmb{X}={\bV}^{\otimes 2d}$ since every
$\pmb{X}$ is
a quotient
of a finite direct sum of $\bV^{\otimes k}$ and $\operatorname{SHom}({\bf
V}^{\otimes k},\triv)\neq 0$ implies that $k$ is even.

Proposition \ref{prop:contractions} implies that
$$\operatorname{SHom}({\bV}^{\otimes k},\triv)=\operatorname{SHom}_{Rep
(\p(\infty))}(V_{\infty}^{\otimes k},\C).$$
Consider the $S_{2d}$-decomposition in $\Rep(\p(\infty))$
$$ V_{\infty}^{\otimes 2d}=\oplus _{|\lambda|=2d}S^\lambda V_{\infty}\otimes
Y_\lambda.$$
The cosocle of $S^\lambda V_{\infty}$ contains $\C$ if and only if $\lambda$ is
quasisymmetric and in this case the multiplicity of $\C$ in $cosoc(S^\lambda
V_{\infty}) $ is $1$. Therefore it suffices to show that for any
quasisymmetric $\lambda$ an epimorphism $\tau:S^\lambda \bV \to\triv$
splits.

As a right $\C[S_{2d}]$-module the space $\operatorname{SHom}(\bV^{\otimes
k},\triv)$ is generated by the map
$$\gamma:\bV^{\otimes 2d}\xrightarrow{\varphi}(\Pi\bV^*)^{\otimes
d}\otimes\bV^{\otimes d}\xrightarrow{ev} \triv, $$
where $\varphi=\omega^{\otimes d}\otimes id^{\otimes d}$ and $\omega:\bV\to
{\Pi\bf V}^*$ is the isomorphism induced by the form.
Moreover,
$\operatorname{SHom}(\bV^{\otimes k},\C)$ as a right $\C[S_{2d}]$ module is
induced from one-dimensional $H$-module $\C\gamma$, where
$H$ is the stabilizer of $\C\gamma$ in $S_{2d}$ and is isomorphic to the
Coxeter group of type $B_{d}$. In particular, this module has a basis which
consists
of $\gamma\sigma$ for some permutations $\sigma\in S_{2d}$. Let $\pi_{\lambda}$
be a Young projector associated with some tableau of shape $\lambda$ for
a quasisymmetric partition $\lambda$. Then $\pi_{\lambda}\sigma\gamma\neq 0$
for some $\sigma$. Therefore, using the conjugation by $\sigma$ we may assume
without
loss of generality that $\pi_{\lambda}\gamma\neq 0$. Then $\tau$ is a
composition
$$S^\lambda(\bV)\xrightarrow{i_\lambda}\bV^{\otimes
2d}\xrightarrow{\varphi}(\Pi\bV^*)^{\otimes d}\otimes \bV^{\otimes
d}\xrightarrow{ev}\triv $$
for an embedding $i_\lambda$ such that $\pi_{\lambda}\circ
i_\lambda=i_{\lambda}$.
Then the epimorphism
$$\bV^{\otimes d}\otimes S^\lambda(\bV)\xrightarrow{id\otimes\tau}{\bf
V}^{\otimes d}$$
splits with splitting morphism given by
$(id\otimes\pi_{\lambda})\circ(id\otimes\varphi^{-1})\circ(coev\otimes id)$.
\end{proof}

\section{The second universal property}\label{sec:uni-2}

In this section we prove another universality property, describing the
collection of categories $\langle \urep, \Rep(\p(n))\rangle_{n \geq 1}$ as
universal tensor $\sVect$-categories generated by a single object with a
pairing into $\Pi\triv$. This result is stated in Theorem \ref{thrm:uni-2}.
The statement and proof of the theorem requires the language of affine group
schemes in pre-Tannakian tensor categories. The necessary definitions and
results are given in Appendix \ref{app:affine_grp_schemes}.

From now on, we will consider pre-Tannakian tensor $\sVect$-categories (see
Section \ref{sec:notn} for definitions), and
will omit the pre-Tannakian assumption, as well as write {\it tensor functor}
instead of ``tensor $\sVect$-functor''.
\subsection{Affine group scheme
\texorpdfstring{$P(X)$}{P(X)}}\label{ssec:aff_grp_scheme}

\subsubsection{Definition}
Let $\T$ be a  $\sVect$-tensor category
and let $X \in \T$ be an object together with a non-degenerate symmetric form
$\omega_X: X \otimes X
\to \Pi \triv$. We denote by $\psi_X: X \to \Pi X^*$ the induced isomorphism.

Consider the functor $\Alg_{\T} \to\Grps$ defined by the formula
$$ A\mapsto\Aut_{ A-\Mod}(i_A(X), i_A(\omega_X)).$$
That is, $A$ is sent to the group of $A$-module automorphisms $\theta: A\otimes
X \to A \otimes X$ such that the diagram below is commutative:
\begin{equation}\label{eq:functor_points_P(X)}
\xymatrix{&{A \otimes X} \ar[d]_{\theta}
\ar[rrr]^-{i_A(\psi_X)}
&{} &{} &{A \otimes\Pi \triv \otimes_A A \otimes X^*}
\ar[d]^{\id \otimes_A \theta^{\vee}}\\ &{A \otimes X}
\ar[rrr]_-{ i_A(\psi_X)} &{} &{}
&{A\otimes \Pi\triv \otimes_A A \otimes X^*} }
\end{equation}
where $\theta^{\vee}: A \otimes X^* \to  A \otimes X^*$ is given by
\begin{align*}
 &A \otimes X^* \xrightarrow{\id \otimes coev_X} A \otimes X^* \otimes X
\otimes X^* \xrightarrow{\sigma_{X, X^*}} A \otimes X
\otimes {X^*}^{\otimes  2}
\xrightarrow{\theta \otimes \id} A \otimes X
\otimes {X^*}^{\otimes  2} \rightarrow \\ &\xrightarrow{\sigma_{X, X^*}} A
\otimes X^* \otimes X
\otimes X^* \xrightarrow{\id \otimes ev_X \otimes \id } A \otimes X^*
\end{align*}
and $\sigma_{A, \Pi}: A \otimes \Pi \triv \to \Pi A$ is the symmetry
isomorphism.

As before, this functor is corepresentable. The Hopf $\T$-algebra representing
this functor is the largest Hopf $\T$-algebra quotient of $\O(GL(X))$ which
preserves the form $\omega$. It can be written explicitly as a quotient of
$Sym(X \otimes X^*)$ by a certain ideal, similar to \cite[Definition 2.3]{Et}.

\begin{definition}
 Let $P(X) \in \Grps_{\T}$ be the $\T$-group representing the above functor.
\end{definition}
We have an obvious $\T$-group inclusion $P(X) \to GL(X)$, so $X$ is a faithful
representation of $P(X)$.

By Lemma
\ref{lem:faithful_rep}, this immediately implies:

\begin{lemma}\label{lem:grp_repr_subquotient_tensor_powers}
 Any representation of $P(X)$ in $\T$ is a subquotient of $\bigoplus_{i
\in I}
X^{\otimes
a_i} \otimes U$ for some finite set $I$, $a_i \in \Z_{\geq 0}$ and some $U
\in\T$ considered with a trivial $P(X)$-action.
\end{lemma}

We now want to consider such representations when $U$ belongs to the essential
image of $\sVect$ in $\T$:
\begin{notation}\label{notn:Rep_P_X}
 We denote by
$\Rep(P(X))$ be the tensor $\sVect$-category of $P(X)$-representations
generated
by tensor powers of $X$ and their \InnaB{parity} shifts. Namely, the objects in this
category are subquotients of the $P(X)$-representation $\bigoplus_{i \in I}
X^{\otimes
a_i} \otimes \left(\Pi\triv\right)^{\otimes b_i}$ for some finite set $I$,
and some $a_i
\in \Z_{\geq 0}, b_i \in \{0, 1\}$. The morphisms are $P(X)$-equivariant
$\T$-morphisms.
\end{notation}

\begin{remark}
 By Lemma \ref{lem:grp_repr_subquotient_tensor_powers}, for $\T = \sVect$
and $X$ a finite dimensional vector superspace, the category $\Rep(P(X))$
contains all the finite dimensional (super-)representations of $P(X)$.
\end{remark}

Let $\widetilde{P}_X$ denote the fundamental group of
$\Rep(P(X))$.

\begin{lemma}\label{lem:fund_group_extension}
 We have an inclusion $i:P(X) \to \widetilde{P}_X$
($P(X)$ considered with conjugation action) and a quotient $q:
\widetilde{P}_X \to \mu_2$, with $q \circ i =1$.
\end{lemma}
\begin{remark}
 The canonical $\T$-group homomorphism $\pi_X:\pi(\T) \to GL(X)$ does
not factor through the inclusion $P(X) \to
GL(X)$; this is manifested in the fact that $\widetilde{P}_X$ does not split
into a direct product of $P(X)$ and $\mu_2$, and the category $\Rep(P(X))$
cannot be ``halved'', unlike for example $\Rep(GL(m|n))$ (see \cite{EHS}).
\end{remark}
\begin{proof}
Let $\widetilde{P}_X$ denote the fundamental group of
$\Rep(P(X))$.
The action of
$P(X)$ on its own representations induces a $\Rep(P(X))$-group inclusion
$i:P(X) \to \widetilde{P}_X$.

We have a tensor functor $J:\sVect \to \Rep(P(X))$ taking a vector superspace
$V$ to $V \otimes \triv$ with trivial $P(X)$-action. This functor induces a
$\Rep(P(X))$-group homomorphism $q:\widetilde{P}_X \to J(\pi(\sVect)) \cong
\mu_2$. This homomorphism is surjective by Lemma \ref{lem:surjective_gen}.

Since $P(X)$ acts trivially on the objects in the image of
$J$, the composition of the homomorphisms $P(X) \to \widetilde{P}_X
\to \mu_2$ is trivial, as required.

\end{proof}

\begin{remark}
 In fact, one can show that $\widetilde{P}_X$ is an extension of $P(X)$ and
$\mu_2$ whenever all trivial representations of $P(X)$ in $\Rep(P(X))$ belong to
the essential image of $J$.
\end{remark}

\subsubsection{Functoriality}\label{sssec:P_functoriality}

\begin{lemma}\label{lem:P_functoriality}
 Let $F:\T' \to \T$ be a tensor
functor between two tensor $\sVect$-categories. Let $X \in \T'$, and $\omega_X:
X \otimes X \to \Pi \triv$ a symmetric form. Then we have an
isomorphism of $\T$-group schemes $F(P(X)) \cong P(F(X))$.
\end{lemma}
\begin{proof}
Let $G = P(X)$. The $\T$-group scheme $F(G)$ is defined by the functor of points
$\Alg_{\T}\to\Grps$ given by the formula
\begin{equation}\label{eq:FofG}
F(G)(B)=G(F^!(B)).
\end{equation}
where $F^!:\Ind-\T\to\Ind-\T'$ is the right adjoint to $F$, which is
automatically lax symmetric monoidal. We can define a $\T$-group homomorphism
$$F(G) \to P(F(X)) $$ on the functors of points in the following way: for a
fixed
$B\in\Alg_{\T}$, we would like to define a group homomorphism
$$ \Aut_{\Mod_{F^!(B)}}\left(i_{F^!(B)}(X), i_{F^!(B)}(\omega_X)\right) \to
\Aut_{\Mod_B}\left(i_B(F(X)), i_B(F(\omega_X))\right) $$
For $$\alpha \in  \Aut_{\Mod_{F^!(B)}}\left(i_{F^!(B)}(X),
i_{F^!(B)}(\omega_X)\right) =  \Aut_{\Mod_{F^!(B)}}\left(F^!(B) \otimes X,
F^!(B) \otimes \omega_X\right)$$ we have a map $F(\alpha) \in
\Aut_{\Mod_{FF^!(B)}}\left(FF^!(B) \otimes X, FF^!(B) \otimes
\omega_X \right)$. Under base change $FF^!(B) \to B$ we obtain an element
of $\Aut_{\Mod_B}\left(i_B(F(X)), i_B(F(\omega_X))\right) $. Hence we built a
$\T$-group homomorphism
$F(G) \to P(F(X)) $. One can then check that it is an isomorphism by
writing explicitly $\O(G) = \O(P(X))$ as a quotient of $Sym(X \otimes X^*)$
(see \cite{Et}).
\end{proof}

 Let us now consider the following situation.

 Let $\T'$ be a tensor
$\sVect$-category, and let $X \in \T'$ with a symmetric form $\omega_X: X
\otimes X \to
\Pi \triv$. Assume $\T'$ is generated by $X,
\omega_X$ as a tensor $\sVect$-category; that is, assume we have a short
exact
sequence of $\T'$-groups
$$ 1 \to P(X) \to \pi(\T') \to \mu_2 \to 1$$
 and any object $M \in \T'$ is a subquotient of $\bigoplus_{i \in I}
X^{\otimes
a_i} \otimes \left(\Pi\triv\right)^{\otimes b_i}$ for some finite set $I$,
and some $a_i
\in \Z_{\geq 0}, b_i \in \{0, 1\}$.

\begin{lemma}\label{lem:funct}
Let $F: \T' \to \T$ be a tensor functor into a tensor
$\sVect$-category $\T$. Let $Y:=F(X)$. Then $F$ induces an equivalence of
categories $\T' \to \Rep(P(Y))$.
\end{lemma}
\begin{proof}

 Let $G = \pi(\T)$, $G' = \pi(T')$ and $\eps: G \to F(G')$ be the induced
morphism (as in Section \ref{sssec:fund_grp}). By Theorem
\ref{old_thm:Deligne_reconstr}, we have an equivalence of tensor
$\sVect$-categories
$$ \T' \stackrel{\sim}{\longrightarrow} Rep_{\T}(F(G'), \eps)$$
(notation as in Theorem
\ref{old_thm:Deligne_reconstr}).

Now, by Lemma \ref{lem:P_functoriality}, $F(P(X)) \cong P(Y)$, hence we
have a $\T$-group inclusion $P(Y) \to F(G')$. The corresponding restriction
functor is a tensor functor $$R: Rep_{\T}(F(G'), \eps) \to \Rep(P(Y))$$ and we
wish
to show that this is an equivalence. For this, let $\widetilde{P}_Y$ denote the
fundamental group of $\Rep(P(Y))$, and $\eps': G \to \widetilde{P}_Y$ the
corresponding $\T$-group homomorphism.

By Theorem \ref{old_thm:Deligne_reconstr}
we have a tensor equivalence $\Rep(P(Y)) \cong \Rep(\widetilde{P}_Y, \eps')$, and
hence commuting $\T$-group homomorphisms
$$\xymatrix{ &{\widetilde{P}_Y} \ar[rr]^{r} &{} &{F(G')} \\ &{} &G
\ar[lu]^{\eps'} \ar[ru]_{\eps} &{}}$$

Clearly, it is enough to check that $r: \widetilde{P}_Y \to F(G')$ is an
isomorphism.

We begin by showing that $r$ is an inclusion. Indeed, $Y$ is a faithful
representation of both $F(G')$ and $\widetilde{P}_Y$ (in the sense of Lemma
\ref{lem:faithful_rep}), and hence both $\T$-groups embed into $GL(Y)$, their
embeddings compatible with maps $r, \eps, \eps'$. Hence $r$ is injective.

By assumption, the $\T'$-group $G'$ is an
extension of $P(X)$, $\mu_2$; hence $F(G')$ is an extension of $P(Y) \cong
F(P(X))$, $\mu_2 \cong F(\mu_2)$.
By Lemma \ref{lem:fund_group_extension}, the $\T$-group $\widetilde{P}_Y$
comes with an inclusion $i:P(Y) \to \widetilde{P}_Y$ and a quotient $q:
\widetilde{P}_Y \to \mu_2$ such that $q \circ i =1$. Hence the $\T$-group
inclusion $r:\widetilde{P}_Y
\to F(G')$ is an isomorphism, and the lemma is proved.

\end{proof}

In particular, we obtain a stronger version of the result in Lemma
\ref{lem:fund_group_extension}:

\begin{corollary}\label{cor:fund_group_extension}
Let $\T$ be a  $\sVect$-tensor category,
and let $X \in \T$ be an object together with a symmetric form
$\omega_X: X \otimes X
\to \Pi \triv$. Let $\widetilde{P}_X$ be the fundamental group of $\Rep(P(X))$.
 Then $\widetilde{P}_X$ is an extension of $P(X)$ (with conjugation action) and
$\mu_2$ (with trivial $P(X)$ action).
\end{corollary}

\subsubsection{Fundamental groups of categories \texorpdfstring{$\Rep(\p(n))$,
$\urep$}{Rep(p(n)), Rep P}}

Let $\T = \sVect$ and $X = V_n$ with form $\omega_n$. We denote $P_n:=P(V_n)$.

It is well known that the derivation functor $\Rep(P_n) \to \Rep(\p(n))$ is an
equivalence of $\sVect$-tensor categories (see \cite{CF,
Ma}).

This statement has two immediate consequences:

\begin{corollary}
The fundamental group of $\T=\Rep(\p(n))$ is a semidirect product of $P_n$
and $\mu_2$, where $V_n$
is the defining representation of $\p(n)$ with isomorphism $\omega_n: V_n
\otimes V_n \to
\Pi \triv$.
\end{corollary}
\begin{proof}
This follows directly from
Proposition \ref{prop:semidirect} (alternatively, from Lemma
\ref{lem:fund_group_extension}) and the equivalence $\Rep(P_n) \cong
\Rep(\p(n))$.
\end{proof}

\begin{corollary}\label{cor:Rep_gen_by_V}
 The category
$\Rep(\p(n))$ is
generated
by $V_n$ under taking finite tensor products, finite direct sums,
subquotients, and \InnaB{parity} shifts. Namely, the objects in this
category are subquotients of $\bigoplus_{i \in I}
\InnaB{V_n}^{\otimes
a_i} \otimes \left(\Pi\triv\right)^{\otimes b_i}$ for some finite set $I$, $a_i
\in \Z_{\geq 0}, b_i \in \{0, 1\}$.
\end{corollary}

Finally, we consider $\T = \urep$ and the object $X = \bV$ with the
form $\omega_{\bV}: \bV \otimes \bV \to \Pi \triv$.

Let $\underline{P}$ be the fundamental group of $\T=\urep$.
\begin{proposition}\label{prop:urep_fund_grp}
We have an extension of $\urep$-groups
\begin{equation}\label{eq:grp_ext}
 1 \to P(\bV) \to \underline{P}
\to \mu_2 \to 1.
\end{equation}
\end{proposition}
\begin{proof}
 Throughout the proof, we will consider $\urep$-groups.

Let $A \in \Alg_{\urep}$. Recall that the functor
$A \mapsto \underline{P}(A)$ is given by $\underline{P}(A) =
\Aut_A^{\otimes}(i_A)$ where $i_A: \urep \to A-Mod, \; M\mapsto A \otimes M$ is
a tensor functor
between tensor categories, Theorem \ref{thrm:uni-1} implies that a
$\otimes$-automorphism $\theta:i_A \to i_A$ is uniquely defined by its
restriction to the full subcategory $I(\PP)
\subset \urep$. In turn, the restriction is completely determined by the maps
$$\theta_{\Pi \triv} : A\otimes \Pi \triv \to A\otimes \Pi \triv  \;
\text{ and } \; \theta_{\bV} : A\otimes \bV \to A\otimes \bV $$ such that the
diagrams below commute:
\begin{equation}
\xymatrix{&{A \otimes \Pi \triv \otimes A \otimes \Pi \triv}
\ar[d]_{\theta_{\Pi \triv}}
\ar[r]  &{A}
\ar[d]^{\id}\\ &{A \otimes\Pi \triv \otimes A \otimes \Pi \triv} \ar[r]
&{A} }  \;\; \xymatrix{&{A \otimes \bV} \ar[d]_{\theta}
\ar[rrr]^-{i_A(\psi_\bV)}
&{} &{} &{A \otimes\Pi \triv \otimes_A A \otimes \bV^*}
\ar[d]^{\id \otimes_A \theta^{\vee}}\\ &{A \otimes \bV}
\ar[rrr]_-{ i_A(\psi_\bV)} &{} &{}
&{A\otimes \Pi\triv \otimes_A A \otimes \bV^*} }
\end{equation}
(here $\psi_{\bV}:\bV \to \Pi \bV^*$ is the isomorphism induced by
$\omega_{\bV}$, as before).
We now construct the maps in the extension \eqref{eq:grp_ext}.

First, consider the group homomorphism $\phi:P(\bV) \to \underline{P}$
defined as follows. Let
$\eta \in P(\bV)(A)$. This is a map $\eta:A \otimes \bV \to A \otimes \bV$
preserving $\omega_{\bV}$ (and hence $\psi_{\bV}$).

We then define the homomorphism between (usual) groups
$\phi_A: P(\bV)(A) \to \underline{P}(A)$ as
$$\eta \in P(\bV)(A)
\longmapsto \widetilde{\eta} \in  \underline{P}(A)$$ with
$\widetilde{\eta}_{\Pi \triv} = \id$ and $\widetilde{\eta}_{\bV} = \eta$.
The map $\phi_A$ is clearly injective.

Next, consider the tensor functor $J:\sVect \to \urep$.

We have a homomorphism $q:\underline{P}
\to \mu_2 \cong J(\pi(\sVect))$ given by $$\theta \in \underline{P}(A)
\mapsto
\theta_{\Pi \triv} \in \mu_2(A)$$ (see also \cite[Section
8]{D1} and the proof of Lemma \ref{lem:fund_group_extension}). This
homomorphism is surjective by Lemma \ref{lem:surjective_gen}.

Thus we have and inclusion $\phi:P(\bV)
\to \underline{P}$ and a quotient $q:\underline{P}
\to \mu_2$. By the constructions above, the composition of these maps is zero.
Furthermore, the fact that $\theta \in \underline{P}(A)$ is determined by
$\theta_{\Pi \triv}$ and $\theta_{\bV}$ implies that $q$ is indeed the cokernel
map of $\phi$.

The statement of the Proposition follows.
\end{proof}

\subsection{Universal property}
\begin{theorem}\label{thrm:uni-2}
 Let $\mathcal A$ be a pre-Tannakian \InnaC{SM} $\sVect$-category, and let $X \in
\mathcal A$ be a \InnaC{dualizable} object
with
a
non-degenerate symmetric form $$\omega_X:
X
\otimes X \to
\Pi \triv_{\mathcal A}.$$ Consider the canonical SM functor
$F_X:
\PP \longrightarrow \mathcal A$ sending the generator
$\tilde{V}$ of $\PP$ to $X$ and the form $\omega_{\tilde{V}}$ on $\tilde{V}$ to
$\omega_X$.
\begin{enumerate}
 \item If $X$ is not annihilated by any Schur functor then $F_X$
factors through the embedding $I:\PP \to\urep$ and gives
rise
to a faithful tensor functor $$\mathbf{F}_X:\urep \longrightarrow \T,
\;\;\;
\bV \longmapsto X, \omega_{\bV} \to \omega_X. $$ Furthermore, the
functor $\mathbf{F}_X$ factors through an equivalence of categories
$\urep
\to \Rep(P(X))$.

 \item If $X$ is annihilated by some Schur functor then there exists a unique
$n
\in \Z_+$ such that $F_X$ factors through the SM
functor
$\PP \longrightarrow \Rep(\p(n))$ and gives rise to a faithful tensor
functor
 $$ \mathbf{F}_X:\Rep(\p(n)) \longrightarrow \T, \;\;\;V_n \longmapsto
X$$ with the form on $V_n$ sent to $\omega_X$. Furthermore, the
functor $\mathbf{F}_X$ factors through an equivalence of categories
$\Rep(\p(n))
\to \Rep(P(X))$.
\end{enumerate}
\end{theorem}
\begin{lemma}
If $X$ is not annihilated by any Schur functor then the
functor $F_X$ is faithful.
\end{lemma}
\begin{proof}
In this proof, we will write $$\mathtt{s}\Hom_{\urep}(U, W):=
\Hom_{\urep}(U, W) \oplus \Hom_{\urep}(U, \Pi W)$$ to avoid
unnecessary repetitions.

Recall that for any $m, k \geq
0$, the space $\mathtt{s}\Hom_{\urep}(\bV^{\otimes m}, \bV^{\otimes k})$ is
zero
iff $m+k \notin 2\Z$. If $m+k \in 2\Z$, let $r= (m+k)/2$. We have an
isomorphism $$\mathtt{s}\Hom_{\urep}(\bV^{\otimes m}, \bV^{\otimes k}) \cong
\mathtt{s}\Hom_{\urep}(\bV^{\otimes r}, \bV^{\otimes r})$$ defined via the
isomorphism $\bV^* \cong \Pi V$. This isomorphism is preserved by $F_X$, since
the latter is a SM functor. Hence it is enough to check that $F_X$ is injective
on the algebra $\mathtt{s}\End_{\urep}(\bV^{\otimes r})$.

Since $F_X$ is SM, it commutes with symmetry isomorphisms, meaning that
$$F_X: \mathtt{s}\End_{\urep}(\bV^{\otimes r}) \to
\mathtt{s}\End_{\T}(X^{\otimes r})$$ is an $S_r\times S_r$-equivariant
map.

The algebra $\mathtt{s}\End(\bV^{\otimes r})$ is called the signed (or odd)
Brauer algebra, and has a diagrammatic basis $\mathcal B$, each element
describing a string diagram on $2r$ endpoints, located in $2$ rows ($r$ dots in
each row). We refer the reader to \cite{BDE+2} for a detailed description of
the
basis, and to \cite{CoulEhr, KT} for more details on the algebra.

We will use the fact that such a basis can be
described using compositions and tensor products of ({\it both even and odd!})
maps $\id_{\bV}: \bV \to \bV$, $\sigma: \bV \otimes \bV  \to  \bV \otimes \bV $
(the symmetry morphism) $\omega_{\bV}:\bV \otimes \bV  \to  \Pi \triv$ and
$\omega_{\bV}^*: \Pi \triv \to \bV \otimes \bV$. In terms of diagrams, these
correspond respectively to vertical strings, crossings, caps and cups.

Now let $\sum_{D \in \mathcal B} \alpha_D D $ be a linear combination of
diagrams such that $F_X \left(\sum_{D \in \mathcal B} \alpha_D D \right) =0$.

Let $\mathcal B' \subset \mathcal B$ be the set of $D$ such that
$\alpha_D \neq 0$. We may assume that $\mathcal B'$ has the minimal possible
size, and will show that we must have $\mathcal B' =\emptyset$.

First, recall that $F_X$ does not annihilate any Schur functor, hence it is
injective on the
subalgebra $\C S_r \subset \mathtt{s}\End(\bV^{\otimes r})$ given by
permutations of the factors. This means that at least one of $D \in \mathcal
B'$ has a cup or a cap.

\begin{sublemma}
 All $D \in \mathcal B'$ have cups and caps in
the same positions.
\end{sublemma}
\begin{proof}

Indeed, let $D_0 \in \mathcal B'$, and assume it has a cup connecting positions
$i, j$ (i.e. it has a tensor factor $\omega_{\bV}^*: \Pi \triv \to \bV \otimes
\bV$ in positions $i, j$). Let us compose $\sum_{D \in \mathcal B} \alpha_D D$
with $f=\omega^*_{i, j} \circ \omega_{i, j} \in
\mathtt{s}\End(\bV^{\otimes r})$, where $\omega_{i, j}$ indicates the
pairing $\omega_{\bV}$ is performed on the $i$-th and $j$-th factors (and
similarly for $\omega_{i, j}^*$).

Since $F_X$ is monoidal, $F_X$ commutes with $\omega_{\bV}, \omega_{\bV}^*$ and
hence $\sum_{D \in \mathcal B} \alpha_D F(f \circ D)=0$.

Now, for
each $D \in \mathcal B$, the morphism $f \circ D$ is either zero an element of
the basis $\mathcal B$, up to a sign (see \cite{BDE+2}). Moreover, $f
\circ D=0$ precisely when $D$ has a cup connecting positions
$i, j$ (so $f \circ D_0 =0$). Hence $\sum_{D \in \mathcal B}
\alpha_D f \circ D$ is a linear combination having less summands than $\sum_{D
\in \mathcal B} \alpha_D D $, and hence must have all coefficients equal to
zero. This implies that all the diagrams in $\mathcal B'$ have a cup connecting
positions
$i, j$.

In the same way,
we can show that diagrams in $\mathcal B'$ have caps in the same positions.
\end{proof}
Assume $\mathcal B'$ is not empty. Using the fact that $F_X$ is $S_r \times
S_r$-equivariant, we may rewrite $\sum_{D
\in \mathcal B} \alpha_D D $ as $\phi \otimes \psi$, where $\phi \in \C S_{r'}
\subset \mathtt{s}\End(\bV^{\otimes r-r'})$ for some $r'>0$ and $\psi \in
\mathtt{s}\End(\bV^{\otimes r'})$ is of the form $$(\omega^*_{1, 2} \circ
\omega_{1, 2}) \otimes (\omega^*_{3,4} \circ \omega_{3,4})
\otimes \ldots$$

Now, by assumption on $F_X$, $F_X(\phi) \neq 0$. Moreover, $\omega_X$ is
non-degenerate, so $F_X(\psi) \neq 0$. This implies that $$F_X \left(
\sum_{D
\in \mathcal B} \alpha_D D \right) \cong F_X(\phi \otimes \psi) \neq 0$$
leading to a contraction.
\end{proof}

\begin{proof}[Proof of Theorem \ref{thrm:uni-2}.]
 First, assume $X$ is not annihilated by any Schur functor. We have seen
that the
functor $F_X$ is faithful.
By Theorem \ref{thrm:uni-1}, $F_X$
would then factor through $I$, giving a tensor functor $$F_X: \urep
\to
\T.$$ By Lemma \ref{lem:funct} and Proposition \ref{prop:urep_fund_grp}, the
functor $F_X$ factors through an equivalence of categories $\tilde{F}_X:\urep
\to \Rep(P(X))$.

Next, assume $X$ is annihilated by some Schur functor. Let $\T':= \Rep(P(X))$,
and let $\widetilde{P}_X = \pi(\T')$ be its fundamental group. By definition of
$\Rep(P(X))$, it is generated by $X$ under taking tensor powers, direct sums,
\InnaB{parity} shifts and subquotients. Hence any object of $\T'$ is annihilated by
some Schur functor. By Deligne's theorem on super-Tannakian reconstruction
(see \cite{D2}), $\T'$ is a (super-)Tannakian tensor category: namely, there
exists a tensor functor $$\Psi: \T' \to \sVect$$ This functor induces an
$\sVect$-group homomorphism $\nu:\mu_2 \to \Psi(\widetilde{P}_X)$.

Recall that the faithful tensor functor $\Psi$ preserves (categorical)
dimensions, so $\Psi(X)$ is a non-zero vector superspace of (super-)dimension
zero. Hence $\Psi(X) \cong V_n$ for some $n \in \Z_{\geq 1}$.

Now, the composition of $\Psi$ with $F_X:\PP \longrightarrow \T'$ is isomorphic
to the composition of functors $$\PP \to \Rep(\p(n)) \xrightarrow{Forget}
\sVect$$ with the functor $\PP \to \Rep(\p(n))$ given by $\tilde{V} \mapsto V_n$.

Let $G:= \Psi(\widetilde{P}_X)$.

By Theorem \ref{old_thm:Deligne_reconstr}, $\Psi$ induces an equivalence of
categories $\T' \to \Rep(G, \nu)$ (notation as in Section
\ref{sssec:repr_grp_schemes}). By Corollary \ref{cor:fund_group_extension} and
Lemma \ref{lem:P_functoriality}, $G$ is an extension of $ P(V_n)$ (as
a subgroup) and $\mu_2$ (as a quotient). Then homomorphism $\nu:\mu_2 \to G$
makes $G$ into a semidirect product of $P(V_n)$ and $\mu_2$. Hence we have an
equivalence of tensor categories $$\mathbf{F}_X:\Rep(P(V_n)) \to \T'
=\Rep(P(X)).$$

 Hence $\Rep(P(X)) \cong \Rep(\p(n))$ for some (unique) $n$, and the functor
$\Psi$ induces an equivalence of tensor categories $\mathbf{F}_X:\Rep(\p(n)) \to
\Rep(P(X))$, whose composition with the functor $\PP \to \Rep(\p(n))$ is
isomorphic to $F_X$, as required.

\end{proof}

\subsection{An alternative
construction of
\texorpdfstring{$\urep$}{Rep(P)}}\label{ssec:second_constr_Rep_p}
As a corollary of Theorem \ref{thrm:uni-2}, we present an alternative
construction of the tensor category $\urep$. Although this construction is less
explicit than the one given in Section \ref{sec:Rep_p}, it has the benefit of
being much more compact.

Let $\Rep(\underline{GL}_t)$ be the abelian Deligne category described in
Section \ref{ssec:notn:GL_t}, and denote by $X_{\bar{0}}$ the $t$-dimensional generating object of
$\Rep(\underline{GL}_t)$.

\InnaC{Let $\T := \Rep(\underline{GL}_t)$.

The fundamental group of $\Rep(\underline{GL}_t)$ is $ \underline{GL}_t:= GL(X_{\bar 0})$ (see \cite{D}).} This group acts naturally on any object $M$ in $\T$ by a $\T$-group homomorphism $\pi_M:\underline{GL}_t \to GL(M)$.

The Lie algebra $\gl_t$ of $\underline{GL}_t$ is isomorphic to $X_{\bar 0} \otimes X_{\bar 0}^*$ (see also Section \ref{ssec:notn:SM}), and acts naturally on any object $M$ of $\T$ by $$act_M:\gl_t \otimes M \to M$$ coming from $\pi_M$.

The fundamental group of $\T$ is $$\pi(\T) \cong \underline{GL}_t \times \mu_2.$$ The natural action of this group on any object $M$ in $\T$ again denoted by $\pi_M$.

Let $$X_{\bar{1}}:= \Pi
X_{\bar{0}}^*, \; X:= X_{\bar{0}} \oplus X_{\bar{1}} \, \in \, \s
\Rep(\underline{GL}_t)$$
and consider the
non-degenerate symmetric pairing
$\omega: X \otimes X \to \Pi \triv$ given by $$\omega_{\bar{i}, \bar{j}}:
X_{\bar{i}}
\otimes X_{\bar{j}} \to \Pi\triv\;\;\; \left(\bar{i},  \bar{j} \in \{\bar{0},
\bar{1}\}\right), \;\;\;
\omega_{\bar{i}, \bar{j}}= \begin{cases}
 \Pi ev  &\text{  if  } \bar{i} \neq \bar{j}\\
 0 &\text{  else }
\end{cases}.
$$

Let $P(X)$ be the corresponding $\T$-group; we have a $\T$-group inclusion $$\eta: GL(X_{\bar{0}}) \cong \underline{GL}_t \longrightarrow P(X).$$
\begin{remark}
 The subgroup $\underline{GL}_t$ can be seen as the ``even part'' of the $\Rep(\underline{GL}_t)$-supergroup $P(X)$.
\end{remark}
Let $\widehat{Rep}(P(X))$ be the category of $P(X)$-representations $(M, \rho)$ in $\T$ such that $\rho \circ \eta \cong \pi_M$. This is a tensor category.

\begin{lemma}
 The category $\widehat{Rep}(P(X))$ is isomorphic to the category $\Rep(P(X))$ of representations of $P(X)$, as defined in Notation \ref{notn:Rep_P_X}.
\end{lemma}
\begin{proof}
Consider the $\T$-group homomorphism $\eps: \underline{GL}_t \times \mu_2 \to P(X) \rtimes \mu_2$ given by $\underline{GL}_t \longrightarrow P(X)$.

We can rewrite the definition of $\widehat{Rep}(P(X))$ as representations $\rho:P(X) \rtimes \mu_2 \longrightarrow GL(M)$, $M \in \T$ whose composition with $\eps$ gives the natural action $\pi_M$.

Deligne's theorem \ref{old_thm:Deligne_reconstr} then implies that $P(X) \rtimes \mu_2$ is the image of the fundamental group of $\widehat{Rep}(P(X))$ under the forgetful functor $F: \widehat{Rep}(P(X)) \to \T$. Now Lemma \ref{lem:fund_group_extension} implies the required statement.
\end{proof}

\begin{proposition}\label{prop:second_constr_Rep_p}
 There is an equivalence of tensor categories
 $$\urep \longrightarrow \Rep(P(X)), \;\; \bV \mapsto X, \; \omega_{\bV} \mapsto
\omega_X.$$
\end{proposition}

\begin{proof}
 We only need to check that $X$ is not annihilated by any Schur functor; if
this holds, then by Theorem \ref{thrm:uni-2}, we are done.

Assume $X$ is annihilated by some Schur functor. Then $X_{\bar{0}}\in
\Rep(\underline{GL}_t)$ is annihilated by the same functor, which contradicts
the
construction of $\Rep(\underline{GL}_t)$ as a ``free'' rigid SM category
generated by $X_{\bar{0}}$ (cf. Section \ref{ssec:notn:GL_t} and \cite{EHS}).
\end{proof}

\begin{remark}\label{rmk:integrable_repr}
\InnaC{Let $\underline{\End}(X)\cong X \otimes X^*$ be the internal endomorphism
algebra of $X$, seen as a Lie algebra object in $\s
\Rep(\underline{GL}_t)$. Let $\p(X) \subset \underline{\End}(X)$ be
Lie subalgebra preserving $\omega_{X}$. Then $\p(X) = Lie P(X)$, and $\gl\left(X_{\bar 0} \right) \subset \p(X)$} \InnaD{is a direct summand}.

\InnaC{We have a ``differentiation'' functor $$\Rep(P(X)) \to \Rep(\p(X))$$ which can be shown to be SM and fully faithful, but is not essentially surjective (that is, surjective on objects).} \InnaD{For example, given $a\in \C \setminus \Z$, we can define a homomorphism $\gl\left(X_{\bar 0} \right) \to \triv$ which does not integrate to an action of $\underline{GL}_t$ on $\triv$; hence we have a homomorphism $\p(X) \to \triv$ which does not integrate to an action of $P(X)$ on $\triv$.} 

\InnaC{By analogy with the representation theory of (usual) supergroups, we may define $\Rep(\p(X))_{int}$ to be the full subcategory of ``integrable'' $\p(X)$ representations: that is, representation of $\p(X)$ whose restriction to $\gl\left(\left(X\right)_{\bar 0} \right)$ can be lifted to $ \Rep(\underline{GL}_t)$. In other words, in the diagram $$ \xymatrix{&{\Rep(\p(X))_{int}} \ar[rr] &{} &{\Rep(\p(X))} \ar[d]^{Forget}\\ &{\Rep(\underline{GL}_t)} \ar[rr]^{diff}  &{} &{ \Rep(\gl_t)}
}$$
the objects $M$ of $\Rep(\p(X))$ which lie in $\Rep(\p(X))_{int}$ are precisely those for which $Forget(M)$ lies in the essential image of the functor $diff$ in the lower row.

Clearly, $\Rep(\p(X))_{int}$ contains the essential image of the functor $\Rep(P(X)) \to \Rep(\p(X))$. One could ask whether this induces an equivalence $\Rep(P(X)) \to \Rep(\p(X))_{int}$ as in the classical Lie superalgebra theory.

We expect the answer to be negative, due to the fact that the ``odd'' part of $\p(X)$ is not nilpotent anymore, but do not have a counterexample at the moment.

More generally, given a tensor category $\T$, one can consider the group schemes and their Lie superalgebras in $\s\T = \sVect \otimes \T$, and ask a similar question in this situation: consider a group scheme $G$ in $\s\T$ with Lie algebra object $\g = \g_{\bar{0}} \oplus \Pi \g_{\bar{1}}$, $\g_{\bar{0}},\g_{\bar{1}} \in \T$, with $\g_{\bar{0}} = Lie(G_{\bar{0}})$, $G_{\bar{0}} \in AffSch(\T)$. Given a $\g$-representation in $\s\T$ with $\g_{\bar{0}}$ action which integrates to $G_{\bar{0}}$, does the $\g$-action integrate to a $G$-action?

Here is an example when the answer is ``no'', as was provided to us by P. Etingof. Let $\T = \sVect$. } \VeraC{Consider the abelian Lie superalgebra $\g$ in $\s\T$ with $\g_{\bar{0}}=0$ and $\g_{\bar{1}}=\mathbb C^{(0|1)}$ (the abelian Lie superalgebra in $\T$). Construct an affine group scheme
  $G$ by setting $\mathcal{O}(G):=S(\g^*)$ in $\s\T$ with the
  comultiplication map $\Delta(x):=x\otimes 1+1\otimes x$ where $x$ is
  a non-zero
  vector in $\g^*$. Note that $\mathcal{O}$ is isomorphic to the
  polynomial algebra $\mathbb C[x]$ and its Lie algebra is isomorphic
  to $\g$. The group $G$ does not have any non-trivial
  one-dimensional representation while its Lie algebra $\g$ has a
  representation $\rho_\lambda(x^*)= \lambda$ for any
  $\lambda\in\mathbb C$.}
\end{remark}

\section{Lower highest weight structure}\label{sec:properties}

  \subsection{Weights in the category
\texorpdfstring{$\urep$}{Rep(P)}}\label{ssec:infinite_weights}

The set of infinite
non-decreasing integer sequences
$\lambda=(\lambda_1, \lambda_2, \ldots)$ with $\lambda_i=0$ for $i>>0$ will be
called the set of
{\it weights} for $\urep$.

\comment{Fix such a sequence $\lambda$. We denote by $\bar{\lambda}$ the
infinite
sequence $\lambda + (0,1,2,\ldots)$, and set
$c_{\lambda}:=\{\bar{\lambda_i}\,|\,i=1, 2, \ldots \}$.

We then define the
corresponding function
$$f_{\lambda}:\Z \to \Z, f_\lambda(i)=\begin{cases}1,
\text{if $i\in c_{\lambda}$}\\0, \text{if $i\notin c_{\lambda}$}\end{cases}.$$ }

We associate to $\lambda$ a weight diagram $d_{\lambda}$, defined as
a labeling of the integer line by symbols $\bullet$ (``black ball'') and
$\circ$
(``empty'') such that such that $i$ has label $\bullet$ if $i$ belongs to the
sequence $\lambda + (0,1,2,\ldots)$, and label
$\circ$ otherwise.

\comment{Let us describe precisely which labeling of the integer line are
obtained
from such sequences: these are the labellings where the number of $\circ$
symbols in positions $i \geq 0$ is finite, and equals the number of $\bullet$
in positions $i<0$.

For any $k \in \Z$, denote $$[k]_2 := 2\lfloor \frac{k}{2} \in 2\Z$$
Given a weight $\lambda$, we define $$\varrho(\lambda) =
[\sum_{i \geq 0} (-1)^{\bar{\lambda_i}}]_2$$ with
the sum $\sum_{k\geq k_0} (-1)^k$ is set to be zero for any $k_0 \in \Z$.}

\begin{definition}\label{def:truncating_weights}
 Let $\lambda$ be a weight for $\urep$. We will define its $n$-th
truncation $\lambda^{(n)}$ as the weight of $\p(n)$ given by the first $n$
entries of $\lambda$.
\end{definition}
Then $d_{\lambda^{(n)}}$ is obtained from $d_{\lambda}$ by taking only the
first $n$ dots (starting from the leftmost one).

With the notation above, we have: $$F_n(\L(\lambda))
\cong
L_n(\lambda^{(n)}), \;\; F_n(\bdel(\lambda))
\cong
\overline{\Delta}^k_n(\lambda^{(n)}), \;\; F_n(\bnab(\lambda))
\cong
\overline{\nabla}^k_n(\lambda^{(n)}) $$ for $n>>k$.

\subsection{Lower highest weight structure}\label{ssec:hw}
We will now show that $\urep$ is a lower highest weight category (i.e.
``locally'' highest weight).

\begin{definition}\label{def:lower_hw_cat}
A {\it lower highest weight category} is an artinian abelian
$\C$-linear
category $\mathcal A$ together with a poset $ (\Lambda, \leq )$ (poset of
weights)
and a filtration $\Lambda = \bigcup_{k \in \Z_+} \Lambda^k$,
 such that the following conditions hold:
 \begin{enumerate}
 \item The set $\Lambda$ is in bijection with the set of isomorphism classes of
simple objects in $\mathcal A$.

 \item For each $\xi \in \Lambda$, the Serre subcategory
$\mathcal A( \leq \xi)$ generated by simples $\{L(\lambda), \lambda \leq \xi\}$
contains a projective cover $\Delta(\xi)$ of $L(\xi)$, and an injective
hull $\nabla(\xi)$ of $\xi$. The objects $\Delta(\xi)$, $\nabla(\xi)$ are
called {\it standard} and {\it costandard} objects in $\mathcal A$.

\item There exists precisely one isomorphism class of indecomposable objects
$T(\xi)$ in $\mathcal A$ which has $\Delta(\xi)$ as a submodule,
$T(\xi)/\Delta(\xi)$
has a filtration with standard subquotients, and $T(\xi)$ also has a filtration
with costandard subquotients.

Such objects $T(\xi)$ are the indecomposable {\it tilting} objects in $\mathcal
A$.

\item Let $k \geq 0$, and let $\mathcal A^k$ be the full subcategory of
$\mathcal A$ whose
objects are subquotients of finite direct sums of objects $T(\xi)$, $\xi \in
\Lambda^k$. Then each $\mathcal A^k$ is a highest weight category with poset
$(\Lambda_k, \leq)$, simple objects $\{L(\xi), \xi \in
\Lambda^k\}$, standard objects $\{\Delta(\xi), \xi \in
\Lambda^k\}$, costandard objects $\{\nabla(\xi), \xi \in
\Lambda^k\}$, and tilting objects $\{T(\xi), \xi \in
\Lambda^k\}$. The category $\mathcal A^k$ also has enough projective and
injective
objects.

\item The subcategories $\mathcal A^k$ form a filtration on the category
$\mathcal A$: $\mathcal A =
\bigcup_{k \geq 0} \mathcal A^k$.

 \end{enumerate}

\end{definition}
\begin{remark}
 The main difference between a highest weight category and a lower
highest weight category is the possible lack of projectives and injectives in
$\mathcal A$.
\end{remark}

Consider the set of
weights $\Lambda \times {\pm}$ ($\Lambda$ as in \ref{sssec:notn:weight_p_n}), where $\pm$ stands for the \InnaB{parity} shift of the simple module. \InnaC{We consider the filtration and order on $\Lambda \times {\pm}$ induced by the filtration $\bigcup_{k \geq 0} \Lambda^k$ and the partial order $\leq$ on $\Lambda$, with $$\Lambda^k = \{ \lam \in \Lambda| \abs{\lambda} \leq k \}.$$
The partial order and the filtration on $\Lambda \times {\pm}$ disregard the possible \InnaB{parity} shift.}

\begin{proposition}
 The category $\urep$ has a lower highest weight structure given
by the set of
weights $\Lambda \times {\pm}$ with order $\leq$. The objects
$\bdel(\lam)$ play the role of standard objects, $\bnab(\lam)$ play the role of
costandard objects, and $I(\PP)$ is the full subcategory of tilting objects.
Furthermore, the tensor duality $(-)^*$ makes $\urepk{k}$ into
highest weight categories with duality, in the sense of \cite{CPS}.
\end{proposition}

\begin{remark}
 \InnaC{There is an important difference between $\urep$ and $\Rep(\p(n))$ as lower highest weight categories: the category $\Rep(\p(n))$ has two obvious highest weight structures (one with standard objects of the form $\Delta_n(\lambda)$ and the other with standard objects of the form $\nabla_n(\lambda)$), but has no duality functor which maps standard objects to costandard objects}.\InnaC{ Meanwhile, the category $\urep$ has only one obvious lower highest weight structure, as described above, but has a duality functor. The objects $\bnab(\lam)$ cannot play the role of standard objects: for instance, both $\bnab(0)$ and $\bnab(-2\eps_1)$ have cosocle $\triv$. We thank K. Coulembier for pointing this out to us.}
\end{remark}

\begin{proof}
 The only non-obvious statements we need to check are
 \begin{enumerate}
  \item\label{itm:proj_filtr} For any $k \geq 0$ and any $k$-admissible $\lam$,
$\pmb{P}_k(\lam)$ has a filtration with standard subquotients (``standard
filtration'') $\bdel(\mu)$ such
that $\mu \geq \lam$ (in this case clearly the top quotient of the filtration
would be $\bdel(\lam)$).
  \item\label{itm:tilt_filtr} For any $\lam$, $\bV^{\otimes s}$ has a
filtration
 with standard subquotients, and a filtration
 with costandard subquotients (``costandard filtration''). Furthermore, each
standard object appears as a subobject in such a standard filtration for some
$s \geq 0$.
 \end{enumerate}

 The proof of \eqref{itm:proj_filtr} will be given in Lemma
\ref{lem:proj_filtr} below.

To prove \eqref{itm:tilt_filtr}, recall from Corollary
\ref{cor:Phi_simples_to_delta_Deligne} that $\Phi:  \Rep(\p(\infty)) \to \urep$
takes the socle filtration of $V_{\infty}^{\otimes s}$ to a standard filtration
of $\bV^{\otimes s}$. Furthermore, this implies for any weight $\lam$ of
$\urep$ such that $\abs{\lam}=s$, $\bdel(\lam)$ appears as a subobject in the
induced standard filtration of $S^{\pmb{\lam}^{\vee}}\bV$. The costandard
filtration on $\bV^{\otimes s}$ is then obtained by applying the functor $(-)^*$
to the standard filtration.
\end{proof}
\begin{definition}
Let $\bT(\lambda)$ be an indecomposable tilting object such that $\bdel(\lam)
\subset \bT(\lam) $ \InnaC{and the cokernel has a filtration with subquotients $\Delta(\mu)$, $\mu \leq \lam$}.
\end{definition}
\begin{remark}
 It is easy to see that $\bdel(\lam) \subset \bT(\lam) \subset
S^{\pmb{\lam}^{\vee}}\bV$, but the latter is not necessarily an isomorphism.
For example, $\bV^{\otimes 3}$ has a direct summand isomorphic (up to change of
\InnaB{parity}) to $\bV$, and hence has more than three indecomposable direct summands.
\end{remark}
\begin{lemma}\label{lem:proj_filtr}
 For any $k \geq 0$ and any $k$-admissible $\lam$,
$\pmb{P}_k(\lam)$ has a filtration with standard subquotients $\bdel(\mu)$ such
that $\mu \geq \lam$.
\end{lemma}
\begin{proof}
We begin by showing that $\pmb{P}_k(\lam)$ has some standard filtration. Let $X
\in I(\PP)$ be such that $\pmb{P}_k(\lam)$ is the maximal quotient of $X$ lying
in $\urepk{k}$. The existence of such $X$ is shown in the proof of Lemma
\ref{lem:trivial_projective_cover}. We may assume that $X$ is indecomposable,
and hence a direct summand of $\bV^{\otimes r}$ for some $r\geq 0$.

We may assume that $r >k$ (otherwise $\pmb{P}_k(\lam)$ is a direct summand of
$\bV^{\otimes r}$).

Consider objects
$$K := \bigcup_{g: \bV^{\otimes r} \to \bV^{\otimes s}, s \leq k} \Ker(g)$$ in
$\urep$ and $$ K' := \bigcup_{f: V_{\infty}^{\otimes r} \to V_{\infty}^{\otimes
s}, s \leq k} \Ker(f)$$ in $\Rep(\p(\infty))$.

Since $\Phi$ is exact, we have:
$$ \Phi(K')\cong\bigcup_{f: V_{\infty}^{\otimes r} \to V_{\infty}^{\otimes
s}, s \leq k} \Ker(\Phi(f))$$

\begin{sublemma}\label{sublem:image_K}
 We have: $\Phi(K') \cong K$.
\end{sublemma}
\begin{proof}
Since $\Phi$ is monoidal, $\Phi(f) \in \Hom_{\urep}(\bV^{\otimes r},
\bV^{\otimes s})$ for any $f: V_{\infty}^{\otimes r} \to V_{\infty}^{\otimes
s}$. Thus $K \subset \Phi(K')$.

On the other hand, we can choose a diagrammatic basis $\{g_i\}$ for
$\Hom_{\urep}(\bV^{\otimes r},
\bV^{\otimes s})$ as in \cite{BDE+2, KT} consisting (up to change of
\InnaB{parity}) of string diagrams with $s$ dots in top row and $r$ dots in bottom
row.
The strings can be caps (connecting two dots in the bottom row), cups
(connecting two dots in top row), or strings connecting 2 dots from different
rows. Similarly, $\Hom_{\p(\infty)}(V_{\infty}^{\otimes r} \to
V_{\infty}^{\otimes s})$ has a diagrammatic basis $\{f_i\}$ with the same
diagrams, but now no cups
are allowed.

On such diagrammatic morphisms, composition is defined by stacking
diagrams (bottom to top) and then transforming the obtained diagram into an
element of the basis by some rules described e.g. in \cite{BDE+2, KT}.

Consider a diagrammatic morphism $g_i \in \Hom_{\urep}(\bV^{\otimes r},
\bV^{\otimes s})$. It has at most $s$ strings going from top to bottom, and so
at least $\lfloor \frac{r-s}{2} \rfloor$ caps in bottom row. Thus can be
written as $$g_i = h \circ \Phi(f_j)$$ for some diagrammatic morphism $f_j \in
\Hom_{\p(\infty)}(V_{\infty}^{\otimes r} \to
V_{\infty}^{\otimes
s'})$ where $s' \leq k$, and some diagrammatic morphism $h \in
\Hom_{\urep}(\bV^{\otimes s'},
\bV^{\otimes s})$.

This implies that $\Ker(g_i) \subset \Ker(\Phi(f_j))$ and hence $\Phi(K')
\subset K$. The sublemma is proved.
\end{proof}

Let $N :=\Ker( X \to \pmb{P}_k(\lam))$.

\begin{sublemma}
 $N = K \cap X$.
\end{sublemma}
\begin{proof}
For any $g: \bV^{\otimes r} \to \bV^{\otimes s}$, $s \leq k$,
$N \subset \Ker\left(g\rvert_X\right)$, and hence $N \subset K \cap X$.

Vice versa, $\pmb{P}_k(\lam)$ is a subobject of some $D \in I(\PP)$, and hence
$\pmb{P}_k(\lam) \subset \bV^{\otimes s}$ for some $s\leq k$. This implies that
$\pmb{P}_k(\lam) = \Im(g)$ for some $g: \bV^{\otimes r} \to \bV^{\otimes s}$
and thus $K \cap X \subset N$. This proves the statement of the sublemma.
\end{proof}

We now consider the socle filtration
$$soc^0=0 \subset soc^1 \subset \ldots \subset soc^{r'} =V_{\infty}^{\otimes
r}$$ where $r' = \lfloor \frac{r}{2} \rfloor+1$. By \cite[Lemma 17]{SerInf} the
$i$-th
term is given by
$$soc^{i} = \bigcup_{f: V_{\infty}^{\otimes r} \to V_{\infty}^{\otimes
s}, s \leq (r-2i)} \Ker(f)$$ (here negative tensor powers are treated as zero)
and hence $soc^{\lceil \frac{r-k}{2} \rceil }=K'$.

Taking the image of the socle filtration above under $\Phi$, we obtain a
filtration $$F^0(\bV^{\otimes r}) =0 \subset
F^1(\bV^{\otimes r}) \subset \ldots
\subset F^l(\bV^{\otimes r}) = \bV^{\otimes r}.$$ The subquotients in the
latter filtration are direct sums of standard
objects in
$\urep$ (by Corollary \ref{cor:Phi_simples_to_delta_Deligne}). As in Sublemma
\ref{sublem:image_K}, we have:
$$F^{i} = \bigcup_{g: \bV^{\otimes r} \to \bV^{\otimes s}, s \leq r-i}
\Ker(g)$$ and hence
$F^{\lceil \frac{r-k}{2} \rceil }(\bV^{\otimes r}) = K$.

We claim that $F^{i}$ induces a {\it unique} a filtration $F^0(X) = 0\subset
F^1(X) \subset \ldots \subset F^l(X) = X$ on $X$.
Indeed, consider the projector $e \in \End_{\urep}( \bV^{\otimes r})$ onto $X$.
For any $g: \bV^{\otimes r} \to \bV^{\otimes s}$, $g \circ e$ is also a map
$\bV^{\otimes r} \to \bV^{\otimes s}$ and hence for any $i$, $$F^{i} \subset
\bigcup_{g: \bV^{\otimes r} \to \bV^{\otimes s}, s \leq r-i} \Ker(g \circ e).$$
This implies that $e(F^i) \subset F^i$, and hence $e(F^i) = X \cap F^i$, which
establishes the uniqueness of the induced filtration on $X$, as required.

Hence the subquotients $F^i(X)/F^{i-1}(X)$ are be direct summands of
$F^i(\bV^{\otimes r})/F^{i-1}(\bV^{\otimes r})$ and are thus direct sums of
standard
objects. In particular, this implies that $\pmb{P}_k(0) \cong
F^l(X)/F^{r-k}(X)$ has a filtration with standard subquotients.

It remains to check that the standard subquotients $\bdel(\mu)$ of
$\pmb{P}_k(\lam)$ satisfy $\mu \geq \lam$. Indeed, consider the
obtained filtration $F^0 =0 \subset F^1 \subset \ldots \subset F^s =
\pmb{P}_k(\lam)$ with standard subquotients $\quotient{F^i}{F^{i-1}} =
\bdel(\mu^i)$. Then for any $\mu$, we have an exact sequence
$$ 0 \to \Hom(\bdel(\mu^i), \bnab(\mu)) \to \Hom(F^i, \bnab(\mu)) \to
\Hom(F^{i-1}, \bnab(\mu)) \to \Ext^1(\bdel(\mu^i), \bnab(\mu)) =0$$ (the last
equality follows from Lemma \ref{lem:ext}).

This implies that $$\dim \Hom(F^i,\bnab(\mu)) = \dim \Hom(F^{i-1},\bnab(\mu)) +
\delta_{\mu^i, \mu}$$ for any $\mu$ and any $i >0$. Hence the number of
times a given $\mu$ appears among $\{ \mu^i\}$ is precisely $\dim
\Hom(\pmb{P}_k(\lam), \bnab(\mu)) = (\bnab(\mu): \mathbf{L}(\lam))$ and the
latter is
non-zero only if $\mu \geq \lam$.
\end{proof}

\begin{corollary}[BGG reciprocity]
 \InnaC{For any $\lambda, \mu \in \Lambda^k$, we have:
 $$(\pmb{P}_k(\lambda):\bdel(\mu))=[\bnab(\mu):\L(\lambda)].$$}
\end{corollary}

\section{Appendix A: Direct summands of tensor powers}\label{appendix_A}
\subsection{Temperley-Lieb action and statement of the result}
The goal of this appendix is to give a combinatorial proof of the following
result of \cite{DLZ}.

\begin{corollary}\label{appcor:contractions}
 The space $\mathtt{s}\Hom_{\mathfrak{p}(n)}(V_n^{\otimes k}, \triv)$ is
non-zero only
when $k$ is even. It is spanned by morphisms given by partitioning the $k$
factors into (disjoint) pairs, and considering a tensor product of $k/2$
contraction maps $V_n^{\otimes 2} \to \Pi \triv$ on these pairs.
\end{corollary}

The proof is based on the results on categorical action of the Temperley-Lieb
algebra \InnaB{$TL_{\infty}(\sqrt{-1})$} on $\Rep(\p(n))$ obtained in \cite{BDE+}, and standard Temperley-Lieb
combinatorics; these combinatorics are fully described here for completeness of
presentation.
\comment{\begin{theorem}\label{old_thm:TL_action_categor}[\cite[Theorem 4.5.1]{BDE+}]
\InnaB{The translation functors $\Theta_k: \Rep(\p(n)) \to \Rep(\p(n))$, $k\in \Z$, satisfy categorical Temperley-Lieb relations. Namely,} there exists natural isomorphisms of functors (for $k,j\in\mathbb{Z}$)
  \begin{eqnarray}\label{eq:TL_categ_relations}
   \Theta_k^2 &\cong& 0,\label{firstTL}\\
   \psi_{k, k \pm 1}:\quad \Theta_k \Theta_{k \pm 1} \Theta_k
&\stackrel{\cong}\Longrightarrow&   \Theta_k,\label{secondTL}\\
   \psi_{k,j}: \quad\quad\quad\Theta_k \Theta_j
&\stackrel{\cong}\Longrightarrow&  \Theta_j \Theta_k \; \text{   if   } \;
|k-j| >1,\label{thirdTL}
  \end{eqnarray}

 \end{theorem}}

\subsection{Temperley-Lieb combinatorics}

 \InnaB{Consider the Temperley-Lieb algebra $TL_{\infty}(\sqrt{-1})$ over $\mathbb{Z}_{\geq 0}$, generated by
elements $\theta_k$, $k \in \mathbb Z$, and the relations
  \begin{eqnarray}\label{eq:TL}
   \theta_k^2 &=& 0\\
 \quad \theta_k \theta_{k \pm 1} \theta_k
&=&   \theta_k,\\\
    \quad\quad\quad\theta_k \theta_j
&=&  \theta_j \theta_k \; \text{   if   } \;
|k-j| >1.
  \end{eqnarray}

  By \cite[Theorem 4.5.1]{BDE+}, the translation functors $\Theta_k: \Rep(\p(n)) \to \Rep(\p(n))$, $k\in \Z$, satisfy categorical Temperley-Lieb relations, obtained from \eqref{eq:TL} by replacing $\theta_k$ by $\Theta_k$ and equalities by explicit isomorphisms.}

\begin{notation}
 Let $I = (i_1, \ldots, i_k)$ be a finite sequence of integers. Consider
translation functors $\Theta_{i_j}$, $j = 1, \ldots, k$ in $\Rep(\p(n))$. Denote
$$\Theta_I := \Theta_{i_k} \ldots \Theta_{i_1}$$ and by $\theta_I =
\theta_{i_k} \ldots \theta_{i_1}$ the corresponding element in the
Temperley-Lieb algebra. We will use the convention $\theta_{\emptyset} = 0$.
\end{notation}

 We say that two sequences $I, I'$ are equivalent if $\theta_I = \theta_I'$,
and that $I$ is {\it reduced} if $I$ has minimal length among the sequences in
its equivalence class.

\begin{lemma}\label{lem:max_seq}
 Let $I $ be a reduced sequence. Let $m = \max\{i \in I\}$.  Then the value $m$
appears in the sequence $I$ exactly once. A similar statement holds for the
minimal value in $I$.
\end{lemma}
\begin{proof}
We prove this statement by induction on $k$. The base case $k =1$ is clear. Let
$k>1$ and assume the statement holds for any sequence of length less than $k$.
Let $I= (i_1, \ldots, i_k)$ be a sequence as in the lemma, and consider
$\theta_I$.

Assume $m = \max I$ appears more than once in $I$. Consider two consecutive
appearances: $i_s = i_r =m$, $s<r$, and $i_j < m$ for $s < j <r$. Consider the
subsequence $I' = (i_{s+1}, i_{s+2} ,\ldots, i_{r-1})$. Then $\max I' \leq
k-1$, and we have, by the induction assumption, at most one appearance of the
value $m-1$ in the sequence $I'$.

If $m-1$ does not appear in $I'$, then $\theta_I \cong 0$: one can ``move''
$\theta_{i_s}$ to be next to $\theta_{i_r}$ due to \eqref{third}, and the
result will be zero by \label{first}.

If $m-1$ appears in $I'$ exactly once, at position $i_j$, then one can ``move''
$\theta_{i_s}$, $\theta_{i_r}$ to be next to $\theta_{i_j}$ (on different sides
of $\theta_{i_j}$!) due to \label{third}. Applying \label{second} to
$\theta_{i_s}\theta_{i_j}\theta_{i_r} = \theta_m \theta_{m-1} \theta_m$ we may
replace this expression by $\theta_m$, giving a shorter expression for the
element $\theta_I$.

\end{proof}
\begin{proposition}\label{prop:presentation_TL_element}
 Let $I = (i_k, \ldots, i_1)$ be a sequence of integers.
Assume there exists a surjection $\Theta_I \C \to \C$ in $\Rep(\p(n))$. Then
there exists an integer $s \geq 0$ and a natural isomorphism
$$\Theta_{I'} \stackrel{\sim}{\longrightarrow} \Theta_I$$
where $I'$ is an integer sequence\footnote{We place delimiters $|$ in the
sequence only to stress the ``building blocks'' of the sequence; these do not
carry any additional meaning.} of length $s(s+1)$ given by concatenation of $s$
decreasing sequences of consecutive integers, where the first terms of the
sequences grow from $1$ to $s-1$:
$$(1, 0, \ldots , -s+1| 2, 1, 0, \ldots, -s+2| 3, 2, 1, 0, \ldots, -s+3|\ldots|
s, s-1, \ldots, 0)$$

\end{proposition}

\begin{proof}
We may assume that $I$ is a reduced sequence.
The condition $\Theta_I \C \twoheadrightarrow \C$ implies that there is a
non-zero map $\Theta_{i_{k-1}} \ldots \Theta_{i_1} \C \to \Theta_{i_k -1} \C$.
Hence $i_k =1, i_1=0$, since $\Theta_0$ is the only translation functor not
annihilating the module $\C$. Moreover, the same holds for any $I' \sim I$.

We will use the following sublemma:
\begin{sublemma}
 Let $I = (i_k, \ldots i_1)$ be a reduced sequence, and assume that for any
$I'= (\ldots, i'_2, i'_1) \sim I$, we have: $i_1 =  i'_1$.

 For any $j \geq 1$, consider the subsequence $I_j =(i_j, i_{j-1}, \ldots,
i_1)$, and let $m_j := \max \{i \in I_j\}$ (in particular, $ m_1 = i_1$). Then
there exists a reduced sequence $I'\sim I$ obtained from $I$ by a permutation
of its elements, such that the rightmost $m_k -m_1 + 1$ terms of the sequence
$I'$ are $m_k, m_k-1, m_k-2, \ldots, m_1+1, m_1$.
 \end{sublemma}
\begin{proof}
Recall that $m_j$ occurs exactly once in $I_j$, due to Lemma \ref{lem:max_seq}.
Let us denote its position by $l$.

We need to show that for any $j$, $m_j-1$ occurs exactly once in $I_j$ to the
right of $m_j$, i.e. occurs exactly once in $I_l$.

Indeed, $I_j$ must contain $m_j -1$ to the right of $m_j$, otherwise one can
move $m_j$ to the rightmost part of $I$, contradicting the maximality of $i_1$.
By Lemma \ref{lem:max_seq}, $m_j -1 \geq m_{l-1}$ will occur in $I_{l-1}$ at
most once, which completes the proof of the sublemma.
\end{proof}

We now return to the proof of the proposition.
Denote $s= \max I$.

For every $1 \leq r \leq s$, we will prove the following claim: $I$ can be
replaced by an equivalent reduced sequence $I' = (i'_k, \ldots, i'_1)$
(obtained by a permutation of the elements of $I$), whose last $(s+1)r$
elements
are
$$(s-r+1, s-r, \ldots, -r+1 | s-r+2, s-r+1, \ldots, -r+2| \ldots| s-1, s-2,
\ldots,-1, \ldots| s, s-1, \ldots, 0),$$
and $\max\{i_l | (s+1)r < l \leq k \} = s-r$.

We prove this claim by induction on $r$. For $r=1$, this is a direct
consequence of the above sublemma.

Now, assume the claim holds for some $r<s$, and denote the corresponding
reduced sequence by $I'$. We prove the claim for $r+1$.

Consider the subsequence $J^r = (i'_k , \ldots, i'_{(s+1)r+2} ,
i'_{(s+1)r+1})$. Denote its rightmost element $i'_{(s+1)r+1}$ by $l$.

We begin by proving that $l = -r$.

Indeed, assume $l <-r$; then $l=i'_{(s+1)r+1}$ can be moved to the right,
obtaining an equivalent reduced sequence whose rightmost element is $-r$, which
contradicts $i_1=0$.

Next, assume $l >-r$. Then it can be moved to the right to obtain an equivalent
reduced sequence with a subsequence $(l, l+1, l)$; by \eqref{eq:TL}, such a
sequence is not reduced, leading to a contradiction.

Hence $l=-r$; moreover, the above reasoning implies that for any reduced
sequence equivalent to $J^r$, its rightmost element is $-r$.

By induction assumption, the maximal element in $J^r$ is $s-r$. Applying the
sublemma to $$J^r= (i'_k , \ldots, i'_{(s+1)r+2} , i'_{(s+1)r+1} = -r)$$ we
obtain an equivalent reduced sequence $J'$ whose rightmost $s+1$ elements are
$(s-r, s-r -1, \ldots, -r)$, and hence a reduced sequence $I'$ equivalent to
$I$ whose last $(s+1)(r+1)$ elements are as required.

Finally, we claim that for $r \leq s-2$ there exists an element in $J^r$ to the
right of the maximal element $s-r$ which equals $s-r+1$. Indeed, otherwise we
would be able to move the maximal element $s-r \geq 2$ to the leftmost position
in $J^r$, contradicting $i_k =1$.

The process stops exactly when $r = s$, and we obtain the desired form for $I'$.

reduced sequence $I' = (i'_k, \ldots, i'_1)$ whose rightmost $s+1$ elements are
$s, s-1, \ldots, 0$.

\end{proof}

The following lemma shows how to ``take a square root of $\Theta_I$'' whenever
$\Theta_I \triv \twoheadrightarrow \triv$: namely, we show that we can find a
sequence $J$ such that $\Theta_I \cong {}^* \left(\Theta_J \right) \Theta_J$,
and hence the morphism $\Theta_I \triv \to \triv$ is just the application of
the counit of the adjunction ${}^* \left(\Theta_J \right) \vdash \Theta_J$ to
the object $\triv$.

\begin{lemma}\label{lem:splitting_translation_func}
Let $$I = (1, 0, \ldots , -s+1| 2, 1, 0, \ldots, -s+2| 3, 2, 1, 0, \ldots
-s+3|\ldots| s, s-1, \ldots, 0).$$
 Consider
  \begin{align*}
  J = &(s-1| s-3, s-2|s-5, s-4, s-3| \ldots| -s+5, \ldots, -1, 0 ,1,2| -s+3,
\ldots \\
  &\ldots, -1, 0 ,1| -s+1, -s+2, -s+3, \ldots, -1, 0)
  \end{align*}
  which is equivalent (due to Sublemma above) to
 $$(-s+1|-s+3, -s+2|\ldots  |s-5, \ldots 0, -1| s-3, \ldots 0, -1|s-1, s-2,
\ldots, 1, 0)$$

 Then the left adjoint functor to $\Theta_J$ is $\Theta \Theta_{J'}$ where
 $$J' = (1, 0, \ldots , -s+2| 2, 1, 0, \ldots, -s+4| 3, 2, 1, 0, \ldots,
-s+6|\ldots |s-1, s-2| s),$$

 and $\Theta_J \Theta_I \cong \Theta_J$, $\Theta_{J'} \Theta_J \cong \Theta_I$.

\end{lemma}

In particular, the last statement implies $\Theta_J \triv \neq 0$.

\begin{proof}
 To prove the isomorphism $\Theta_J \Theta_I \cong \Theta_J$, consider first
the concatenation of $J, I$, with $J$ in its original form:
 \begin{align*}
  (J|I) = &(s-1| \ldots| -s+5, \ldots, -1, 0 ,1,2| -s+3, \ldots, -1, 0 ,1|
-s+1, -s+2, -s+3, \ldots \\
  &\ldots , -1, 0 | 1, 0, \ldots , -s+1| 2, 1, 0, \ldots, -s+2| 3, 2, 1, 0,
\ldots -s+3|\ldots| s, s-1, \ldots, 0).
 \end{align*}

 The subsequence $(-s+1, -s+2, -s+3, \ldots, -1, 0 | 1, 0, \ldots , -s+1)$ is
equivalent to $(-s+1)$, so we replace the former by the latter and obtain a
sequence
 \begin{align*}
  (J|I) \sim &(s-1| \ldots| -s+5, \ldots, -1, 0 ,1,2| -s+3, \ldots, -1, 0 ,1|
-s+1| 2, 1, 0, \ldots\\
  &\ldots, -s+2| 3, 2, 1, 0, \ldots -s+3|\ldots| s, s-1, \ldots, 0).
  \end{align*}

 Now $-s+1$ can be moved to the leftmost position, and we can next consider the
subsequence $( -s+3, \ldots, -1, 0 ,1| 2, 1, 0, \ldots, -s+3, -s+2)$ which is
congruent to $(-s+3, -s+2)$. Again, these can be moved to the right, and so on.

 The sequence obtained at the end of this process, equivalent to the
concatenation $(J| I)$, is just $J$ in its second incarnation,
  $$(-s+1|-s+3, -s+2|\ldots  |s-5, \ldots 0, -1| s-3, \ldots 0, -1|s-1, s-2,
\ldots, 1, 0).$$
 Hence $\Theta_J \Theta_I \cong \Theta_J$. The statement $\Theta_{J'} \Theta_J
= \Theta_I$ is verified directly.
\end{proof}

\newpage
\section{Appendix B: Affine group schemes in tensor
categories}\label{app:affine_grp_schemes}
We recall briefly the notion of an
affine group scheme in a tensor category.
This is discussed in detail in \cite[Section 7]{D1} and \cite[Section
11.2]{EHS}.
\subsection{Definition}

Let $\T$ be a tensor category. Its ind-completion
$\mathrm{Ind}-\T$ inherits a symmetric monoidal structure.

We denote by $\Alg_{\T}$ the category of commutative algebra objects
in ${\mathrm{Ind}-\T}$. Its opposite category is called the category of
$\T$-affine schemes.
\begin{remark}
 The bifunctor $-\otimes-$ on $\T$ is biexact and faithful, so the all the
schemes discussed below are faithfully flat.
\end{remark}

\begin{definition}
 We define {\it $\T$-algebraic groups} (called $\T$-groups for short)
as group objects in the category of $\T$-affine schemes; that is, the
category $\Grps_{\T}$ is antiequivalent to the category of commutative
Hopf algebra objects in $\mathrm{Ind}-\T$.
\end{definition}

We denote by $\mathcal{O}(G)$ the
commutative Hopf algebra object corresponding to a $\T$-group $G$.

By Yoneda lemma, we may also
identify $\T$-algebraic groups with the corresponding corepresentable functors
$\Alg_{\T}\to\Grps$ from commutative algebra objects in
$\mathrm{Ind}-\T$ to groups.

Given a SM functor $F:\T\to\T'$ and $G \in \Grps_{\T}$, the image
$F(G)$ is obtained by applying the functor $F$ to
$\mathcal{O}(G) \in \mathrm{Ind}-\T$.

When the SM functor $F:\T\to\T'$ is right exact,
the $\T'$-algebraic group $F(G)$ can be also described in terms of
the functor of points (see \cite[Section 11.2]{EHS}).

\subsection{Fundamental group}\label{sssec:fund_grp}
For $A\in \Alg_{\T}$, the category of
$A$-modules $ A-\Mod$ is
defined in the standard way. The functor $i_A:\T\to A-\Mod$ carries $X\in\T$ to
$A\otimes X$, and is clearly monoidal.

The fundamental group of $\T$, denoted $\pi(\T) \in \Grps_{\T}$, is
defined by the functor of points
\begin{equation}
A\mapsto \Aut^{\otimes}(i_A:\T\to A-\Mod).
\end{equation}
where $\Aut^{\otimes}$ means that we are only considering monoidal natural
transformations of $i_A$ (that is, $\eta \in Aut(i_A)$ such that $\eta_{X
\otimes Y} \cong \eta_X \otimes \eta_Y$).

For the tensor categories considered in this paper, their fundamental groups
turn out to be affine (that is, the functor of
points is corepresentable). In general, this happens for instance when the base
field is perfect and the category is pre-Tannakian (see Section
\ref{ssec:notn:SM}), which holds for
all the categories which will be considered in this paper.

\begin{example}
 For $\T = \Vect$, the group $\pi(\Vect)$ is trivial. For $\T = \sVect$, we
have: $\pi(\sVect)\cong \mu_2$, where $\mu_2$ is
the group of square roots of unity, seen as a supergroup.
\end{example}

\subsection{Representations of group schemes}\label{sssec:repr_grp_schemes}
Let $X\in\T$. Define a functor $\Alg_{\T} \to\Grps$ by the formula

$$ A\mapsto\Aut_{ A-\Mod}(i_A(X)).$$

Again, if the base field is perfect and the category is pre-Tannakian, then
this functor in corepresentable
by a $\T$-algebraic group denoted $GL(X)$.

Given $G \in \Grps_{\T}$, a {\it representation} $X$ of $G$
is an object $X\in\T$ endowed with a structure of left comodule of the
appropriate Hopf algebra $\mathcal{O}(G)$. Alternatively, we can define a
representation of $G$ as a $\T$-group homomorphism $G \to GL(X)$. We denote the
category of representations of $G$ in $\T$ by $\Rep_{\T}(G)$.

In particular, every object $X \in \T$ is endowed with a canonical action
$\pi_X:\pi(\T)\to GL(X)$. This homomorphism is given, on the level of functors
of
 points, by the assignment of $\theta(X):i_A(X)\to i_A(X)$ to an
 automorphism $\theta$ of the functor $i_A:\T\to A-\Mod$.

 Let $F:\T' \to \T$ be a tensor
functor between two tensor categories. Then $F(\pi(\T'))$ is a $\T$-group
given by the functor of points
$$A \in \Alg_{\T} \longmapsto \Aut^{\otimes}_A(A \otimes F(-))
.$$
Such a tensor functor $F$ induces a morphism of $\T$-groups $\eps: \pi(\T)\to
F(\pi(\T'))$, which is given on the functors of points by
$$\Aut^{\otimes}_A(A \otimes (-))
\stackrel{\eps_A}{\longrightarrow} \Aut^{\otimes}_A(A \otimes F(-)) .$$

The following theorem, due to Deligne, is a generalization of the Tannakian
reconstruction theory. Let $\Rep_{\T}(F(\pi(\T')), \eps)$
the full subcategory of representations $F(\pi(\T')) \to GL(X)$ in $\T$ whose
composition with $\eps$ gives the natural action $\pi_X$ of $\pi(\T)$ on $X$.

\begin{theorem}[Deligne, \cite{D1}]\label{old_thm:Deligne_reconstr}
 The functor $F$ induces an equivalence of tensor categories $\T' \to
Rep_{\T}(F(\pi(T')), \eps)$.
\end{theorem}

 \subsection{Group homomorphisms and functors}

 The
following statement is proved in the same way as for classical group schemes
(see for
example \cite[Chapter X, Section 4]{Milne}).

\begin{lemma}\label{lem:injective_surjective}
Let $G, G'$ be $\T$-groups, and $f: G' \to G$ a $\T$-group homomorphism.
Consider the corresponding tensor functor $F: Rep_{\T}(G)\to
Rep_{\T}(G')$.
 \begin{itemize}
  \item The homomorphism $f$ is a quotient map iff the functor $F$ is full and
the essential image of $F$ is closed under taking subobjects.
\item The homomorphism $f$ is injective iff any $X \in Rep_{\T}(G')$ is a
subquotient of an object in the essential image of $F$.
 \end{itemize}

\end{lemma}

Now, let $F:\T \to \T'$ be a tensor functor between tensor categories.
Denote $G':= \pi(\T')$, $G:= F(\pi(\T))$.
As
in Theorem
\ref{old_thm:Deligne_reconstr}, we have a group homomorphism
$$\eps: G'  \to G.$$

\begin{lemma}\label{lem:surjective_gen}
The homomorphism $\eps$ is a quotient iff the
functor $F$ is full and
the essential
image of $F$ is closed under taking subobjects.
\end{lemma}
\begin{proof}
Recall that by Deligne's Theorem \ref{old_thm:Deligne_reconstr},
$Forget:Rep_{\T'}(G, \eps)
\to \T'$ satisfies the condition in the statement of the lemma iff $F$ does.

By Lemma \ref{lem:injective_surjective}, the
homomorphism $\eps$ is a quotient iff the functor $$\widetilde{F}: Rep_{\T'}(G)
\to Rep_{\T'}(G')$$ induced by $\eps$ is full, and its essential image
closed under taking subobjects.

Now, the functors
$$
\T \boxtimes \T' \xrightarrow{\,X \boxtimes Y \longmapsto F(X) \otimes Y\,}
\T', \;\;  \T' \boxtimes \T' \xrightarrow{\;Y \boxtimes Z \longmapsto Y
\otimes Z\;} \T' $$
induce equivalences $$
\T \boxtimes \T' \cong Rep_{\T'}(G), \;\; \T' \boxtimes \T' \cong
Rep_{\T'}(G')$$ (cf. \cite[Propositions 8.22, 8.23]{D1}), such that
$$ \xymatrix{ &{\T \boxtimes \T'} \ar[d] \ar[rr]^{F \, \boxtimes \,\id}
&{} &{\T' \boxtimes \T'} \ar[d] \\
&{Rep_{\T'}(G)} \ar[rr]^{\widetilde{F}} &{}
&{Rep_{\T'}(G')} }$$

Thus whenever $\widetilde{F}$ satisfies the conditions in the lemma, so does
$F$. The lemma is proved.
\end{proof}
\subsection{The fundamental group of \texorpdfstring{$\Rep_{\T}(G)$}{Rep(G)}}\label{ssec:fund_group_Rep_G}

 Let $G$ be a $\T$-group, and consider the tensor category $\Rep_{\T}(G)$ of
$G$-representations in $\T$.

This category comes with equipped with two functors:
\begin{itemize}
 \item A tensor functor
$I:\T \to Rep_{\T}(G)$, where
$M \mapsto
M_{triv}$, considered with trivial $G$-action.
\item A forgetful functor
$F:Rep_{\T}(G) \to \T$, taking $M \in Rep_{\T}(G)$ to the underlying
$\T$-object.
\end{itemize}
Clearly, we have a natural isomorphism $FI \cong \id$.

Let $\widetilde{G}$ denote the fundamental
group of $\Rep_{\T}(G)$.

By Deligne's Theorem
\ref{old_thm:Deligne_reconstr}, the functor $I$ induces a homomorphism of
$\Rep_{\T}(G)$-groups $$q:\widetilde{G} \to I(\pi(\T))$$ (here $I(\pi(\T))$ can
be thought of as the group $\pi(\T)$ with trivial $G$-action).

On the other
hand, we have a homomorphism of
$\Rep_{\T}(G)$-groups $i:G^{adj} \to \widetilde{G}$ where $G^{adj}$ is the
group $G$ with conjugation action.
\begin{proposition}\label{prop:semidirect}
 The
group
$\widetilde{G}$ is an
extension of $G^{adj}$ and
$I(\pi(\T))$. Moreover, the $\T$-group $F(\widetilde{G})$ is a
semidirect product of $G$ and $\pi(\T)$, the inclusion $\pi(\T) \to
F(\widetilde{G})$ again given as in
Theorem \ref{old_thm:Deligne_reconstr}.
\end{proposition}
\begin{proof}
 The fact that $i$ is injective and $q$ is surjective follows from Lemmata
\ref{lem:injective_surjective} and \ref{lem:surjective_gen}.
Next, we note that
$G^{adj}$ acts trivially on an object $X$ in $\Rep_{\T}(G)$ iff this object
belongs to the essential image of $I$. Hence for any object $X \in
Rep_{\T}(G)$, the action $\pi_X: \widetilde{G} \to GL(X)$ factors through $q$
iff $\pi_X \circ i = 1$. Hence $Ker(q) =i$, as required.
\end{proof}
 \subsection{Faithful representations}
We recall the notion of faithful representation.

Let $G \to GL(X)$ be a representation of $G$.

\begin{lemma}\label{lem:faithful_rep}
The following are equivalent:
\begin{enumerate}
 \item\label{itm:faith1} The $\T$-group homomorphism $G \to GL(X)$ is
injective.
 \item\label{itm:faith2} Any representation of $G$ in $\T$ is a subquotient of
$\bigoplus_{i
\in I}
X^{\otimes
a_i} \otimes (X^*)^{\otimes
b_i}\otimes U$ for some finite set $I$, $a_i, b_i \in \Z_{\geq 0}$ and
some $U
\in\T$ considered with a trivial $G$-action.
\end{enumerate}
\end{lemma}
\begin{proof}

\mbox{}

\begin{itemize}
 \item[\eqref{itm:faith1} $\Rightarrow$ \eqref{itm:faith2}]
 Let $(M, \rho)$ be a representation of $G$. Then we have
an injective map $M
\xrightarrow{\rho}
\O(G) \otimes M_{triv}$ where $M_{triv}$ has trivial $G$
action. Recall that $\O(G)$ is a quotient of $O(GL(X))$, which in turn is a
quotient of $Sym(X \otimes X^*)$; the required statement now follows.
 \item[\eqref{itm:faith2} $\Rightarrow$ \eqref{itm:faith1}]
Assume \eqref{itm:faith2} holds. Consider the functor $F:Rep_{\T}(GL(X)) \to
Rep_{\T}(G)$ induced by $\rho$. The essential image of this functor contains
representations of the form $\bigoplus_{i
\in I}
X^{\otimes
a_i} \otimes (X^*)^{\otimes
b_i}\otimes U$ as in \eqref{itm:faith2}. Hence by
Lemma \ref{lem:injective_surjective}, the homomorphism $\rho$ is
injective.

\end{itemize}
The statement
of
the lemma is now proved.
\end{proof}

\end{document}